\documentclass{article}

\usepackage[margin=3cm]{geometry}
\usepackage{amssymb}
\usepackage{amsfonts}
\usepackage{mathrsfs}
\usepackage{url}
\usepackage{latexsym}
\usepackage{newlfont}

\usepackage{leqno}

\usepackage{xspace}
\usepackage{bbm}

\usepackage{amsthm}
\usepackage{mathtools}
\usepackage{algorithm}
\usepackage[noend]{algorithmic}
\usepackage{multicol}
\usepackage{multirow}
\usepackage{enumerate}
\usepackage{hyperref}
\usepackage{csquotes}
\usepackage{enumitem}
\setlist[enumerate]{leftmargin=.5in}
\setlist[itemize]{leftmargin=.5in}
\usepackage{wrapfig}

\usepackage{cleveref}

\newtheorem{theorem}{Theorem}[section]
\newtheorem{lemma}[theorem]{Lemma}
\newtheorem{proposition}[theorem]{Proposition}
\newtheorem{corollary}[theorem]{Corollary}

\theoremstyle{remark}
\newtheorem{remark}[theorem]{Remark}
\newtheorem{example}[theorem]{Example}

\theoremstyle{definition}
\newtheorem{definition}[theorem]{Definition}

\crefname{equation}{eq.}{equations}
\crefname{definition}{def.}{definitions}
\crefname{section}{sec.}{sections}
\Crefname{section}{Sec.}{Sections}
\crefname{subsection}{sec.}{sections}
\Crefname{subsection}{{Sec.}}{sections}
\crefname{lemma}{lem.}{lemmata}
\crefname{corollary}{cor.}{corollaries}
\crefname{proposition}{prop.}{propositions}
\crefname{remark}{rem.}{remarks}
\crefname{enumi}{step}{steps}
\crefname{theorem}{thm.}{theorems}
\crefname{algorithm}{alg.}{algorithms}
\crefname{figure}{fig.}{figures}
\crefname{observation}{obs.}{observations}
\crefname{example}{ex.}{examples}

\newcommand{\email}[1]{{\href{mailto:#1}{#1}}}

\def\NN{{\mathbb N}} \def\ZZ{{\mathbb Z}}
  \def\bm{\boldsymbol}

\def\N{\mathbb{N}}
\def\Z{\mathbb{Z}}

\def\K{\mathbb{K}}
\def\Pr{\mathbb{P}}
\def\P{\mathcal{P}}
\def\A{\mathcal{A}}

\def\MHB{\ensuremath{\mathtt{MHB}}\xspace}
\def\Res{\mathrm{Res}}
\def\Im{\mathrm{Im}}
\def\Ker{\mathrm{Ker}}

\def\eqs{\mathcal{E}}
\def\inDev{\Delta}
\def\nsols{\Upsilon}

\def\wrt{with respect to\xspace}

\def\bezout{B\'ezout\xspace}
\def\mhbb{multihomogeneous \bezout bound\xspace}

\newcommand{\res}{\mathtt{res}\xspace}

\PassOptionsToPackage{usenames}{color}
\usepackage{tikz}
\usetikzlibrary{matrix,positioning,chains,fit,shapes,calc}

\usepackage{stmaryrd}

\newcommand\possiblebreak{\ifhmode\unskip\space\hfil\penalty0\hfilneg\fi}

\usepackage{linegoal}

\makeatletter
\newif\if@borderstar
\def\bordermatrix{\@ifnextchar*{%
  \@borderstartrue\@bordermatrix@i}{\@borderstarfalse\@bordermatrix@i*}%
}
\def\@bordermatrix@i*{\@ifnextchar[{%
  \@bordermatrix@ii}{\@bordermatrix@ii[()]}
}
\def\@bordermatrix@ii[#1]#2{%
  \begingroup
    \m@th\@tempdima8.75\p@\setbox\z@\vbox{%
      \def\cr{\crcr\noalign{\kern 2\p@\global\let\cr\endline }}%
      \ialign {$##$\hfil\kern 2\p@\kern\@tempdima & \thinspace %
      \hfil $##$\hfil && \quad\hfil $##$\hfil\crcr\omit\strut %
      \hfil\crcr\noalign{\kern -\baselineskip}#2\crcr\omit %
      \strut\cr}}%
    \setbox\tw@\vbox{\unvcopy\z@\global\setbox\@ne\lastbox}%
    \setbox\tw@\hbox{\unhbox\@ne\unskip\global\setbox\@ne\lastbox}%
    \setbox\tw@\hbox{%
      $\kern\wd\@ne\kern -\@tempdima\left\@firstoftwo#1%
        \if@borderstar\kern2pt\else\kern -\wd\@ne\fi%
      \global\setbox\@ne\vbox{\box\@ne\if@borderstar\else\kern 2\p@\fi}%
      \vcenter{\if@borderstar\else\kern -\ht\@ne\fi%
        \unvbox\z@\kern-\if@borderstar2\fi\baselineskip}%
        \if@borderstar\kern-2\@tempdima\kern2\p@\else\,\fi\right\@secondoftwo#1 $%
    }\null \;\vbox{\kern\ht\@ne\box\tw@}%
  \endgroup
}
\makeatother

\usepackage{soul}
\usepackage{wasysym}
  
\title{Koszul-type determinantal formulas for families of mixed multilinear
    systems}

  \author{
    Mat\'{i}as R. Bender\thanks{
      Technische Universit\"at Berlin, Institut f\"ur Mathematik,
      Berlin,
      Germany
    (\email{mbender@math.tu-berlin.de}).}
    \and
  Jean-Charles Faug\`ere\thanks{CryptoNext Security and Sorbonne Universit\'e, \textsc{CNRS}, \textsc{INRIA},
    Laboratoire d'Informatique de Paris~6, \textsc{LIP6},
    \'Equipe \textsc{PolSys},
    4 place Jussieu, F-75005, Paris, France
    (\email{jean-charles.faugere@inria.fr}).}
  \and
  Angelos Mantzaflaris\thanks{Inria Sophia-Antipolis M\'editerran\'ee, Universit\'e C\^ote d'Azur, 06902, France (\email{angelos.mantzaflaris@inria.fr})}
  \and
  Elias Tsigaridas\thanks{Inria Paris,  Institut de Mathématiques de Jussieu - Paris Rive Gauche,  Sorbonne Universit\'e and Paris Universit\'e,
    4 place Jussieu, F-75005, Paris, France
    (\email{elias.tsigaridas@inria.fr})}
}

\providecommand{\keywords}[1]
{
  \small	
  \textbf{{Key words.}} \parbox[t]{\linegoal}{#1}
}

\providecommand{\AMS}[1]
{
  \small	
  \textbf{{AMS subject classifications.}} {#1}
}

\begin{document}

\maketitle


\begin{abstract}
  \textit{Effective computation of resultants} is a central problem in
  elimination theory and polynomial system solving. Commonly, we
  compute the resultant as a quotient of determinants of matrices and
  we say that there exists a \textit{determinantal formula} when we
  can express it as a determinant of a matrix whose elements are the
  coefficients of the input polynomials.
  We study the resultant in the context of
  \textit{mixed multilinear polynomial systems}, that is multilinear systems with
  polynomials having 
  different supports, on which determinantal formulas were not known.
  We construct determinantal formulas for two kind of multilinear
  systems related to the \textit{Multiparameter Eigenvalue Problem}
  (MEP): first, when the polynomials agree in all but one block of
  variables; second, when the polynomials are bilinear with different
  supports, related to a bipartite graph.
  We use the \textit{Weyman complex} to construct \textit{Koszul-type
    determinantal formulas} that  generalize Sylvester-type
  formulas. We can use the matrices associated to these formulas to
  solve square systems without computing the resultant.
  The combination of the resultant matrices with the eigenvalue and
  eigenvector criterion for polynomial systems leads to a new
  approach for solving MEP.
\end{abstract}

\medskip

\noindent \keywords{
  Resultant matrix, Multilinear polynomial, Weyman complex, 
  Determinantal formula, Koszul-type matrix, Multiparameter eigenvalue
  problem
  }

\medskip
  
\noindent 
\AMS{
  13P15
  14Q20
  15A18
}


\section{Introduction} 
\label{sec:intro}

One of the main questions in (computational) algebraic geometry is to
decide efficiently when an overdetermined polynomial system has a
solution over a projective variety.
The resultant answers this question. The resultant is a
multihomogeneous polynomial in the coefficients of the polynomials of
the system that vanishes if and only if the system has a solution. We
can also use it to solve square systems.
When we restrict the supports of the input polynomials to make them
sparse, we have an analogous concept called the sparse resultant
\cite{gelfand2008discriminants}.  The sparse resultant is one of the
few tools we can use to solve systems taking into account the sparsity
of the support of the polynomials. Hence, its efficient computation is
fundamental in computational algebraic geometry.

We are interested in the computation of the multiprojective resultant,
as it is defined in
\cite{dandrea_heights_2013,dandrea_poisson_2015,remond_elimination_2001},
of sparse systems given by a particular kind of multilinear
polynomials.
To define the multiprojective resultant as a single polynomial, we
restrict ourselves to systems where the number of equations is one
greater than the dimension of the ambient multiprojective space.
In
what follows, we refer to this specific situation as an overdetermined system%
\footnote{For
  general overdetermined systems, there exists the concept of  \emph{resultant system}, see
  \cite[Sec.~16,5]{waerden_algebra_1991}. }.
In general, we compute the resultant of a polynomial system
$(f_0,\dots,f_n)$ as a quotient of determinants of two matrices whose
elements are polynomials in the coefficients of the input polynomials
\cite{chtcherba2000conditions,chtcherba_constructing_2004,dandrea_macaulay_2002,dandrea2001explicit,jouanolou1997formes,kapur1997extraneous,macaulay1902some};
thus the best we can hope for are linear polynomials.
A classical example of such a matrix is the Macaulay matrix, which
represents a map 
$(g_0,\dots,g_n) \mapsto \sum_i g_i \, f_i$,
where each $g_i$ is a polynomial in a finite dimensional vector space.
In this case, we say that we have a \emph{Sylvester-type} formula.
Other classical formulas include \emph{B\'{e}zout-} and
\emph{Dixon-type}; nevertheless, the elements of the corresponding
matrices are not linear anymore. We refer to
\cite{emiris1999} and references therein for details.

When we can compute the resultant as the determinant of a matrix
we say
that we have a \textit{determinantal formula}.
Besides general constructions that express any multivariate polynomial
as a determinant of a matrix, see for example
\cite{kaltofen_expressing_2008,valiant_completeness_1979}, we are
interested in formulas such that the row/column dimension of the
corresponding matrix depends linearly on the degree of the resultant.
The existence of such formulas is not known in general.
When we consider \textit{unmixed} multihomogeneous systems, that is
when every polynomial has the same support, these formulas are
well studied, e.g.,~\cite{chtcherba2000conditions,dickenstein2003multihomogeneous,sturmfels1994multigraded,weyman_determinantal_1994}.
However, when the supports are not the same, that is in the case of
\textit{mixed} multihomogeneous systems, there are
very few results.
We know determinantal formulas for scaled multihomogeneous systems~\cite{EmiMan-mhomo-jsc-12}, in which case the supports are scaled
copies of one of them,
for bivariate tensor-product polynomial systems 
\cite{buse_matrix_2020}, and for bilinear systems with two different
supports \cite{bender_bilinear_2018}.
One tool to obtain such formulas is using the Weyman complex
\cite{weyman_calculating_1994}. For an introduction to this
complex we refer to \cite[Sec.~9.2]{weyman2003cohomology} and
\cite[Sec.~2.5.C, Sec.~3.4.E]{gelfand2008discriminants}.

Resultant computations are also useful in solving $0$-dimensional square polynomial systems, say
$(f_1,\dots,f_N)$, taking into account the sparsity; here
"square" refers to systems having $N$ polynomials in $N$ variables.
For example we can use the $u$-resultant or hide a variable; we refer to
\cite[Ch.~3]{cox2006using} for a general introduction.
Whenever a \emph{Sylvester-type formula} is available, through the
resultant matrix, we obtain a matrix representing the multiplication
map by a polynomial $f_0$ in
$\K[\bm{x}] / \langle f_1,\dots,f_N \rangle$.
Then, we solve the
system $(f_1,\dots,f_N)$ by computing the eigenvalues and eigenvectors
of this matrix, e.g.,
\cite{auzinger1988elimination,emiris1996complexity}.
The eigenvalues correspond to the evaluations of $f_0$ at the solutions of
the system. From the eigenvectors, at least when there are no
multiplicities, we can recover the coordinates of the solutions.  For
a generalization of this approach to a broader class of resultant
matrices, that encapsulates Sylvester-type matrices as a special case,
we refer to \cite{bender_bilinear_2018}.

\subsection{Multilinear polynomial systems}
\label{sec:contrib}
We focus on computing determinantal formulas for mixed multilinear polynomial
systems. Besides their mathematical importance, as they are the first
non-trivial case of polynomial systems beyond the linear ones, multilinear
systems are also ubiquitous in applications, e.g.,
cryptography~\cite{FauLevPer-crypto-08,Joux-index-14} and game
theory~\cite{McLennan-random-NE}.

For $A, B \in \NN$ let $\bm{X}_1,\dots,\bm{X}_A$,\!\!
$\bm{Y}_1,\dots,\bm{Y}_B$ be blocks of variables.  We present various
\textit{Koszul-type determinantal formulas} (related to the maps in
the Koszul complex, see \Cref{def:KoszulDetFormula})
for the following 
two kinds of \textit{mixed multilinear polynomials systems}
{\small $(f_0,f_1 \dots f_N)$:}
\begin{itemize}[leftmargin=24pt]
\item \textit{star multilinear systems}: these are polynomial 
  systems $(f_1,\dots,f_N)$, where for each $f_k$, there is a
  $j_k \in [B]$
  such that 

\centerline{  $f_k \in \K[\bm{X}_1]_1 \otimes \dots \otimes \K[\bm{X}_{A}]_1 \otimes \K[\bm{Y}_{j_k}]_1,$}

\item \textit{bipartite bilinear systems}: these are polynomial
  systems $(f_1,\dots,f_N)$, where for each $f_k$, there are
  $i_k \in [A]$ and $j_k \in [B]$ such that
  
\centerline{  $f_k \in \K[\bm{X}_{i_k}]_1 \otimes \K[\bm{Y}_{j_k}]_1 .$}
\end{itemize}
To make the system overdetermined and so, to consider its resultant,
we complement it with
several types of multilinear polynomials $f_0$ (see the beginning of
sections \ref{multilinearMixed:sec:mixedUmixed} and
\ref{sec:bipartiteBilinear}).

  Our first main contribution is the theorem below which is an extract of
  Thm.~{\small\ref{multilinearMixed:thm:caseUnmixedMixed}}~and~{\small\ref{bipartite:thm:caseBipartite}.}

  \begin{theorem}
    Let $\bm{f} := (f_1,\dots,f_N)$ be a \emph{star multilinear
      system} (\Cref{def:multilinearMixed}) or a \emph{bipartite
      bilinear system} (\Cref{def:bipartiteBilinear}). Then, for
    certain choices of multilinear polynomials $f_0$ (we present them at the
    beginning of Sec.~\ref{multilinearMixed:sec:mixedUmixed} and
    \ref{sec:bipartiteBilinear}), there is a square matrix $M$ such
    that,
    \begin{itemize}[leftmargin=24pt]
    \item The resultant of $\bm{f}$ agrees with the
      determinant of the matrix,
      $
      \res(\bm{f}) = \pm \det(M)
      $.
    \item The number of columns/rows of $M$ is
      $degree(\res(\bm{f}))$ and
      its elements are coefficients
      of $\bm{f}$, possibly with a sign change.
    \end{itemize}
    The matrix $M$ corresponds to a \emph{Koszul-type
      determinantal formula} (\Cref{def:KoszulDetFormula}) for
    $\bm{f}$.

\end{theorem}

The size of the resultant matrix and the degree of the resultant
depend on the multidegree of $f_0$.  We relate the (expected) number
of solutions of $(f_1,\dots,f_N)$ to the degree of the resultant of
$(f_0,f_1,\dots,f_N)$. For {star multilinear systems}, we present
closed formulas for the expected number of solutions of the system
$(f_1,\dots,f_N)$ and the size of the matrices; we also express the
size of the matrix in terms of the number of solutions.
Our techniques to obtain determinantal formulas exploit the
properties and the parametrization, through a carefully chosen degree
vector, of the Weyman complex and are of independent
interest.
These results generalize the ones  in
\cite{bender_bilinear_2018} that correspond to the case
$(A = 1, B = 2)$.

\subsection{Multiparameter Eigenvalue Problem}
\label{sec:MEP}

A motivating application for the systems and the determinantal
formulas that we study comes from the \emph{multiparameter eigenvalue
  problem} (MEP).  We can model MEP using \emph{star multilinear
  systems} or \emph{bipartite bilinear systems}.  The resultant
matrices that we construct together with the eigenvalue and eigenvector
criterion for polynomial systems, e.g.,
\cite{cox2005inDickensteinEmiris}, lead to a new approach for
solving MEP.

MEP generalizes  the classical
eigenvalue problem.
It arises in mathematical physics as a way of solving ordinary and
partial differential equations when we can use separation of variables
(Fourier method) to solve boundary eigenvalue problems. Its
applications, among others,  include the Spectral and the Sturm-Liouville theory
\cite{atkinson_multiparameter_2010,atkinson_multiparameter_1972,gheorghiu_spectral_2012,hochstenbach_jacobi--davidson_2004,volkmer_multiparameter_1988}.
MEP allows us to solve different eigenvalue problems, e.g.,
the polynomial and the quadratic two-parameter eigenvalue problems
\cite{gohberg_matrix_2005,muhic_quadratic_2010}. It is
an old problem; its origins date from the 1920' in the works of
R.~D.~Carmichael \cite{carmichael_boundary_1921} and A.~J.~Pell~\cite{pell_linear_1922}.

The precise definition of the problem is as follows.
Assume $\alpha \in \N$, $\beta_1,\dots,\beta_\alpha \in \N$, and
consider
matrices
$\{M^{(i,j)}\}_{0 \leq j \leq \alpha} \in \K^{(\beta_i + 1) \times
  (\beta_i + 1)}$, where $0 \leq i \leq \alpha$.
The MEP consists in finding
$\bm{\lambda} = (\lambda_0,\dots,\lambda_{\alpha}) \in \Pr^{\alpha}(\K)$
and
$\bm{v_1}\in~\Pr^{\beta_1}(\K),\dots,\bm{v_\alpha} \in
\Pr^{\beta_\alpha}(\K)$ such that
\begin{equation}
  \label{eq:MEP}
    \Big( \sum\nolimits_{j = 0}^\alpha \lambda_j \, M^{(1,j)} \Big) \, \bm{v_1} = 0 
    \ ,\,\dots, \
    \Big( \sum\nolimits_{j = 0}^\alpha \lambda_j \, M^{(\alpha,j)} \Big) \, \bm{v_\alpha} = 0 
    ,
\end{equation}
where $\K$ is an algebraically closed field and $\Pr^n(\K)$
is the (corresponding) projective space of dimension $n \in  \N$.
We refer to $\bm{\lambda}$ as an eigenvalue,
$(\bm{v_1}, \dots, \bm{v_\alpha})$ as an eigenvector,
and to 
$(\bm{\lambda}, \bm{v_1}, \dots, \bm{v_\alpha})$ as an
eigenpair.
For $\alpha = 1$, MEP is  the \emph{generalized
eigenvalue problem}.

To exploit our tools we need to write MEP as a mixed square bilinear
system. For this we introduce the variables
$\bm{X}_1 = (x_0,\dots,x_\alpha)$ to represent the multiparameter
eigenvalues and, for each $1 \leq i \leq \alpha$, the vectors
$\bm{Y}_t = (y_{i,0},\dots,y_{i,\beta_i})$ to represent the
eigenvectors. This way, we obtain a bilinear system
$\bm{F} = (f_{1,0},\dots,f_{\alpha,\beta_\alpha})$, where for each
$1 \leq t \leq \alpha$,
\begin{align} \label{eq:bilinearMEP}
\small
  \begin{array}{c } \displaystyle
    \left( \sum_{j = 0}^\alpha x_j \, M^{(t,j)} \right)
    \cdot
        \begin{pmatrix}
           y_{t,0} \\
           y_{t,1} \\
           \vdots \\
           y_{t,\beta_t}
         \end{pmatrix}
    = 
    \begin{pmatrix}
      \displaystyle
      \sum_{i = 0}^{\beta_t} \sum_{j = 0}^\alpha M^
      {(t,j)}_{i,0} \, x_j \, y_{t,i} \\
           \vdots \\
      \displaystyle
      \sum_{i = 0}^{\beta_t} \sum_{j = 0}^\alpha  M^{(t,j)}_{i,\beta_t} \, x_j \,  y_{t,i} \\
    \end{pmatrix} = 
    \begin{pmatrix}
      \\[0pt]
           f_{t,0} \\[10pt]
           \vdots \\[10pt]
           f_{t,\beta_t} \\[10pt]
         \end{pmatrix}
  \end{array}
\end{align}
and, for each $1 \leq t \leq \alpha$,
$f_{t,0},\dots,f_{t,\beta_t} \in \K[\bm{X}_1]_1 \otimes \K[\bm{Y}_t]_1$.
In this formulation, the system in \eqref{eq:bilinearMEP} is a
particular case of a \textit{star multilinear system}
(\Cref{def:multilinearMixed}) with $A = 1$ and $B = \alpha$,
or a particular case of 
\textit{bipartite bilinear system}
(\Cref{def:bipartiteBilinear}) with $A = 1$ and $B = \alpha$.
There is a one to one correspondence between the eigenpairs of MEP and
solutions of $\bm{F}$, that is
$$
\begin{array}[]{c}
  (\bm{\lambda}, \bm{v}_1, \dots, \bm{v}_\alpha)
  \text{ is an}
  \\
\text{eigenpair of }
\{M^{(i,j)}\}.
\end{array}
\iff
\!\!\!\!
\begin{array}[]{c}
(\bm{\lambda}, \bm{v}_1, \dots, \bm{v}_\alpha) \in \Pr^{\alpha}(\K) \times \Pr^{\beta_1}(\K) \times \dots
  \times \Pr^{\beta_\alpha}(\K) \\  \text{ and }
  \bm{F}(\bm{\lambda}, \bm{v}_1, \dots,
\bm{v}_\alpha) = 0 .
\end{array}
$$

The standard method to solve MEP is Atkinson's \emph{Delta method}
\cite[Ch.~6, 8]{atkinson_multiparameter_1972}. For each
$0 \leq k \leq \alpha$, it considers the overdetermined system
$\bm{F}_k$ resulting from $\bm{F}$ \eqref{eq:bilinearMEP} by
setting $x_k = 0$.
Then, it constructs a matrix $\Delta_k$ which is nonsingular if and
only if $\bm{F}_k$ has no solutions
\cite[Eq.~6.4.4]{atkinson_multiparameter_1972}.
Subsequently, it applies linear algebra operations to these matrices to
solve the MEP $\bm{F}$ \cite[Thm.~6.8.1]{atkinson_multiparameter_1972} .
It turns out that the matrices $\Delta_k$ are determinantal formulas
for the resultants of the corresponding overdetermined systems
$\bm{F}_k$.
The elements of the matrices of the determinantal formulas $\Delta_k$
are polynomials of degree $\alpha$ in the elements of the matrices
$M^{(i,j)}$ \cite[Thm.~8.2.1]{atkinson_multiparameter_1972}.
The Delta method can only solve \emph{nonsingular MEPs}; these are
MEPs where there exists a finite number of eigenvalues
\cite[Ch.~8]{atkinson_multiparameter_1972}.
The main computational disadvantage of Atkinson's Delta method is the
cost of computing the matrices $\Delta_k$. To compute these matrices
one needs to consider multiple Kronecker products corresponding to the
Laplace expansion of a determinant of size $\alpha \times
\alpha$. The interested reader can find more details in
\cite[Sec.~6.2]{atkinson_multiparameter_1972}.

Besides Atkinson's Delta method, there are recent algorithms such as the
\emph{diagonal coefficient homotopy method} \cite{dong_homotopy_2016}
and the \emph{fiber product homotopy method}
\cite{rodriguez_fiber_2018} which exploit homotopy continuation
methods. These methods seems to be slower than the Delta method, but
they can tackle MEPs of bigger size as they avoid the construction of
the matrices $\Delta_k$.
The Delta method and the homotopy approaches can only compute the
\emph{regular} eigenpairs of the MEP, that is, those where the
eigenpair is an isolated solution of \eqref{eq:bilinearMEP}.
In contrast to the Delta method, experimentally and in some cases,
the \emph{fiber product homotopy method} can also solve \emph{singular
  MEPs}, see \cite[Sec.~10]{rodriguez_fiber_2018}.
We can also use general purpose polynomial system solving algorithms,
that exploit sparsity, to tackle MEP.  We refer reader to
\cite{Faugere2011}, see also \cite{spaenlehauer-phd-12}, for an
algorithm to solve unmixed multilinear systems using Gr\"obner bases,
and to \cite{emiris2016bit,emiris_multilinear_2021} using resultants. We also refer to
\cite{bender_towards_2018,bender_groebner_2019} for an algorithm based
on Gr\"{o}bner bases to solve square mixed multihomogeneous systems
and to \cite{bender_toric_2020,telen2019numerical} for a numerical
algorithm to solve these systems using eigenvalue computations.
However, these generic approaches do not fully exploit the structure
of the problem.

Our second main contribution is a new approach to solve MEP based on
the determinantal formulas that we develop (\Cref{alg:solveMEP}).  We
present a novel way to compute the $\Delta$ matrices by avoiding the
expansion to minors.  Moreover, the resultant matrix from which this
construction emanates has elements that are linear in the elements of
$M^{(i,j)}$, is highly structured, and with few non-zero elements
(\Cref{rem:comparison-w-Atkinson}).  As the Delta method and the
homotopy algorithms, it can only solve \emph{nonsingular MEPs} and
recover the \emph{regular eigenvalues}.
Contrary to the general purpose approaches, it fully exploits the
structure of the system and its complexity relates to the number of
eigenpairs.
Our approach involves the computation of (standard) eigenvalues and
eigenvectors to recover the solutions of the MEP, as the classical symbolic-numeric methods, e.g.,
\cite{cox2005inDickensteinEmiris}.
In particular, our method works with exact coefficients as well
  as with approximations. 
The code for solving MEP using these resultant matrices is
freely available at {\small\url{https://mbender.github.io/multLinDetForm/sylvesterMEP.m}}.

\vspace{3pt}
\paragraph{Organization of the paper}
In \Cref{sec:multihomSys} we define the multihomogeneous systems and
introduce some notation. Later, in \Cref{sec:mult-result}, we
introduce the multihomogeneous resultant. In
\Cref{sec:weyman-complex}, we introduce Weyman complexes and then, in
\Cref{sec:koszul-type-formula}, we explain the Koszul-type
formulas. 
In \Cref{multilinearMixed:sec:mixedUmixed}, we define the \emph{star
  multilinear systems} and we construct \emph{Koszul-type determinantal
  formulas} for multihomogeneous systems involving them, and in
\Cref{multilinearMixed:sec:size-determ-form} we study the number of
solutions of the systems and we compare them with the sizes of the
determinant. We also present an example in
\Cref{multilinearMixed:example}.
In \Cref{sec:bipartiteBilinear}, we define the \emph{bipartite
  bilinear systems} and construct \emph{Koszul-type determinantal
  formulas} for multihomogeneous systems involving them.
Finally, in \Cref{sec:ex-MEP}, we present an algorithm and an
example for solving MEP using our \emph{determinantal~formulas}.

\section{Preliminaries}
\label{sec:preliminaries}

For a  number $N \in \N$ we use the abbreviation
$[N] = \{1, \dots, N \}$.

\subsection{Multihomogeneous systems}
\label{sec:multihomSys}

Let $\K$ be an
algebraically closed field, 
$\K^m$ a vector space, and $(\K^m)^*$ its dual space, where $m \in \NN$.
Let $q \in \N$ and consider $q$ positive natural numbers,
$n_1,\dots,n_q \in \N$.  For each $i \in [q]$,
we consider the following sets of $n_i+1$ variables
$$
\bm{x_i} := \{x_{i,0},\dots, x_{i,n_i}\}
\qquad \text{and} \qquad
\bm{\partial x_i} := \{\partial x_{i,0},\dots, \partial x_{i,n_i}\} . 
$$
We identify the polynomial algebra $\K[\bm{x_i}]$ with the symmetric
algebra of the vector space $\K^{n_i+1}$ and the algebra
$\K[\bm{\partial x_i}]$ with the symmetric algebra of
$(\K^{n_i+1})^{*}$.  That is
$$
\K[\bm{x_i}] \cong S \left(\K^{n_i+1} \right) = \bigoplus_{d \in \Z} S_i(d)
\qquad \text{and} \qquad
\K[\bm{\partial x_i}] \cong S \left((\K^{n_i+1})^* \right) = \bigoplus_{d \in \Z} S^*_i(-d).
$$
Therefore, for each $i$, $S_i(d)$ corresponds to the $\K$-vector space
of polynomials in $\K[\bm{x_i}]$ of degree $d$ and $S_i^*(- d)$ to the
$\K$-vector space of polynomials in $\K[\bm{\partial x_i}]$ of degree
$d$. Note that if $d < 0$, then $S_i(d) = S_i^*(-d) = 0$.
 
We identify the monomials in $\K[\bm{x_i}]$ and
$\K[\bm{\partial x_i}]$ with vectors in $\Z^{n_i+1}$. For each
$\bm{\alpha} = (\alpha_0,\dots,\alpha_{n_i}) \in \Z^{n_i+1}$ we set
$
\bm{x_i}^{\bm{\alpha}} := \prod_{j = 0}^{n_i} x_{i,j}^{\alpha_j},
$ and $
\bm{\partial x_i}^{\bm{\alpha}} := \prod_{j = 0}^{n_i} \partial x_{i,j}^{\alpha_j}.
$
We consider the $\ZZ^q$-graded polynomial algebra
$$
\K[\bm{\bar{x}}] :=
\K[\bm{x_1}] \otimes \dots \otimes \K[\bm{x_q}] \cong
\bigoplus_{ (d_1,\dots,d_q)  \in  \Z^q}
S_1(d_1) \otimes \dots \otimes S_q(d_q).
$$
Notice that for each $\bm{d} = (d_1,\dots,d_q) \in \Z^q$,
$\K[\bm{\bar{x}}]_{\bm{d}}$ is the $\K$-vector space of the
multihomogeneous polynomials of multidegree $\bm{d}$, that is  
polynomials in $\K[\bm{\bar{x}}]$ having degree  $d_i$ \wrt the set of
variables $\bm{x_i}$, for each $i \in [q]$. We say that
a polynomial $f \in \K[\bm{\bar{x}}]$ is multihomogeneous of
multidegree $\bm{d} \in \Z^q$, if
$f \in \K[\bm{\bar{x}}]_{\bm{d}} = S_1(d_1) \otimes \dots \otimes
S_q(d_q)$.

\begin{example}
  Consider the two blocks of variables
  $\bm{x_1} \! = \{x_{1,0},x_{1,1}\}$ and
  $\bm{x_2} \! = \{x_{2,0},x_{2,1}, x_{2,2}\}$ and the polynomial
  $x_{1,0}^2 \, x_{1,1} \otimes x_{2,0} \, x_{2,1} + x_{1,0} \,
  x_{1,1}^2 \otimes x_{2,2}^2 \in \K[\bm{x_1},\bm{x_2}]$.  It is
  multihomogeneous of multidegree $(3,2)$ and we write it simply as
  $x_{1,0}^2 \, x_{1,1} \, x_{2,0} \, x_{2,1} + x_{1,0} \, x_{1,1}^2
  \, x_{2,2}^2$.
\end{example}

Following standard notation, we write the monomials of
$\K[\bm{\bar{x}}]$ as $\prod_{i = 1}^{q} \bm{x_{i}}^{\bm{\alpha_i}}$
instead of $\bigotimes_{i = 1}^{q} \bm{x_{i}}^{\bm{\alpha_i}}$.  We
identify these monomials with vectors in
$\Z^{n_1+1} \times \dots \times \Z^{n_q+1}$. For each
$\bm{\alpha} = (\bm{\alpha_1},\dots,\bm{\alpha_q}) \in \Z^{n_1+1}
\times \dots \times \Z^{n_q+1}$ we set
$
\bm{x^{\alpha}} := \prod_{i = 1}^{q} \bm{x_{i}}^{\bm{\alpha}_i}.
$
For each multidegree $\bm{d} \in \Z^q$, we denote by
$\A(\bm{d})$ the set of exponents of the monomials of multidegree
$\bm{d}$, that is 
$
\A(\bm{d}) = \{ \bm{\alpha} \in \Z^{n_1+1} \times \dots \times
\Z^{n_q+1} : \bm{x}^{\bm{\alpha}} \in
\K[\bm{\bar{x}}]_{\bm{d}}\}.
$
The cardinality of $\A(\bm{d})$ is 
$
\#\A(\bm{d}) = \prod_{i = 1}^q {d_i + n_i \choose n_i}.
$

We fix $N = n_1 + \dots + n_q$.  Let $\P := \Pr^{n_1} \times \dots
\times \Pr^{n_q}$ be a multiprojective space over
$\K$, where $\Pr^{n_i} :=
\Pr^{n_i}(\K)$. A system of multihomogeneous polynomials is a set of
multihomogeneous polynomials in $\K[\bm{\bar{x}}]$.
We say that a system of
multihomogeneous polynomials $\{f_1,\dots,f_r\}$ has a solution in
$\P$, if there is $\gamma \in \P$ such that $f_i(\gamma) = 0$, for
every $1 \leq i \leq r$. We call a system of multihomogeneous
polynomials \emph{square} if $r = N$, and
\emph{overdetermined} if $r=N+1$.  Generically, square
multihomogeneous systems have a finite number of solutions over $\P$,
while the overdetermined ones do not have solutions.
The following proposition bounds the number of solutions
of square multihomogeneous polynomial systems.

  \begin{proposition}
    [{{Multihomogeneous \bezout bound, \cite[Example 4.9]{safarevic_varieties_2013}}}]
    \label{intro:thm:bezoutBound}
    Consider a  square multihomogeneous system
    $\bm{f} := \{f_1,\dots,f_{N}\}$ of multidegrees 
    $\bm{d}_1,\dots,\bm{d}_N \in \N^q$ with
    $\bm{d}_i = (d_{i,1},\dots,d_{i,q})$, for each $1 \leq i \leq
    N$. If $\bm{f}$ has a finite number of solutions over $\P$, then
    their number, counted with multiplicities, see
    \cite[Sec.~4.2]{cox2006using}, is the coefficient of the
    monomial $\prod_i Z_i^{n_i}$ in the polynomial
    $
    \prod_{k=1}^{N}
    \left( \sum_{i = 1}^q d_{k,i} \, Z_i \right).
    $
    We refer to this coefficient as the
    \emph{multihomogeneous \bezout bound} and we will write it as
    $\MHB(\bm{d}_1,\dots,\bm{d}_N)$.
  \end{proposition}
  
  The \mhbb is generically tight, see
    {\cite[Thm.~1.11]{dandrea_heights_2013}}.  
 
  \subsection{Multihomogeneous resultant}
  \label{sec:mult-result}
  
  We fix $N+1$ multidegrees
  $\bm{d_0},\bm{d_1},\dots,\bm{d_N} \in \N^q$.  To characterize the
  overdetermined multihomogeneous systems of such  multidegrees
  having solutions over $\P$, we parameterize them and we introduce the
  resultant. The latter is a polynomial in the coefficients of the
  polynomials of the system that vanishes if and only if the system has a solution
  over $\P$. Our presentation follows \cite[Ch.~3.2]{cox2006using}
  adapted to the multihomogeneous~case.

\begin{definition}
  [Generic multihomogeneous system]
  \label{intro:def:genMultiHom}
  Consider the set of variables
  $\bm{u} := \{u_{k,\bm{\alpha}} : 0 \leq k \leq N \text{ and }
  \bm{\alpha} \in \A(\bm{d}_k)\}$ and the ring $\Z[\bm{u}]$. The generic
  multihomogeneous polynomial system is the system
  $\bm{F} := \{F_0,\dots,F_N\} \subset \Z[\bm{u}][\bm{\bar{x}}]$, where
  \begin{equation}
    \label{eq:Fk-u}
    F_k := \sum\nolimits_{\bm{\alpha}\in \A({\bm{d}_k})} u_{k,\bm{\alpha}} \bm{x}^{\bm \alpha}.
  \end{equation}
  \end{definition}

  \noindent
  The generic multihomogeneous system $\bm{F}$ parameterizes every
  overdetermined multihomogeneous system with polynomials of
  multidegrees ${\bm{d}_0},{\bm{d}_1},\dots,{\bm{d}_N}$,
  respectively. For each
  $\bm{c} = (c_{k,\bm{\alpha}})_{0 \leq k \leq N, \bm{\alpha} \in
    \A({\bm{d}_k})} \in \Pr^{\#\A({\bm{d}_0})-1} \times \dots \times
  \Pr^{\#\A({\bm{d}_N})-1}$, the specialization of $\bm{F}$ at
  $\bm{c}$, that is $\bm{F}(\bm{c})$, is a multihomogeneous
  polynomial system in $\K[\bm{\bar{x}}]$, say $(f_0,\dots,f_N)$,
  where
  \begin{equation}
    \label{eq:fk-c}
  f_k := F_k(\bm{c}) = \sum\nolimits_{\bm{\alpha}\in \A({\bm{d}_k})} c_{k,\bm{\alpha}} \bm{x}^{\bm \alpha}.
  \end{equation}

  Let $\Omega$ be its \textit{incidence variety}, that is the
  algebraic variety containing the overdetermined multihomogeneous
  systems that have solutions over $\P$ and their solutions,
  $$
  \Omega = \left\{ (\bm{p}, \bm{c}) \in
    \P \times
    (\Pr^{\#\A({\bm{d}_0})-1} \times \dots \times \Pr^{\#\A({\bm{d}_N})-1})
    : (\forall k \in [N]) \, F_k(\bm{c})(\bm{p}) = 0 \right\}.
  $$

  Let $\pi$ be the projection of $\Omega$ to
  $\Pr^{\#\A({\bm{d}_0})-1} \times \dots \times \Pr^{\#\A({\bm{d}_N})-1}$,
  that is
  $\pi(\bm{p}, \bm{c}) = \bm{c}$.
  We can think of $\pi(\Omega)$ as the set of overdetermined
  multihomogeneous polynomial systems with solutions over $\P$. This
  set is an irreducible hypersurface
  \cite[Prop.~3.3.1]{gelfand2008discriminants}. Its defining
  ideal in $\Z[\bm{u}]$ is principal and it is generated by an
  irreducible polynomial $\mathtt{elim} \in \Z[\bm{u}]$
  \cite[Prop.~8.1.1]{gelfand2008discriminants}. In particular,
  it holds
  $$
  \text{The system }\bm{F}(\bm{c})\text{ has a solution over }\P \iff
  \bm{c} \in \pi(\Omega) \iff \mathtt{elim}(\bm{c}) = 0. $$

  Following
  \cite{remond_elimination_2001,dandrea_heights_2013,dandrea_poisson_2015},
  we call $\mathtt{elim} \in \Z[\bm{u}]$ the \textit{eliminant}.
  We warn the reader that the
  polynomial $\mathtt{elim} \in \Z[\bm{u}]$ is called the \textit{resultant}
  in \cite{gelfand2008discriminants}.
  In this work we reserve the word resultant for a power of $\mathtt{elim}$.
  More precisely, the resultant $\res$ is a polynomial in $\Z[\bm{u}]$
  such that
  $\res = \pm \mathtt{elim}^{\mathcal{D}},$
  where $\mathcal{D}$ is the degree of the restriction of $\pi$ to the
  incidence variety $\Omega$, see
  \cite[Def.~3.1]{dandrea_poisson_2015}.
  Consequently, we have 
  \begin{align}
    \label{eq:defResultant}
  \text{The system }\bm{F}(\bm{c})\text{ has a solution over }\P \iff
    \res(\bm{c}) = 0.
  \end{align}
  
  \begin{proposition}[{{\cite[Prop.~3.4]{remond_elimination_2001}}}]
    \label{intro:thm:degRes}
    Let ${\bm{u}_{k}}$ be the blocks of variables in $\bm{u}$ related to
    the polynomial $F_k$, that is
    ${\bm{u}_k} = \{{\bm{u}_{k,\bm{\alpha}}}\}_{\bm{\alpha} \in \A({\bm{d}_i})}.$
    The resultant $\res \in \Z[\bm{u}]$ is a multihomogeneous polynomial
    with respect to the blocks of variables
    ${\bm{u}_0},\dots,{\bm{u}_N}$. The degree of $\res$ \wrt the
    variables ${\bm{u}_k}$ is the multihomogeneous \bezout bound
    (\Cref{intro:thm:bezoutBound}) of a square system with
    multidegrees
    {\small ${\bm{d}_0} \dots {\bm{d}_{k-1}},{\bm{d}_{k+1}} \dots {\bm{d}_N}$,}
    $$\text{degree}(\res,{\bm{u}_k}) = \MHB({\bm{d}_0},\dots,{\bm{d}_{k-1}},{\bm{d}_{k+1}},\dots,{\bm{d}_N}).$$
    The total degree of the resultant is
    $\text{degree}(\res) = \sum_{k = 0}^N \MHB({\bm{d}_0},\dots,{\bm{d}_{k-1}},{\bm{d}_{k+1}},\dots,{\bm{d}_N}).$
  \end{proposition}

  Usually we compute the resultant as the  quotient of the determinants of two 
  matrices, see \cite[Thm.~3.4.2]{gelfand2008discriminants}. When we
  can compute it as the (exact) determinant of a matrix, then we say
  that we have a \emph{determinantal formula}.
  
  \subsection{Weyman complex}
  \label{sec:weyman-complex}
  
  A complex $K_{\bullet}$ is a sequence of free modules
  $\{K_v\}_{v \in \Z}$ together with morphisms
  $\delta_v : K_v \rightarrow K_{v-1}$, such that the image of
  $\delta_v$ belongs to the kernel of $\delta_{v-1}$, that is
  $(\forall v \in \Z) \; \Im(\delta_v) \subseteq \Ker(\delta_{v-1})$
  or, equivalently, $\delta_{v-1} \circ \delta_{v} = 0$.
  We write  $K_\bullet$ as
  $$ K_\bullet : \cdots
  \xrightarrow{\delta_{v+1}}
  K_v \xrightarrow{\delta_v} K_{v-1} \xrightarrow{\delta_{v-1}}
  \cdots .$$
  The complex is  exact if
  for all 
  $v \in \Z$ it holds $\Im(\delta_v) = \Ker(\delta_{v-1})$.
  A complex is bounded when there are two constants $a$ and $b$ such
  that, for every $v$ such that $v < a$ or $b < v$, it holds $K_v = 0$.
  If we fix a basis for each $K_v$, then we can represent the maps
  $\delta_v$ using matrices.
  For a particular class of bounded complexes, called generically
  exact (see for example \cite[Ap.~A]{gelfand2008discriminants}), we
  can extend the definition of the determinant of matrices to
  complexes. The non-vanishing of the determinant is related to the
  exactness of the complex.
  When there are only two non-zero free modules in the complex (that
  is all the other modules are zero) we can define the determinant of
  the complex if and only if both the non-zero free modules have the
  same rank. In this case, the determinant of the complex reduces to
  the determinant of the (matrix of the) map between the two non-zero
  vector spaces.
  We refer the reader to \cite{anokhina_resultant_2009} for an
  accessible introduction to the determinant of a complex and to
  \cite[Ap.~A]{gelfand2008discriminants} for a complete formalization.

  The Weyman complex
  \cite{weyman_calculating_1994,weyman_determinantal_1994,weyman2003cohomology}
  of an overdetermined multihomogeneous system
  $\bm{f} = (f_0,\dots,f_N)$ in $\K[\bm{\bar{x}}]$ is a bounded complex that is
  exact if and only if the system $\bm{f}$ has no solutions over
  $\P$. More precisely, the determinant of the Weyman complex of the
  multihomogeneous generic system $\bm{F}$ (see
  \Cref{intro:def:genMultiHom}) is well-defined and it is equal to
  the multihomogeneous resultant
  \cite[Prop.~9.1.3]{weyman2003cohomology}.
  If the Weyman complex involves only two non-zero vector spaces, then
  the resultant of $\bm{F}$ is the determinant of the map between
  these spaces. Thus, in this case, there is a determinantal formula
  for the resultant.

    \begin{theorem}[Weyman complex,
      {\cite[Prop.~2.2]{weyman_determinantal_1994}}]
    \label{intro:res:thm:weymanComplex}
    Let $\bm{F} = (F_0,\dots,F_N)$ in $\Z[\bm{u}][\bm{\bar{x}}]$ be a
    generic multihomogeneous system having multidegrees
    ${\bm{d}_0},\dots,{\bm{d}_N}$, respectively (see
    \Cref{intro:def:genMultiHom}). Given a \emph{degree vector}
    $\bm{m} \in \Z^q$, there exists a complex of free
    $\Z[\bm{u}]$-modules $K_\bullet(\bm{m})$, called the \emph{Weyman
      complex} of $\bm{F}$,
      such that the determinant of the complex $K_\bullet(\bm{m})$ agrees
  with the resultant $\res(F_0,\dots,F_N)$.
  $$ K_\bullet(\bm{m}) \! : \! 0 \rightarrow K_{N+1}(\bm{m})
  \xrightarrow{\delta_{N+1}(\bm{m})} \cdots \rightarrow 
  K_1(\bm{m}) \xrightarrow{\delta_1(\bm{m})} K_0(\bm{m}) \xrightarrow{\delta_0(\bm{m})}
  \cdots \rightarrow 
  K_{-N}(\bm{m}) \rightarrow 0.$$

  Moreover, for each $v \in \{-N, \dots, N+1 \}$ the
  $\Z[\bm{u}]$-module $K_v(\bm{m})$ is
  \begin{align}\label{eq:Kvp}
    K_v(\bm{m}) \! := \!\!\bigoplus_{p=0}^{N+1} K_{v,p}(\bm{m}) \otimes_\Z \Z[\bm{u}], \text{ where } \,\,\!
    K_{v,p}(\bm{m}) := \!\!\!\!\!\!\!
    \bigoplus_{\substack{I \subset \{0,\dots,N\} \\ \#I = p}} \!\!\!\!\!
    H^{p-v}_{\P} ( \bm{m} - \sum_{k \in I} {{\bm{d}_k}} ) \otimes \!\! \bigwedge_{k \in I} e_k,
  \end{align}
  the term $H^{p-v}_{\P}(\bm{m} - \sum_{k \in I} {{\bm{d}_k}})$ is the
  $(p-v)$-th sheaf cohomology of $\P$ with coefficients in the sheaf
  $\mathcal O_\P(\bm{m} - \sum_{k \in I} {{\bm{d}_k}})$ whose global
  sections are $\K[\bm{\bar{x}}]_{\bm{m} - \sum_{k \in I} {{\bm{d}_k}}}$ (see
  \cite[Sec.~II.5]{Hart77}),
  \footnote{ The standard notation for the sheaf cohomology
    $H^{p}_{\P}(\bm{m})$, e.g., \cite{Hart77}, is
    $H^{p}({\P},\mathcal L(\sum_i m_i \, D_i))$, where each $D_i$ is
    a Cartier divisor given by the pullback of a hyperplane on
    $\Pr^{n_i}$ (via projection) and $\mathcal L(\sum_i m_i \, D_i)$
    is the line bundle associated to the Cartier divisor
    $(\sum_i m_i \, D_i)$ on $\P$. We use our notation for
    simplicity.
  }
  and the element $\bigwedge_{k \in I} e_k$ is the singleton
  $\{ e_{I_1} \wedge \dots \wedge e_{I_p} \}$, where
  $I_1 < \dots < I_p$ are the elements of $I$, $e_0, \! \dots, \! e_N$ is the
  standard basis of $\K^{N+1}$, and $\wedge$ is the wedge (exterior)
  product.
\end{theorem}

    For a multihomogeneous system
  $\bm{f} = (f_0,\dots,f_N)$ in $\K[\bm{\bar{x}}]$ that is the
  specialization of $\bm{F}$ at $\bm{c}$, see \eqref{eq:fk-c}, the
  Weyman complex $K_\bullet(\bm{m} ; \bm{f})$ is the Weyman complex
  $K_\bullet(\bm{m})$ where we specialize each variable
  $u_{k,\bm{\alpha}}$ at $c_{k,\bm{\alpha}} \in \K$.

\begin{proposition}
  [{{\cite[Prop.~2.1]{weyman_determinantal_1994}}}]
  The vector spaces $K_v(\bm{m},\bm{f})$ are independent of the
  specialization of the variables $\bm{u}$, in particular
  $K_v(\bm{m}, \bm{f}) = \bigoplus_{p=0}^{N+1} K_{v,p}(\bm{m}).$
  Hence, the rank of $K_v(\bm{m})$ as a $\Z[\bm{u}]$-module equals the
  dimension of $K_v(\bm{m}, \bm{f})$ as a $\K$-vector space.  The
  differentials $\delta_v(\bm{m},\bm{f})$ depend on the coefficients
  of $\bm{f}$.
\end{proposition}

Following \cite{weyman_determinantal_1994}, as $\P$ is a product of
projective spaces, we use K\"unneth formula 
(\Cref{intro:thm:kunnethFormula}) to write the cohomologies in
\eqref{eq:Kvp} as a product of cohomologies of projective spaces, that in
turn we can identify with polynomial rings.

\begin{proposition}[K\"unneth Formula]
  \label{intro:thm:kunnethFormula}
  The cohomologies of the product of projective spaces in each
  $K_{v,p}(\bm{m})$ of \eqref{eq:Kvp} are the direct sum of the tensor
  product of the cohomologies of each projective space, that is
  \begin{align}
    \label{eq:kunneth}
    H^{p-v}_{\P} \Big( \bm{m} - \sum_{k \in I} {\bm{d}_k} \Big) \cong \bigoplus_{r_1 + \dots + r_q = p-v} \bigotimes_{i=1}^q H^{r_i}_{\Pr^{n_i}} \Big( m_i - \sum_{k \in I} d_{k,i} \Big).
  \end{align}
\end{proposition}

By combining Bott formula and Serre's duality, see
\cite[Sec.~1.1]{okonek_vector_1980}, we can identify the cohomologies of
the previous proposition with the  rings $\K[\bm{x_i}]$ and
$\K[\bm{\partial x_i}]$.
Moreover, for each $p - v$, there is at most one set of values for
$(r_1,\dots,r_q)$ such that
every cohomology in the tensor product of the right hand side of the
previous equation does not vanish. In other words, the right hand side of
\eqref{eq:kunneth} reduces to the tensor product of certain cohomologies
of different projective spaces.

\begin{remark} \label{intro:thm:bottFormula}
    For each $1 \leq i \leq q$, $a \in \Z$, it holds
  \begin{itemize}[leftmargin=.65cm]
  \item $H^0_{\Pr^{n_i}}(a) \cong S_i(a)$, that is the $\K$-vector
    space of the polynomials of degree $a$ in the polynomial
    algebra $\K[{\bm{x}_i}]$.
  \item $H^{n_i}_{\Pr^{n_i}}(a) \cong S_i^*(a + n_i + 1)$, that is the
    $\K$-vector space of the polynomials of degree $a + n_i + 1$ in
    the polynomial algebra $\K[{\bm{\partial x}_i}]$.
  \item If $r_i \not\in \{0,n_i\}$, then $H^{r_i}_{\Pr^{n_i}}(a) \cong 0$.
  \end{itemize}
\end{remark}

\begin{remark}
  \label{intro:thm:possiblesRi}
    For each $1 \leq i \leq q$, if $H^{r_i}_{\Pr^{n_i}}(a) \neq 0$, then
  $r_i \in \{0,n_i\}$. Moreover,
  \begin{itemize}[leftmargin=.65cm]
  \item If $a > -n_i - 1$, then \hfil 
    $\displaystyle H^{r_i}_{\Pr^{n_i}}(a) \neq 0 \iff r_i = 0 \text{ and } a \geq 0.$
  \item If $a < 0$, then  \hfil 
    $\displaystyle
    \qquad H^{r_i}_{\Pr^{n_i}}(a) \neq 0 \iff r_i = n_i  \text{ and } a \leq -n_i - 1.$
  \end{itemize} 
\end{remark}

We define the dual of a complex by dualizing the modules and the
maps. The dual of the Weyman complex is isomorphic to another Weyman
complex. By exploiting Serre's duality, we can construct the degree
vectors of a dual Weyman complex from the degree vector of the primal.
\begin{proposition} 
  [{{\cite[Thm.~5.1.4]{weyman2003cohomology}}}]
  \label{prop:dualDegreeVectors}
  Let $\bm{m}$ and
   $\bm{\bar{m}}$ be any degree vectors such that
   $\bm{m} + \bm{\bar{m}} = \sum_i \bm{d}_i - (n_1 + 1,\dots,n_q + 1)$.
   Then, $K_v(\bm{m}) \cong K_{1-v}(\bm{\bar{m}})^*$ for all $v\in\ZZ$ and
   $K_\bullet(\bm{m})$ is dual to
   $K_\bullet(\bm{\bar{m}})$.
\end{proposition}

\subsection{Koszul-type formula}
\label{sec:koszul-type-formula} 

Our goal is to obtain determinantal formulas
given by matrices whose
elements are linear forms in the coefficients of the input
polynomials, that is linear in $\bm{u}$, see \eqref{eq:Fk-u}.
Hence, by \cite[Prop.~5.2.4]{weyman2003cohomology}, our objective is
to choose a degree vector $\bm{m}$ so that the Weyman complex reduces
to
\begin{align} \label{eq:determinantalMap}
  K_\bullet(\bm{m}) : 0
  \rightarrow K_{v, p+v}(\bm{m}) \otimes \Z[\bm{u}]
  \xrightarrow{\delta_v(\bm{m})}
  K_{v-1, p+v-1}(\bm{m}) \otimes \Z[\bm{u}] \rightarrow 0 \, ,
 \end{align}
 where $p = \sum_{k \in I} n_k$, for some set
 $I \subset \{1,\dots,q\}$. That is, it holds
 $K_{v}(\bm{m}) = K_{v, p+v}(\bm{m}) \otimes \Z[\bm{u}]$,
 $K_{v - 1}(\bm{m}) = K_{v - 1, p+v-1}(\bm{m}) \otimes \Z[\bm{u}]$,
 and, for all $t \not\in \{v-1,v\}$, $K_{t}(\bm{m}) = 0$,
 
 We will describe the map $\delta_v(\bm{m})$ through an auxiliary map
 $\mu$ that acts as multiplication.
 For this, we need to introduce some additional notation.
 This notation is independent from the rest of the paper and the
 readers that are familiar with the Weyman complex can safely skip the
 rest of the section.
 Let $R$ be a ring; for example $R = \Z[\bm{u}]$ or $R = \K$.
  For each $1 \leq i \leq q$, the polynomial ring $R[{\bm{x}_i}]$,
 respectively $R[{\bm{\partial x}_i}]$, is a free $R$-module with
 basis $\{ \bm{x}_i^{\bm\alpha} : \bm{\alpha} \in \A(d), d \in \Z \}$,
 respectively
 $\{ \bm{\partial x}_i^{\bm\alpha} : \bm{\alpha} \in \A(d), d \in \Z \}$.
 We define the bilinear map
 \begin{align}
   \label{eq:mapMui}
 \mu_{(i)} : R[{\bm{x}_i}] \times (R[{\bm{x}_i}] \oplus  R[{\bm{\partial x}_i}])
 \rightarrow R[{\bm{x}_i}] \oplus R[{\bm{\partial x}_i}] ,
 \end{align}
 which acts follows: for each $d_1, d_2 \in \Z$,
 $\bm{\alpha} \in \A(d_1)$ and
 $\bm{\beta}, \bm{\gamma} \in \A(d_2)$, we have
 \begin{align*}
   \mu_{(i)}(\bm{x}_i^{\bm{\alpha}}, \bm{x}_i^{\bm{\gamma}}) =  \; \bm{x}_i^{{\bm{\alpha}} + {\bm{\gamma}}} \,
   \text{ and }\,
   \mu_{(i)}(\bm{x}_i^{\bm{\alpha}}, \bm{\partial x}_i^{\bm{\beta}}) = &
   \left\{
   \begin{array}{l l}
     \bm{\partial x}_i^{{\bm{\beta}} - {\bm{\alpha}}} & \text{if $d_1 \leq d_2$ and ${\bm{\beta}} - {\bm{\alpha}} \in \A(d_1 - d_2)$} \!\!\!\!\!\! \\
     0 & \text{otherwise}.
   \end{array}
         \right. 
   \end{align*}
   The map $\mu_{(i)}$ is graded in the following way, for
   $f \in S_i(d)$ it holds
   $$\mu_{(i)}(f,S_i(D)) \subseteq S_i(D + d) \quad \text{and} \quad
   \mu_{(i)}(f,S_i^*(D)) \subseteq S_i^*(D + d).$$
 
 We define the bilinear map $\mu := \bigotimes_{i = 1}^q \mu_{(i)}$.

\begin{remark}
  If we restrict the domain of $\mu$ to
  $
  \bigotimes_{i = 1}^q
  \left(  R[{\bm{x}_i}] \times 
  R[{\bm{x}_i}] \right) 
\simeq
\left( \bigotimes_{i = 1}^q R[{\bm{x}_i}] \right) \times \left(
    \bigotimes_{i = 1}^q R[{\bm{x}_i}] \right) $, then $\mu$ acts as
  multiplication, ie for
  $f,g \in \left( \bigotimes_{i = 1}^q R[{\bm{x}_i}] \right)$, it
  holds
   $
   \mu(f, g) = f \, g.
   $
\end{remark}

Given $f \in \left( \bigotimes_{i = 1}^q R[{\bm{x}_i}] \right)$, we
define the linear map
 \begin{align*}
   \begin{array}{r r c}
 \mu_f : & \bigotimes_{i = 1}^q (R[{\bm{x}_i}] \oplus R[\bm{\partial
   x_i}])  \rightarrow & \bigotimes_{i = 1}^q (R[{\bm{x}_i}] \oplus
 R[{\bm{\partial x}_i}]) \\
 & g  \mapsto &  \mu_f(g) = \mu(f,g).
   \end{array}
\end{align*}

\noindent
Using the isomorphisms of \Cref{intro:thm:kunnethFormula} and
\Cref{intro:thm:bottFormula}, for $\bm{d} \in \N^q$ and
$f \in \K[\bm{\bar{x}}]_{\bm{d}}$, if we restrict the map $\mu_f$ to
$H^{r}_{\P}(\bm{m})$, for any $r \in \N$, then we obtain the map
$\mu_f : H^{r}_{\P}(\bm{m}) \rightarrow H^{r}_{\P}(\bm{m} +
\bm{d}).$

\begin{definition}
  [Inner derivative {\cite{weyman_determinantal_1994}}]
  \label{def:innerDerivative}
  \!\!\!\!\!\!
  Let $E$ be a $\K$-vector space generated by \!\! $\{e_1, \dots e_N\}$. We
  define the $k$-th inner derivative, $\inDev_k$, of the exterior
  algebra of $\bigwedge E$ as the $(-1)$-graded map such that, for
  each $i$ and $1 \leq j_1 < \dots < j_i \leq N$,
  \begin{align*}
    \small
    \begin{array}{r c l}
      \inDev_k : & \bigwedge^i E & \!\!\!\!\! \rightarrow  \bigwedge^{i-1} E \\ 
                 & \!\!\!\!\!\!\!\!\!\! e_{j_1} \wedge \dots \wedge e_{j_i} & \!\!\!\!\! \mapsto \inDev_k(e_{j_1} \wedge \dots \wedge e_{j_i}) =
                                                                 \left\{ 
                                                                 \begin{array}{c l}
                                                                   (-1)^{t+1}
                                                                   e_{j_1} \wedge \dots \wedge e_{j_{t-1}} \wedge e_{j_{t+1}} \wedge \dots \wedge e_{j_i}
                                                                   \!\!\! &
                                                                     \text{if } j_t = k\\
                                                                   0 & \text{otherwise} 
                                                                 \end{array}
                                                                       \right.
    \end{array}
  \end{align*}
\end{definition}

Given $R$-modules $A_1,A_2,B_1,B_2$ and homomorphisms
$\mu_1 : A_1 \rightarrow B_1$ and $\mu_2 : A_2 \rightarrow B_2$, their
tensor product is the map
$\mu_1 \otimes \mu_2: A_1 \otimes A_2 \rightarrow B_1 \otimes B_2$
such that, for $a_1 \in A_1$ and $a_2 \in A_2$,
 $(\mu_1 \otimes \mu_2)(a_1 \otimes a_2) = \mu_1(a_1) \otimes \mu_2(a_2).$

\begin{proposition}
  [{{\cite[Prop.~2.6]{weyman_determinantal_1994}}}]
  \label{intro:def:KoszulMap}
  Consider the generic overdetermined multihomogeneous system
  $\bm{F} \in \Z[\bm{u}][\bm{\bar{x}}]^{N+1}$ with polynomials of
  multidegrees ${\bm{d}_0},\dots,{\bm{d}_N}$, respectively
  (\Cref{intro:def:genMultiHom}). Given a degree vector
  $\bm{m} \in \Z^q$, we consider the Weyman complex
  $K_\bullet(\bm{m})$. If there is $v \in \{-N+1,\dots,N+1\}$ and
  $p \in \{0,\dots,N+1\}$ such~that
  $$K_v(\bm{m}) = K_{v,p}(\bm{m}) \otimes \Z[\bm{u}] \qquad \text{and} \qquad
  K_{v-1}(\bm{m}) = K_{v-1,p-1}(\bm{m}) \otimes \Z[\bm{u}],$$ then the map
  $\delta_v(\bm{m}) : K_v(\bm{m}) \rightarrow K_{v-1}(\bm{m})$ is
  $\delta_v(\bm{m}) = \sum_{k = 0}^N \mu_{F_k} \otimes \inDev_k,$
  where $\mu_{F_k} \otimes \inDev_k$ denotes the tensor product of the
  maps $\mu_{F_k}$ and $\inDev_k$
  (\Cref{def:innerDerivative}).
\end{proposition}

\begin{definition}[Koszul-type determinantal formula]
  \label{def:KoszulDetFormula}
  With the notation of \Cref{intro:def:KoszulMap}, when the Weyman
  complex reduces to
  \begin{align} \label{eq:KoszulDetFormula}
    K_\bullet(\bm{m}) : 0
    \rightarrow K_{1, p+1}(\bm{m}) \otimes \Z[\bm{u}]
    \xrightarrow{\delta_1(\bm{m})}
    K_{0, p}(\bm{m}) \otimes \Z[\bm{u}] \rightarrow 0 \, ,
  \end{align}
  we say that the map $\delta_1(\bm{m})$ is a \emph{Koszul-type determinantal formula}.
\end{definition}

\begin{example}

   Consider the blocks of variables
   $\bm{x}_1 := \{ x_{1,0}, x_{1,1} \}$ and
   $\bm{x}_2 \! := \! \{x_{2,0}, x_{2,1}\}$, and the systems
   $\bm{f} := (f_0,f_1,f_2)$ of multidegrees
   $\bm{d}_0 = \bm{d}_1 = \bm{d}_2 = (1,1)$. That
   is,
\begin{align}
\begin{cases}
  f_0 = (a_{0,0} \, x_{1,0} + a_{1,0} \, x_{1,1}) \, x_{2,0} +
        (a_{0,1} \, x_{1,0}  + a_{1,1} \,  x_{1,1} )  \, x_{2,1} \\
  f_1 = (b_{0,0} \, x_{1,0} + b_{1,0} \, x_{1,1}) \, x_{2,0} +
        (b_{0,1} \, x_{1,0}  + b_{1,1} \,  x_{1,1} )  \, x_{2,1} \\
  f_2 = (c_{0,0} \, x_{1,0} + c_{1,0} \, x_{1,1}) \, x_{2,0} +
        (c_{0,1} \, x_{1,0}  + c_{1,1} \,  x_{1,1} )  \, x_{2,1}.
\end{cases}
\end{align}
As in \cite[Lem.~2.2]{emiris2016bit}, consider the degree vector
$\bm{m} = (2,-1)$. So, the Weyman complex is
  $$ K_\bullet(\bm{m}, \bm{f}) : 0 \rightarrow
  K_{1, 2}(\bm{m}, \bm{f}) \xrightarrow{\delta_1(\bm{m}, \bm{f})} K_{0, 1}(\bm{m}, \bm{f})
  \rightarrow 0,$$
where
\begin{align*}
  \left\{
  \begin{array}{l c}
    K_{1, 2}(\bm{m}, \bm{f}) = & S_1(0) \otimes S_2^*(-1) \otimes
                         \Big(
                         \{e_0 \wedge e_1\} \oplus
                         \{e_0 \wedge e_2\} \oplus
                         \{e_1 \wedge e_2\}
                         \Big) \\
    K_{0, 1}(\bm{m}, \bm{f}) = & S_1(1) \otimes S_2^*(0) \otimes
                         \Big(
                         \{e_0\} \oplus \{e_1\} \oplus \{e_2\}
                         \Big).
  \end{array}
                     \right.
\end{align*}
If we consider monomial bases for $K_{1, 2}(\bm{m})$ and
$K_{0, 1}(\bm{m})$, then we can represent $\delta_1(\bm{m})$ with the
transpose of the matrix that follows.  Note that, the element
$\partial 1 \in \K[\bm{\partial x}_1,\bm{\partial x}_1]$ corresponds
to the dual of $1 \in \K[{\bm{x}_1}, {\bm{x}_2}]$.

{\small
\begin{align*}
  \begin{array}{c | c c c c c c}
&
    x_{1,0} \otimes \partial 1 \otimes e_0  &
     &
    x_{1,0} \otimes \partial 1 \otimes e_1 &
     &
    x_{1,0} \otimes \partial 1 \otimes e_2 &
    \\
    &
     &
    \hspace*{-.75cm}   x_{1,1} \otimes \partial 1 \otimes e_0 \hspace*{-.75cm} &
     &
    \hspace*{-.75cm} x_{1,1} \otimes \partial 1 \otimes e_1 \hspace*{-.75cm} &
     &
    \hspace*{-.75cm} x_{1,1} \otimes \partial 1 \otimes e_2  \\ \hline
    1 \otimes {\partial x}_{2,0} \otimes (e_0 \wedge e_1) &
                                                            -b_{0,0} & -b_{1,0} & a_{0,0} & a_{1,0} & 0 & 0\\
    1 \otimes {\partial x}_{2,1} \otimes (e_0 \wedge e_1) &
                                                            -b_{0,1} & -b_{1,1} & a_{0,1} & a_{1,1} & 0 & 0\\
    1 \otimes {\partial x}_{2,0} \otimes (e_0 \wedge e_2) &
                                                            -c_{0,0} & -c_{1,0} & 0 & 0 & a_{0,0} & a_{1,0} \\
    1 \otimes {\partial x}_{2,1} \otimes (e_0 \wedge e_2) &
                                                            -c_{0,1} & -c_{1,1} & 0 & 0 & a_{0,1} & a_{1,1} \\
    1 \otimes {\partial x}_{2,0} \otimes (e_1 \wedge e_2) &
                                                            0 & 0 & -c_{0,0} & -c_{1,0} & b_{0,0} & b_{1,0} \\
    1 \otimes {\partial x}_{2,1} \otimes (e_1 \wedge e_2) &
                                                            0 & 0 & -c_{0,1} & -c_{1,1} & b_{0,1} & b_{1,1} \\  
  \end{array}
\end{align*}
}

The resultant is equal (up to sign) to the determinant of the above
matrix, it has total degree 6 (same as the size of this matrix) and 66
terms.
 \end{example}

  As we saw in the previous example, once we have fixed a basis for the map
  in \Cref{intro:def:KoszulMap}, we can represent the Koszul-type
  determinantal formula by the determinant of a matrix. We refer to
  this matrix as a \emph{Koszul resultant matrix}.

  \begin{corollary}
    [{{\cite[Prop.~5.2.4]{weyman2003cohomology}}}]
    \label{intro:thm:sizeVectorSpacesKoszulMap}
       Let $\bm{F}$ be a generic multihomogeneous system of polynomials
   with multidegrees ${\bm{d}_0},\dots,{\bm{d}_N}$, respectively.
   Let $\bm{m} \in \Z^q$ be a degree vector so that the  Weyman complex $K_\bullet(\bm{m})$
   becomes
   \begin{align*}
     K_\bullet(\bm{m}) : 0 \rightarrow K_{v,
     p+v}(\bm{m}) \otimes \Z[\bm{u}] \xrightarrow{\delta_v(\bm{m})} K_{v-1,
     p+v-1}(\bm{m}) \otimes \Z[\bm{u}] \rightarrow 0 \, .
   \end{align*}
   Then, the map $\delta_v(\bm{m})$ of \Cref{intro:def:KoszulMap}
   is linear in the coefficients of $\bm{F}$.
   Moreover, as the determinant of the
   complex is the resultant, the rank of both $K_{v}(\bm{m})$ and
   $K_{v+1}(\bm{m})$, as $\Z[\bm{u}]$-modules, equals the degree of
   the resultant (\Cref{intro:thm:degRes}).
 \end{corollary}

 We remark that Koszul-type formulas generalizes Sylvester-type
 formulas.
 
 \begin{proposition} \label{intro:thm:sylvesterMatrix}
   Under the assumptions of \Cref{intro:def:KoszulMap}, if $p = 1$ and
   $v = 0$, the map $\delta_v(\bm{m})$ acts as a Sylvester map, that
   is $(g_0,\dots,g_N) \mapsto \sum_{k = 0}^N g_k \, F_k$. In this
   case, it holds
   $$\delta_v(\bm{m})( g_0 \otimes e_0 + \dots + g_N \otimes e_N)
   =  \Big( \sum_{k = 0}^N g_k \, F_k \Big) \otimes 1.$$
 \end{proposition}

Determinantal formulas for the multiprojective resultant of unmixed
systems, that is systems where the multidegree of each polynomial is
the same, were extensively studied by several authors
\cite{sturmfels1994multigraded,weyman_determinantal_1994,chtcherba_constructing_2004,dickenstein2003multihomogeneous}.
However, there are very few results about determinantal formulas for
\textit{mixed} multihomogeneous systems, that is, when the supports
are not the same. We know such formulas for scaled multihomogeneous
systems \cite{EmiMan-mhomo-jsc-12}, that is when the supports
are scaled copies of one of them, and for bivariate tensor-product
polynomial systems
\cite{mantzaflaris2017resultants,buse_matrix_2020}.
In what follows, we use the Weyman complex to derive new formulas for
families of mixed multilinear systems.

\section{Determinantal formulas for star multilinear systems}
\label{multilinearMixed:sec:mixedUmixed}

We consider four different kinds of overdetermined multihomogeneous
systems, related to \emph{star multilinear systems}
(\Cref{def:multilinearMixed}) and we construct
determinantal formulas for each of them. These formulas are Koszul-
and Sylvester-type determinantal formulas
(\Cref{def:KoszulDetFormula}).
To simplify the presentation, we change
somewhat the notation that we used for the polynomial systems in
\cref{sec:multihomSys}.
We split the blocks of variables in two groups; we replace the blocks
of variables $\bm{x}_k$ by ${\bm{X}_i}$ or ${\bm{Y}_j}$ and the
constants $n_k$, that correspond to the cardinalities of the blocks,
by $\alpha_i$ or $\beta_j$.
Let $A, B \in \N$ and $q = A + B$.
Let $\bm{\bar{X}}$ be the set of $A$ blocks of variables
$\{\bm{X}_1, \dots, \bm{X}_A\}$.  For each $i \in [A]$,
${\bm{X}_{i}} := \{x_{i,0},\dots,x_{i,\alpha_i}\}$; so the number of
affine variables in each block $\bm{X}_i$ is $\alpha_i \in \NN$.  We also consider the
polynomial algebra
$\K[{\bm{X}_i}] = \bigoplus_{d \in \Z} S_{\!\bm{X}_i}(d)$, where
$S_{\!\bm{X}_i}(d)$ is the $\K$-vector space of polynomials of degree
$d$ in $\K[{\bm{X}_i}]$.
Similarly, $\bm{\bar{Y}}$ is the set of $B$ blocks of variables
$\{\bm{Y}_1, \dots, \bm{Y}_B\}$.
For each $j \in [B]$,
${\bm{Y}_{j}} := \{y_{j,0},\dots, y_{j,\beta_j}\}$; hence the number
of variables in each block ${\bm{Y}_{j}}$ is $\beta_j \in \NN$.
Moreover, $\K[{\bm{Y}_j}] = \bigoplus_{d \in \Z} S_{\bm{Y}_j}(d)$,
where $S_{\bm{Y}_j}(d)$ is the $\K$-vector space of polynomials of
degree $d$ in $\K[{\bm{Y}_j}]$.

Consider the $\Z^{q}$-multigraded algebra $\K[\bm{\bar{X},\bar{Y}}]$,
given by
$$
\K[\bm{\bar{X},\bar{Y}}] :=
\!\!\!\!
\!\!\!\!
\bigoplus_{
  (d_{\bm{X}_1},\dots,d_{\bm{X}_A},d_{\bm{Y}_1},\dots,d_{\bm{Y}_B}) \in \Z^{q}
}
\!\!\!\!
\!\!\!\!
S_{\bm{X}_1}(d_{\bm{X}_1}) \otimes \dots \otimes S_{\bm{X}_A}(d_{\bm{X}_A}) \otimes
S_{\bm{Y}_1}(d_{\bm{Y}_1}) \otimes \dots \otimes S_{\bm{Y}_B}(d_{\bm{Y}_B}).
$$
For a multihomogeneous polynomial $f \in \K[\bm{\bar{X},\bar{Y}}]$ of
multidegree $\bm{d} \in \Z^q$, we denote by $d_{\bm{X}_i}$,
respectively $d_{\bm{Y}_j}$, the degree of $f$ \wrt the block of
variables ${\bm{X}_i}$, respectively ${\bm{Y}_j}$.
For each group of indices $1 \leq i_1 < \dots < i_r \leq A$ and
$1 \leq j_1 < \dots < j_s \leq B$, we denote by
$\K[{\bm{X}_{i_1}},\dots,{\bm{X}_{i_r}},{\bm{Y}_{j_1}},\dots,{\bm{Y}_{j_s}}]_{\bm{1}}$
the set of multilinear polynomials in $\K[\bm{\bar{X},\bar{Y}}]$
with multidegree $(d_{\bm{X}_1},\dots,d_{\bm{X}_A},d_{\bm{Y}_1},\dots,d_{\bm{Y}_B})$, where
\begin{align*}
d_{\bm{X}_{l}} = \left\{
  \begin{array}{c l}
    1 & \text{if } l \in \{i_1,\dots,i_r\} \\
    0 & \text{otherwise}
  \end{array} \right.
  \quad \text{ and } & \quad
  d_{\bm{Y}_{l}} = \left\{
  \begin{array}{c l}
    1 & \text{if } l \in \{j_1,\dots,j_s\} \\
    0 & \text{otherwise} .
  \end{array} \right. .
\end{align*}

Let $N = \sum_{i=1}^A \alpha_i + \sum_{j=1}^B \beta_j$. We say that a
polynomial system is \emph{square} if it has $N$ equations and
\emph{overdetermined} if it has $N+1$. We consider the multiprojective space
$$\P := \Pr^{\alpha_1} \times \dots \times \Pr^{\alpha_A} \times
\Pr^{\beta_1} \times \dots \times \Pr^{\beta_B} .$$

\begin{definition}[Star multilinear systems]
  \label{def:multilinearMixed}
    A square multihomogeneous system $\bm{f} = (f_1,\dots,f_N)$ in
  $\K[\bm{\bar{X},\bar{Y}}]$ with multidegrees
  ${\bm{d}_1},\dots,{\bm{d}_N} \in \Z^{q}$, respectively, is a
  \emph{Star} \emph{multilinear system} if for every $k \in [N]$, there is
  $j_k \in [B]$ such that
  $$f_k \in \K[{\bm{X}_1},\dots,{\bm{X}_A},{\bm{Y}_{j_k}}]_{\bm{1}}.$$
  For each $j \in [B]$, we denote by $\eqs_j$ the number of
  polynomials of $\bm{f}$ in
  $\K[{\bm{X}_1},\dots,{\bm{X}_A},{\bm{Y}_j}]_{\bm{1}}.$
\end{definition}

We use the term \emph{star} because we can represent such systems
using a star graph with weighted edges.  The vertices of the
graph are the algebras $\K[{\bm{Y}_{1}}],\dots,\K[{\bm{Y}_{B}}]$, and
 $\K[{\bm{X}_{1}},\dots,{\bm{X}_{A}}]$. For each ${\bm{d}_k}$ there is
an edge between the vertices $\K[{\bm{X}_{1}},\dots,{\bm{X}_{A}}]$ and
$\K[{\bm{Y}_{j}}]$ whenever $d_{k,Y_j} = 1$.
The weight of the edge between the vertices
$\K[{\bm{X}_{1}},\dots,{\bm{X}_{A}}]$ and $\K[{\bm{Y}_{j}}]$
corresponds to $\eqs_j$.
That is, when it holds
$f_k \in \K[{\bm{X}_{1}},\dots,{\bm{X}_{A}}, {\bm{Y}_{j}}]_{\bm{1}}$.
The graph is a star because every vertex is connected to
$\K[{\bm{X}_{1}},\dots,{\bm{X}_{A}}]$ and there is no edge between two
vertices $\K[{\bm{Y}_{j_1}}]$ and $\K[{\bm{Y}_{j_2}}]$.

\begin{example}
  \label{ex:multilinearMixed}  
  Let ${\bm{X}_1},{\bm{X}_2},{\bm{Y}_1},{\bm{Y}_2},{\bm{Y}_3}$ be
  blocks of variables. Consider the multihomogeneous system
  $(f_1,f_2,f_3,f_4) \subset \K[\bm{\bar{X},\bar{Y}}]$ with
  the following (pattern of) multidegrees
  \setlength{\columnseprule}{.5pt}
  \def\columnseprulecolor{\color{black}}
  \begin{multicols}{2}
    {
    \small
  \begin{align*}
    \begin{array}{r c c | c c c c}
      (\!\!\!\! & d_{k,X_1}, & d_{k,X_2}, & d_{k,Y_1}, & d_{k,Y_2}, & d_{k,Y_3} & \!\!\!\!)\\
    {\bm{d}_1} = ( \!\!\!\!& 1, & 1,  & 1, & 0, & 0 & \!\!\!\!) \\
    {\bm{d}_2} = ( \!\!\!\!& 1, & 1,  & 1, & 0, & 0 & \!\!\!\!) \\
    {\bm{d}_3} = ( \!\!\!\!& 1, & 1,  & 0, & 1, & 0 & \!\!\!\!) \\
    {\bm{d}_4} = ( \!\!\!\!& 1, & 1,  & 0, & 0, & 1 & \!\!\!\!)
  \end{array}
  \end{align*}
  }
  \columnbreak
    \begin{center}
\begin{tikzpicture}[thick,scale=0.85, every node/.style={scale=0.85},
  fsnode/.style={draw,circle},
  ssnode/.style={draw,circle},
  every fit/.style={ellipse,draw,inner sep=-2pt,text width=2cm},
  -,shorten >= 3pt,shorten <= 3pt
]

\begin{scope}
  \node[fsnode] (x1) [label=left: ${\K[{\bm{X}_{1}},{\bm{X}_{2}}]}$] {};
\end{scope}

\begin{scope}[xshift=4cm,yshift=.75cm,start chain=going below,node distance=5mm]
\foreach \j in {1,2,3}
  \node[ssnode,on chain] (y\j) [label=right: ${\K[\bm{Y}_{\j}]}$] {};
\end{scope}


\draw (x1) -- (y1) node [near end, above,sloped] {$\eqs_{1} = 2$};
\draw (x1) -- (y2) node [near end, above,sloped] {$\eqs_{2} = 1$};
\draw (x1) -- (y3) node [near end, above,sloped] {$\eqs_{3} = 1$};
\end{tikzpicture}
\end{center}
\end{multicols}
\noindent
It is a star multilinear system
  where $\eqs_1 = 2$, $\eqs_2 = 1$, and $\eqs_3 = 1$.
    The corresponding star graph is the one above.

\end{example}

\begin{remark} \label{intro:thm:sumOfVariablesStar}
  For each square star multilinear system, it holds
  $N = \sum_{j=1}^B \eqs_{j}$.
  Moreover, if the system has a nonzero finite number of solutions,
  then for each $j \in \{1,\dots,B\}$ it holds
  $\eqs_{j} \geq \beta_j$, see \Cref{intro:thm:bezoutBound}.
\end{remark}

\subsection{Determinantal formulas}
\label{sec:det-formulas}

 To be able to define the resultant, we study overdetermined
  polynomial systems $(f_0,f_1,\dots,f_N)$ in
  $\K[\bm{\bar{X},\bar{Y}}]$ where $(f_1,\dots,f_N)$ is a square star
  multilinear system and $f_0$ is a multilinear polynomial.
The obvious choice for $f_0$ is to have the same structure as one of
the polynomials $f_1,\dots,f_N$; still we also choose $f_0$ to have a
different support. This leads to resultants of smaller degrees and so
to matrices of smaller size.
Aiming at resultant formulas for \emph{any} $A$ and $B$,
we were able to identify the following 
choices of $f_0$, that lead to a determinantal Weyman
complex\footnote{
    For other choices of $f_0$ we found specific
    values of $A$ and $B$ for which every possible Weyman complex is
    not determinantal.
  }.
In particular, the  following $f_0$ lead to determinantal
formulas:  
\smallskip
\begin{itemize}
\item {\bf Center-Vertex case:} $f_0 \in \K[\bm{X}_1,\dots,\bm{X}_A]_{\bm{1}}$.
\item {\bf Outer-Vertex case:} $f_0 \in \K[\bm{Y}_j]_{1}$, for any  $j \in [B]$.
\item {\bf Edge case:} $f_0 \in \K[\bm{X}_1,\dots,\bm{X}_A,\bm{Y}_j]_{\bm{1}}$,
  for any $j \in [B]$,
\item {\bf Triangle case:}
$f_0 \in \K[\bm{X}_1,\dots,\bm{X}_A,\bm{Y}_{j_1},\bm{Y}_{j_2}]_{\bm{1}}$,
  for any  $j_1,j_2 \in [B]$,  $j_1 \neq j_2$.
\end{itemize}

\smallskip

We can view the various multidegrees of $f_0$,
${\bm{d}_0} = (d_{0,{X}_1},\dots,d_{0,{X}_A}, d_{0,{Y_1}}, \dots,
d_{0,{Y_B}})$, in the cases above 
as solutions of the following system of inequalities:
\begin{align} \label{multilinearMixed:eq:casesForD0inMixedUnmixed}
    \begin{cases}
    (\forall\, 1 \leq i \leq A) \, & 0 \leq d_{0,\bm{X}_i} \leq 1, \\
    (\forall\, 1 \leq j \leq B) \, & 0 \leq d_{0,\bm{Y}_j} \leq 1, \\
    (\forall\, 1 \leq i_1 < i_2 \leq A) & d_{0,\bm{X}_{i_1}} =   d_{0,\bm{X}_{i_2}} \text{, and } \\
    & \sum_{j = 1}^B d_{0,\bm{Y}_j} \leq 1 + d_{0,\bm{X}_1}.
   \end{cases}
\end{align}

 Consider the set $\{0, \dots, N\}$ that corresponds to generic
 polynomials $\bm{F}=(F_0, \dots, F_N)$
 (\Cref{intro:def:genMultiHom}). As many of the polynomials have the
 same support, we can gather them to simplify the cohomologies of
 \eqref{eq:Kvp}. We need the  following notation.
 For each tuple $(s_0,\dots,s_B) \in \N^{B+1}$, let
 $\mathcal{I}_{s_0,s_1,\dots,s_B}$ be the set of all the subsets of
 $\{0,\dots,N\}$, such that
 \begin{itemize}
 \item For $1 \leq j \leq B$, the index $s_j$ indicates that we
   consider exactly $s_j$ polynomials from $(F_1,\dots,F_N)$ that belong to
   $\Z[\bm{u}][\bm{X}_1,\dots,\bm{X}_A,\bm{Y}_j]_{\bm{1}}$.
   
 \item In addition, if $s_0 = 1$, then 0 belongs to all the sets in
   $\mathcal{I}_{s_0,s_1,\dots,s_B}$.
 \end{itemize}
 That is,
 \begin{align}
   \label{multilinearMixed:eq:definitionSetOfIndices}
   \nonumber
   \mathcal{I}_{s_0,s_1,\dots,s_B} := \Big\{ I
     : \; & I \subset \{0,\dots,N\}, \left( 0 \in I \!\Leftrightarrow\!
       s_0 = 1 \right) \text{ and } \\ & (\forall 1 \leq j \leq B) \;
     s_j = \#\{k \in I \setminus \{0\}: F_k \in
     \Z[\bm{u}][\bm{X}_1,\dots,\bm{X}_A,\bm{Y}_j]_{\bm{1}}\} \Big\}. 
   \end{align}

   \begin{example}
     If we consider a system $(F_1,\dots,F_4)$ as in 
     \Cref{ex:multilinearMixed} and introduce some $F_0$, it holds for $\bm{F} = (F_0,\dots,F_4)$ that
     $\mathcal{I}_{1,1,1,0} = \{\{0,1,3\}, \{0,2,3\}\}$ and
     $\mathcal{I}_{0,2,0,1} = \{\{1,2,4\}\}$.
   \end{example}
   
   Notice that if $I, J \in \mathcal{I}_{s_0,s_1,\dots,s_B}$, then $I$
   and $J$ have the same cardinality and
   $\sum_{k \in I} {\bm{d}_k} = \sum_{k \in J} {\bm{d}_k}$, as they
   correspond to subsets of polynomials of $\bm{F}$ with the same
   multidegrees.

   The following lemma uses the sets
   $\mathcal{I}_{s_0,s_1,\dots,s_B}$ to simplify the cohomologies of
   \eqref{eq:Kvp}.

 \begin{lemma} \label{multilinearMixed:thm:rewriteKvp}
   Consider a generic overdetermined system $\bm{F} \!  = \! (F_0,\dots,F_N)$
   in $\Z[\bm{u}][\bm{\bar{X},\bar{Y}}]$ of multidegrees
   ${\bm{d}_0},\dots,{\bm{d}_N}$ (\Cref{intro:def:genMultiHom}), where
   $(F_1,\dots,F_n)$ is a square star multilinear system such that,
   for every $j \in \{1,\dots,B\}$, $\eqs_j \geq \beta_j$, and
   ${\bm{d}_0}$ is the multidegree of $F_0$.
   Following \eqref{eq:Kvp}, we can rewrite the
   modules of the Weyman complex 
   $
   K_v(\bm{m}) = \bigoplus_{p=0}^{N+1} K_{v,p} \otimes \Z[\bm{u}]
   $  
   in the more detailed form
   \begin{align}
     \label{multilinearMixed:eq:rewriteKvp}
     K_{v,p}(\bm{m}) \cong
     \!\!\!\!\!\!\!\!\!
     \bigoplus_{\substack{0 \leq s_0 \leq 1 \\ 0 \leq s_1 \leq \eqs_1 \\ \dots \\ 0 \leq s_B \leq \eqs_B \\ s_0 + s_1 + \dots + s_B = p}}
     \!\!\!\!\!\!\!\!\!\!\!\!
 H^{p-v}_{\P} \Big(
     \bm{m} - 
     ( \sum_{j=1}^{B} s_j,\dots,\sum_{j=1}^{B} s_j, s_1, \dots, s_B ) - s_0 \, {\bm{d}_{0}}
     \Big)
     \otimes
   \!\!\!
     \bigoplus_{I \in \mathcal{I}_{s_0,s_1,\dots,s_B}}  \bigwedge_{k \in I} e_k .
     \nonumber
     \\[-30pt]
     \\\nonumber
   \end{align}
   Moreover, the following isomorphisms hold for the cohomologies:
   \begin{multline}
     \label{multilinearMixed:eq:rewriteHpvGeneral}
       H^{p-v}_{\P} \Big(
     \bm{m} - 
     ( \sum_{j=1}^{B} s_j,\dots,\sum_{j=1}^{B} s_j, s_1, \dots, s_B) - s_0 \, {\bm{d}_{0}}
     \Big) \cong \\      
     \bigoplus_{\substack{
         \\ \\
       r_{{\bm{X}_1}} \dots r_{{\bm{X}_1}},r_{\bm{Y}_1} \dots r_{\bm{Y}_B} \in \N \\
       \sum_i r_{{\bm{X}_i}}  + \sum_j r_{\bm{Y}_j} = p-v
     }} \!\!\!
     \Big(
         \bigotimes_{i = 1}^A
       H^{r_{{\bm{X}_i}}}_{\Pr^{\alpha_{i}}}\Big(m_{\bm{X}_i} - \sum_{j=1}^{B} s_j - s_0 \, d_{0,\bm{X}_i}\Big)
         \displaystyle \otimes
         \bigotimes_{j = 1}^B  H^{r_{\bm{Y}_j}}_{\Pr^{\beta_j}}(m_{{\bm{Y}_j}} - s_j - s_0 \, d_{0,{\bm{Y}_j}})
       \Big) .
     \end{multline}
 \end{lemma} 
 
 \begin{proof}
   Consider $I,J \subset \mathcal{I}_{s_0,s_1,\dots,s_B}$. Then,
   by definition, $\#I = \#J$ and
   $\sum\limits_{k \in I} {\bm{d}_k} = \sum\limits_{k \in J} {\bm{d}_k} =
   $ \linebreak $
   (\sum_{j=1}^{B} s_j,\dots,\sum_{j=1}^{B} s_j, s_1, \dots, s_B) +
   s_0 \, {\bm{d}_{0}}$. Hence,
   \begin{multline*}
   \Big(
     H^{p-v}_{\P}(\bm{m} - \sum_{k \in I} {\bm{d}_k}) \otimes \bigwedge_{k \in I} e_k
   \Big)
   \oplus
   \Big(
     H^{p-v}_{\P}(\bm{m} - \sum_{k \in J} {\bm{d}_k}) \otimes \bigwedge_{k \in J} e_k
     \Big)
   \cong \\
    H^{p-v}_{\P} \Big(
     \bm{m} - 
     ( \sum_{j=1}^{B} s_j,\dots,\sum_{j=1}^{B} s_j, s_1, \dots, s_B ) - s_0 \, {\bm{d}_{0}}
     \Big)
     \otimes
   \Big(\bigwedge_{k \in I} e_k \oplus \bigwedge_{k \in J} e_k \Big) .
   \end{multline*}  
   By definition of $\eqs_1,\dots,\eqs_B$
   (\cref{def:multilinearMixed}), the set
   $\mathcal{I}_{s_0,s_1,\dots,s_B}$ is not empty if and only if
   $0 \leq s_0 \leq 1$ and
   for all $i \in \{1,\dots,B\}$ it holds  $0 \leq s_i \leq \eqs_i$.
   Hence,
   $$ \{I : I \subset \{0,\dots,N\}, \#I = p\} = \bigcup\limits_{\substack{0
       \leq s_0 \leq 1 \\ 0 \leq s_1 \leq \eqs_1 \\ \dots \\ 0 \leq s_B
       \leq \eqs_B \\ s_0 + s_1 + \dots + s_B = p}} \mathcal{I}_{s_0,s_1,\dots,s_B}.$$
   Thus, \eqref{multilinearMixed:eq:rewriteKvp} holds.
   The isomorphism in \eqref{multilinearMixed:eq:rewriteHpvGeneral} follows
   from \Cref{intro:thm:kunnethFormula}.
   \end{proof}

   In what follows, we identify the degree vectors that reduce the
   Weyman complex to have just two elements and, in this way, they
   provide us Koszul-type determinantal formulas for star multilinear
   systems (\Cref{def:KoszulDetFormula}).
   These degree vectors are associated to tuples called
   \emph{determinantal data}. The determinantal data parameterize the
   different Koszul-type determinantal formulas that we can obtain
   using the Weyman complex.

\begin{definition} \label{multilinearMixed:def:admissibleTriplet}
  Consider a partition of $\{1,\dots,B\}$ consisting of two sets $P$
  and $D$ and a constant $c \in \N$. We say that the triplet $(P,D,c)$
  is \emph{determinantal data} in the following cases:
  \begin{itemize}
  \item When $f_0$ corresponds to \textbf{Center-Vertex or Edge case}: if holds, $0 \leq c \leq A$.
  \item When $f_0$ corresponds to \textbf{Outer-Vertex case}: if the following holds,
    $$
    \begin{cases}
      c = 0 & \text{if } \sum_{j \in P} d_{0,\bm{Y}_j} = 0, \text{ or} \\
      c = A & \text{if }  \sum_{j \in D} d_{0,\bm{Y}_j} = 0.
    \end{cases}$$
  \item When $f_0$ corresponds to \textbf{Triangle case}: if the following holds,
    $$ \left\{
    \begin{array}{l}
      0 \leq c \leq A, \\
      \sum_{j \in P} d_{0,\bm{Y}_j} \leq 1, \text{ and} \\
      \sum_{j \in D} d_{0,\bm{Y}_j} \leq 1 .
    \end{array}
  \right.
  $$
    
\end{itemize}

Equivalently, we say that the triplet $(P,D,c)$ is \emph{determinantal
  data} for the multidegree ${\bm{d}_0}$ if the following conditions
are satisfied:
   \begin{align}\label{multilinearMixed:eq:partition,eq:valuesForC}
     \left\{
     \begin{array}{c l}
       \sum_{j \in P} d_{0,\bm{Y}_j} \leq 1 \\
       \sum_{j \in D} d_{0,\bm{Y}_j} \leq 1 \\
       0 \leq c \leq A  & \text{ when } (\forall\, i \in [A]) \, \text{ it holds } d_{0,\bm{X}_i} = 1, \\
       c = 0  & \text{ when } (\forall\, i \in [A]) \, \text{ it holds } d_{0,\bm{X}_i} = 0 \text{ and } \sum_{j \in P} d_{0,\bm{Y}_j} = 0, \\
       c = A & \text{ when } (\forall\, i \in [A]) \, \text{ it holds } d_{0,\bm{X}_i} = 0 \text{ and } \sum_{j \in D} d_{0,\bm{Y}_j} = 0.
     \end{array}
     \right.
   \end{align}  
 \end{definition}
 
 \begin{theorem} \label{multilinearMixed:thm:caseUnmixedMixed}
   \!\!\!\!
   Consider a generic overdetermined system $\bm{F} \! = \! (F_0,\dots,F_N)$ \!\!
   in \! $\Z[\bm{u}][\text{\small $\bm{\bar{X},\bar{Y}}$}]$ of multidegrees
   ${\bm{d}_0},\dots,{\bm{d}_N}$ (\Cref{intro:def:genMultiHom}), where
   $(F_1,\dots,F_n)$ is a square star multilinear system. Assume that
   for every $j \in \{1,\dots,B\}$ it holds $\eqs_j \geq \beta_j$
   (see \Cref{intro:thm:sumOfVariablesStar}) and that the multidegree of $F_0$, ${\bm{d}_0}$, is a
   solution of the system in
   \eqref{multilinearMixed:eq:casesForD0inMixedUnmixed}.
   Then, for each {determinantal data} $(P,D,c)$
   (\Cref{multilinearMixed:def:admissibleTriplet})  and a
   permutation $\sigma : \{1,\dots,A\} \rightarrow \{1,\dots,A\}$,
   the degree vector
   $\bm{m} =
   (m_{\bm{X}_1},\dots,m_{\bm{X}_A},m_{\bm{Y}_1},\dots,m_{\bm{Y}_B})$,
   defined by
   \begin{align*}
          \left\{
     \begin{array}{lcl}
             m_{\bm{X}_{i}} = \sum_{j \in D} \beta_{j} + \sum_{k = 1}^{\sigma(i)-1} \alpha_{\sigma^{-1}(k)} + d_{0,\bm{X}_i} && \text{for } 1 \leq i \leq A  \text{ and } \sigma(i) > c \\            
             m_{\bm{X}_{i}} = \sum_{j \in D} \beta_{j} + \sum_{k = 1}^{\sigma(i)-1} \alpha_{\sigma^{-1}(k)} - 1 && \text{for } 1 \leq i \leq A  \text{ and } \sigma(i) \leq c \\            
             m_{\bm{Y}_j} = \eqs_j - \beta_j + d_{0,{\bm{Y}_j}} && \text{for } j \in P \\
             m_{\bm{Y}_j} = -1 && \text{for } j \in D
            \end{array}\right.
   \end{align*}
   corresponds to the Koszul-type determinantal
   formula (\Cref{def:KoszulDetFormula})
      $$K_\bullet( \bm{m}) : 0 \rightarrow
   K_{1,\omega+1}(\bm{m}) \otimes \Z[\bm{u}]
   \xrightarrow{\delta_1(\bm{m})} K_{0,\omega}(\bm{m}) \otimes
   \Z[\bm{u}] \rightarrow 0,$$ where
   $\omega = \sum_{k = 1}^c \alpha_{\sigma^{-1}(k)} + \sum_{j \in D}
   \beta_j$.
 \end{theorem}
 
 \begin{proof}
   To simplify the presentation of the proof, we assume with no loss
   of generality that $\sigma$ is the identity map.
   We rewrite \eqref{eq:Kvp} using
   \Cref{multilinearMixed:thm:rewriteKvp}. Hence, we obtain the
   following isomorphism,
   \begin{align}
     \label{multilinearMixed:eq:rewriteHpv}
     \small
     H^{p-v}_{\P} \Big(
     \bm{m} - 
     ( \sum_{j=1}^{B} s_j,\dots,\sum_{j=1}^{B} s_j, s_1, \dots, s_B) - s_0 \, {\bm{d}_{0}}
   \Big) \cong       \nonumber  \\
   \bigoplus_{\substack{
       \sum_i r_{{\bm{X}_i}}  + \sum_j r_{\bm{Y}_j} = p-v
     }} 
      \left( \arraycolsep=1pt
       \begin{array}{c l}
         \displaystyle
       \bigotimes_{j \in P}  H^{r_{\bm{Y}_j}}_{\Pr^{\beta_j}}(\eqs_j - \beta_j + d_{0,{\bm{Y}_j}} - s_j - s_0 \, d_{0,{\bm{Y}_j}})  \;\otimes 
       & \text{\footnotesize [Case Y.1]} 
       \\  \displaystyle
        \bigotimes_{j \in D}
         H^{r_{\bm{Y}_j}}_{\Pr^{\beta_j}}(-1 - s_j  - s_0 \, d_{0,{\bm{Y}_j}})
         \; \otimes
       & \text{\footnotesize [Case Y.2]} 
         \\ 
         \displaystyle
                  \bigotimes_{i = 1}^c
       H^{r_{{\bm{X}_i}}}_{\Pr^{\alpha_{i}}} \Big(\sum_{j \in D} \beta_j + \sum_{k=1}^{i-1} \alpha_{k} - 1 - \sum_{j=1}^{B} s_j - s_0 \, d_{0,\bm{X}_i}\Big) \; \otimes & \text{\footnotesize [Case X.1]} \\\displaystyle
                  \bigotimes_{i = c+1}^A \!\!
       H^{r_{{\bm{X}_i}}}_{\Pr^{\alpha_{i}}} \Big(\sum_{j \in D} \beta_j + \sum_{k=1}^{i-1} \alpha_{k} + d_{0,\bm{X}_i} - \sum_{j=1}^{B} s_j - s_0 \, d_{0,{\bm{X}_i}} \!\Big)
       & \text{\footnotesize [Case X.2]} 
     \end{array}
       \right)
   \end{align}

   \noindent
   We will study the values for
   $p,v,s_0,\dots,s_B,r_{{\bm{X}_1}},\dots,r_{{\bm{X}_A}},r_{\bm{Y}_1},\dots,r_{\bm{Y}_B}$
   such that $K_{v,p}(\bm{m})$ \eqref{multilinearMixed:eq:rewriteKvp} does not
   vanish. Clearly, if $0 \leq s_0 \leq 1$ and
   $(\forall i \in \{1,\dots,B\}) \; 0 \leq s_i \leq \eqs_i$, then the
   module
   $\bigoplus_{I \in \mathcal{I}_{s_0,s_1,\dots,s_B}} \bigwedge_{k \in
     I} e_k$ is not zero. Hence, assuming $0 \leq s_0 \leq 1$ and
   $(\forall i \in \{1,\dots,B\}) \; 0 \leq s_i \leq \eqs_i$, we study
   the vanishing of the modules in
   \eqref{multilinearMixed:eq:rewriteHpv}.
   By \Cref{intro:thm:bottFormula}, the modules in the right hand side
   of \eqref{multilinearMixed:eq:rewriteHpv}
   are not zero only when, for $1 \leq i \leq A$,
   $r_{\bm{X}_i} \in \{0,\alpha_i\}$ and, for $1 \leq j \leq B$,
   $r_{\bm{Y}_j} \in \{0,\beta_j\}$.
   We can use \Cref{intro:thm:possiblesRi} to show that if
   \eqref{multilinearMixed:eq:rewriteHpv} does not vanish, then
   we have the following cases
     \begin{align} \label{multilinearMixed:eq:valuesForR}
       \hfill
     \begin{array}{| c | l | c| }
       \hline
       \text{[Case Y.1]}  &
                            \text{For } j \in P &
              \begin{array}{c}
                r_{\bm{Y}_j} = 0 \text{ and }
                                 \eqs_j \!-\! \beta_j \!+\! d_{0,\bm{Y_j}} \geq s_j \!+\! s_0 \, d_{0,\bm{Y_j}}
              \end{array}
       \\[10px] \hline
       \text{[Case Y.2]}  &
                            \text{For } j \in D &
              \begin{array}{c}
              r_{\bm{Y}_j} = \beta_j \text{ and }
                s_j + s_0 \, d_{0,\bm{Y_j}} \geq \beta_j
              \end{array}
       \\[10px] \hline
       \text{[Case X.1]}  &
                            \text{For } 1 \leq i \leq c &
              \begin{array}{c}
                \displaystyle
                r_{\bm{X}_{i}} = \alpha_{i} \text{ and }
                \displaystyle
                \sum_{j=1}^{B} s_j + s_0 \,
                d_{0,\bm{X}_{i}} \geq \sum_{j \in D} \beta_j +
                \sum_{k = 1}^{i} \alpha_{\bm{X}_{k}}
              \end{array}
       \\[20px] \hline
       \text{[Case X.2]} &
                           \text{For } c < i \leq A &
              \begin{array}{c}
                r_{\bm{X}_{i}} = 0 \text{ and } 
                \displaystyle
                \sum_{j \in D} \beta_{j} +
                \sum_{k = 1}^{i-1} \alpha_{k} 
                \geq \sum_{j=1}^{B} s_j + (s_0 - 1) \, \,
                d_{0,\bm{X}_{i}}
              \end{array}
       \\[20px] \hline
     \end{array}
       \hspace*{-1px}
   \end{align}

 From \eqref{multilinearMixed:eq:valuesForR}, we can deduce the
   possible values for $v$ such that $K_{v,p}(\bm{m})$ does not
   vanish. From \eqref{multilinearMixed:eq:rewriteKvp}, it holds
   $p = \sum_{j=1}^{B} s_j + s_0$. 
   By \Cref{intro:thm:kunnethFormula},
   $p - v = \sum_{i = 1}^A r_{\bm{X}_i} + \sum_{j = 1}^B r_{\bm{Y_j}}$.
   Hence, we deduce that
     $$v = \sum_{j=1}^{B} s_j + s_0 -
     \sum_{j \in D} \beta_j - \sum_{i = 1}^c \alpha_{i}.$$
     \begin{itemize}[leftmargin=*]
     \item First we provide a lower bound for $v$. 
     Assume that $c > 0$. By [Case X.1], if
     $i = c$, then
     $$\sum_{j=1}^{B} s_j + s_0 \,
     d_{0,\bm{X}_{c}} \geq \sum_{j \in D} \beta_j +
     \sum_{k = 1}^{c} \alpha_{\bm{X}_{k}}.$$
     
     Hence, as $0 \leq s_0 , d_{0,\bm{X}_{c+1}} \leq 1$, we conclude that $v \geq 0$ as
     $$v = s_0  + \sum_{j=1}^{B} s_j -
     \sum_{j \in D} \beta_j - \sum_{k = 1}^c \alpha_{k}
     \geq s_0 \, (1-d_{0,\bm{X}_{c}})
     \geq 0.
     $$

     Assume instead that $c = 0$. Then
     $v = \sum_{j=1}^{B} s_j + s_0 - \sum_{j \in D} \beta_j$.
     By [Case Y.2], for each $j \in D$,
     $\beta_j \leq s_j + s_0 \, d_{0,\bm{Y_j}}$. Moreover, it holds,
     for each $j \in P$, $0 \leq s_j$. Adding the inequalities we
     deduce that,
     \begin{align}
       \label{multilinearMixed:eq:sumIneqsD}
       \sum_{j = 1}^B
       s_j + s_0 \, \sum_{j \in D} d_{0,\bm{Y_j}} \geq \sum_{j \in D}
       \beta_j.
     \end{align}
     By definition, $0 \leq \sum_{j \in D} d_{0,\bm{Y_j}} \leq
     1$. Hence, by \eqref{multilinearMixed:eq:sumIneqsD}, $v \geq 0$ as,
     $ v \geq s_0 - s_0 \sum_{j \in D} d_{0,\bm{Y_j}} \geq 0. $
     
   \item Finally we provide an upper bound for $v$.
     Assume that $c < A$. 
     By [Case X.2], if we consider $i = c + 1$, then
     $$\sum_{j \in D} \beta_{j} + \sum_{k = 1}^{c} \alpha_{k}
     \geq \sum_{j=1}^{B} s_j + (s_0 - 1) \, \,
     d_{0,\bm{X}_{c+1}}.$$
     Hence we conclude that $v \leq 1$, as $0 \leq s_0,d_{0,\bm{X}_{c+1}} \leq 1$ and so,
     $$v = s_0 + \sum_{j=1}^{B} s_j  -
     \sum_{j \in D} \beta_j - \sum_{k = 1}^c \alpha_{k}
     \leq s_0 + (1-s_0) \, d_{0,\bm{X}_{c+1}} \leq 1.
     $$

     Assume instead that $c = A$. Then
     $v = \sum_{j=1}^{B} s_j + s_0 - \sum_{j \in D} \beta_j -
     \sum_{i=1}^A \alpha_i$.
     By [Case Y.1], for $j \in P$,
     $\eqs_j - \beta_j + d_{0,\bm{Y_j}} \geq s_j + s_0 \,
     d_{0,\bm{Y_j}}$. Moreover, it holds, for each $j \in D$,
     $\eqs_j \geq s_j$.  As the system $(F_1,\dots,F_N)$ is square, it
     holds
     $\sum_{j = 1}^B \eqs_j = \sum_{i = 1}^A \alpha_i + \sum_{j \in P}
     \beta_j + \sum_{j \in D} \beta_j$. Hence, adding the inequalities
     we obtain
     \begin{align} \label{multilinearMixed:eq:sumIneqsP}
       \sum_{i = 1}^A \alpha_i + \sum_{j \in D} \beta_j
       =
       \sum_{j = 1}^B \eqs_j - \sum_{j \in P} \beta_j
       \geq \sum_{j = 1}^B s_j +
       (s_0 - 1)  \sum_{j \in P} d_{0,\bm{Y_j}}.
     \end{align}
     By definition,
     $0 \leq s_0 , \sum_{j \in P} d_{0,\bm{Y_j}} \leq 1$. Hence, by
     \eqref{multilinearMixed:eq:sumIneqsP}, $v \leq 1$ as,
     $$v \leq s_0 - (s_0 - 1) \,  \sum_{j \in P} d_{0,\bm{Y_j}} \leq 1.$$
   \end{itemize}
   
   We conclude that the possible values for $v$, $p$,
   $r_{\bm{X}_1},\dots,r_{\bm{X}_A}$,
   $r_{\bm{Y}_1},\dots,r_{\bm{Y}_B}$ such that
   \eqref{multilinearMixed:eq:rewriteHpv} is not zero are
   $v \in \{0, 1\}$, the possible values for $r_{\bm{X}_1},\dots,r_{\bm{X}_A}$,
   $r_{\bm{Y}_1},\dots,r_{\bm{Y}_B}$ are the ones in
   \eqref{multilinearMixed:eq:valuesForR} and
   $p = \sum_{k = 1}^c \alpha_{k} + \sum_{j \in D} \beta_j
   + v$.  Let
   $\omega = \sum_{k = 1}^c \alpha_{k} + \sum_{j \in D}
   \beta_j$. Hence, our Weyman complex looks like
   \eqref{eq:determinantalMap}, where
   $$\delta_1(\bm{m}) : K_{1,\omega+1}(\bm{m}) \otimes \Z[\bm{u}]
   \to K_{0,\omega}(\bm{m}) \otimes \Z[\bm{u}].$$
   In what follows, we prove each case in
   \eqref{multilinearMixed:eq:valuesForR}. Consider the modules related
   to the variables $\bm{Y_j}$, for $j \in \{1,\dots,B\}$.
   
\textbf{Case (Y.1)}
We consider the modules that involve the variables in the
block $\bm{Y_j}$, for $j \in P$.
As $s_j \leq \eqs_j$ and $s_0, d_{0,\bm{Y_j}} \leq 1$, it holds
$\eqs_j - \beta_j + d_{0,\bm{Y_j}} - s_j - s_0 \, d_{0,\bm{Y_j}} > - \beta_j - 1$. Hence, by 
\Cref{intro:thm:possiblesRi},
\begin{align}
  \label{multilinearMixed:eq:ineqsYP} 
  H^{r_{\bm{Y}_j}}_{\Pr^{\beta_j}}(\eqs_j  -  \beta_j  +  d_{0,\bm{Y_j}}
   - s_j  -  s_0 \, d_{0,\bm{Y_j}})  \neq  0
  \iff \\  \nonumber
    r_{\bm{Y}_j}  = 0  \text{ and }
    \eqs_j  -  \beta_j  + 
    d_{0,\bm{Y_j}} \geq s_j  +  s_0 \, d_{0,\bm{Y_j}}.
\end{align}

\textbf{Case (Y.2)}
We consider the modules that involve the variables in the
block $\bm{Y_j}$, for $j \in D$.
As $s_j, s_0, d_{0,\bm{Y_j}} \geq 0$, then
$-1 - s_j - s_0 \, d_{0,\bm{Y_j}} < 0$. Hence, by 
\Cref{intro:thm:possiblesRi},
\begin{align}  \label{multilinearMixed:eq:ineqsYD}
    H^{r_{\bm{Y}_j}}_{\Pr^{\beta_j}}(-1 - s_j  - s_0 \, d_{0,\bm{Y_j}}) \neq 0 \iff  
    r_{\bm{Y}_j} = \beta_j \quad \text{ and } \quad s_j + s_0 \, d_{0,\bm{Y_j}} \geq \beta_j  .
\end{align}

Now we consider the cohomologies related to the blocks of variables
$\bm{X_i}$, for $i \in \{1,\dots,A\}$. We assume that the cohomologies
related to the blocks of variables $\bm{Y_j}$ do not vanish.
 
\textbf{Case (X.1)} We consider the module related to the blocks
$\bm{X}_{1}\dots,\bm{X}_{c}$. We only need to
consider this case if $c > 0$, so we assume $c > 0$.
We prove that for each $1 \leq i \leq c$,
if the cohomologies related to the variables in the blocks $\bm{Y}_j$,
for $1 \leq j \leq B$, and the ones related to
$\bm{X}_{1}\dots,\bm{X}_{i-1}$, do not
vanish, then
\begin{align}
  \label{multilinearMixed:eq:ineqsXD}
  \begin{gathered}        
    H^{r_{\bm{X}_{k}}}_{\Pr^{\alpha_{i}}}
    \Big( m_{\bm{X}_{i}} - \sum_{j=1}^{B} s_j -
    s_0 \, d_{0,\bm{X}_{i}} \Big) \neq 0
    \iff \\ 
    \quad 
       r_{\bm{X}_{i}} = \alpha_{i} \text{
         and }  \sum_{j=1}^{B} s_j + s_0 \,
       d_{0,\bm{X}_{i}} \geq \sum_{j \in D} \beta_j +
       \sum_{k = 1}^{i} \alpha_{\bm{X}_{k}}.
     \end{gathered}
\end{align}

We proceed by induction on $1  \leq i \leq c$.

\begin{itemize}[leftmargin=*]
\item Consider $i = 1$ and the cohomology related to the block
  $\bm{X}_{1}$,
  $$
  H^{r_{\bm{X}_{1}}}_{\Pr^{\alpha_{1}}}
  \Big( \sum_{j \in D} \beta_{j} - 1 - \sum_{j=1}^{B} s_j - s_0 \,
  d_{0,\bm{X}_{1}} \Big).
  $$
  As we assumed that $c > 0$ and the triplet $(P,D,c)$ is
  determinantal data (\cref{multilinearMixed:def:admissibleTriplet}),
  by definition either $d_{0,\bm{X}_{1}} = 1$ or both
  $d_{0,\bm{X}_{1}} = 0$ and
  $\sum_{j \in D} d_{0,\bm{Y_j}} = 0$. Also, it holds
  $0 \leq s_0, \sum_{j \in D} d_{0,\bm{Y_j}} \leq 1$. Hence, from
  \eqref{multilinearMixed:eq:sumIneqsD}, we conclude that,
  $$
  \sum_{j \in D} \beta_{j} - 1 - \sum_{j=1}^{B} s_j - s_0 \,
  d_{0,\bm{X}_{k}}
  \leq
  s_0 \, \sum_{j \in D} d_{0,\bm{Y_j}} - 1 -  s_0 \,
  d_{0,\bm{X}_{k}}
  < 0
  $$
  Therefore,  by 
  \Cref{intro:thm:possiblesRi},
  \begin{align*}
    \begin{gathered}
      H^{r_{\bm{X}_{1}}}_{\Pr^{\alpha_{1}}}
      \Big( \sum_{j \in D} \beta_{j} - 1 - \sum_{j=1}^{B} s_j - s_0
      \, d_{0,\bm{X}_{1}} \Big) \neq 0
      \iff \\
      r_{\bm{X}_{1}} = \alpha_{1} \quad \text{
        and } \quad
      \sum_{j=1}^{B} s_j + s_0 \, d_{0,\bm{X}_{1}}
      \geq  \sum_{j \in D} \beta_j + \alpha_{\bm{X}_{1}}
    \end{gathered}
  \end{align*}

\item We proceed by induction, assuming that \eqref{multilinearMixed:eq:ineqsXD} holds for
  $i-1$, we prove the property for $i$. We consider the cohomology
  $$H^{r_{\bm{X}_{i}}}_{\Pr^{\alpha_{i}}}
  \Big( \sum_{j \in D} \beta_{j} - 1 + \sum_{k = 1}^{i-1}
  \alpha_{k} - \sum_{j=1}^{B} s_j - s_0 \,
  d_{0,\bm{X}_{i}} \Big).$$
  
  By definition
  (see \eqref{multilinearMixed:eq:casesForD0inMixedUnmixed}),
  $d_{0,\bm{X}_{i-1}} = d_{0,\bm{X}_{i}}$,
  and by inductive hypothesis, if the previous modules do not
  vanish, then
  $$
  \sum_{j=1}^{B} s_j + s_0 \, d_{0,\bm{X}_{i}}  =
  \sum_{j=1}^{B} s_j + s_0 \, d_{0,\bm{X}_{i-1}} \geq
  \sum_{j \in D} \beta_j + \sum_{k = 1}^{i-1}
  \alpha_{\bm{X}_{k}}.$$
  Hence,  by 
  \Cref{intro:thm:possiblesRi},
  \begin{align*}
    \begin{gathered}
      H^{r_{\bm{X}_{i}}}_{\Pr^{\alpha_{i}}} \Big(
      \sum_{j \in D} \beta_{j} - 1 + \sum_{k = 1}^{i-1}
      \alpha_{k} - \sum_{j=1}^{B} s_j - s_0 \,
      d_{0,\bm{X}_{i}} \Big) \neq 0 \iff \\
    r_{\bm{X}_{i}} = \alpha_{i}
    \quad \text{and} \quad
    \sum_{j=1}^{B}
    s_j + s_0 \, d_{0,\bm{X}_{i}} \geq \sum_{j \in D}
    \beta_{j} + \sum_{k = 1}^{i-1} \alpha_{k} + \alpha_{i}.
  \end{gathered}
  \end{align*}
\end{itemize}

\textbf{Case (X.2)}
  We consider the module related to the blocks
  $\bm{X}_{c + 1}\dots,\bm{X}_{A}$. We only need
  to consider this case if $c < A$, so we assume $c < A$.
  We prove that for each $c < i \leq A$, if the cohomologies related to
  the variables in the blocks  $\bm{Y}_j$, for $1 \leq j \leq B$, and
  the ones related to
  $\bm{X}_{i+1},\bm{X}_{i+2},\dots,\bm{X}_{A}$,
  do not vanish, then
       \begin{align}
       \label{multilinearMixed:eq:ineqsXP}
       \begin{gathered}
         H^{r_{\bm{X}_{i}}}_{\Pr^{\alpha_{i}}}
       \Big( \sum_{j \in D} \beta_{j} +
       \sum_{k = 1}^{i-1} \alpha_{k} + (1 - s_0) \, d_{0,\bm{X}_{i}}
       - \sum_{j=1}^{B} s_j  \Big) \neq 0 \iff \\ 
       r_{\bm{X}_{i}} = 0 \quad \text{and} \quad \sum_{j \in D} \beta_{j} +
       \sum_{k = 1}^{i-1} \alpha_{k} 
       \geq \sum_{j=1}^{B} s_j + (s_0 - 1) \, \,
       d_{0,\bm{X}_{i}}.
     \end{gathered}
       \end{align}

       We proceed by induction.
       
     \begin{itemize}[leftmargin=*]
     \item Consider $i = A$ and the cohomology related to the block
     $\bm{X}_{A}$,
     $$H^{r_{\bm{X}_{A}}}_{\Pr^{\alpha_{A}}}
     \Big( \sum_{j \in D} \beta_{j} + \sum_{k = 1}^{A-1}
     \alpha_{k} + (1 - s_0) \,
     d_{0,\bm{X}_{A}} - \sum_{j=1}^{B} s_j \Big).$$   
     
     As we assumed that $c < A$ and the triplet $(P,D,c)$ is
     determinantal data (\Cref{multilinearMixed:def:admissibleTriplet}), by
     definition, either $d_{0,\bm{X}_{A}} = 1$ or both
     $d_{0,\bm{X}_{A}} = 0$ and
     $\sum_{j \in P} d_{0,\bm{Y_j}} = 0$. Also it holds
     $0 \leq s_0, \sum_{j \in P} d_{0,\bm{Y_j}} \leq 1$. Hence, from
     \eqref{multilinearMixed:eq:sumIneqsP}, we conclude that
     $$
     \sum_{j \in D} \beta_{j} + \sum_{i = 1}^{A-1} \alpha_{i} +  (1 - s_0) \, d_{0,\bm{X}_{A}}
     - \sum_{j=1}^{B} s_j \geq - \alpha_{A}
     $$
     Therefore, by \Cref{intro:thm:possiblesRi},
     \begin{gather} \label{multilinearMixed:eq:ineqsCaseA}
       \nonumber
       H^{r_{\bm{X}_{A}}}_{\Pr^{\alpha_{A}}}
       \Big(
         \sum_{j \in D} \beta_{j} + \sum_{k = 1}^{A-1}
         \alpha_{k} + (1 - s_0) \, d_{0,\bm{X}_{A}} - \sum_{j=1}^{B} s_j
       \Big) \neq 0 \iff \\
       r_{\bm{X}_{A}} = 0 \text{ and } \sum_{j \in D} \beta_{j} + \sum_{k = 1}^{A-1}
       \alpha_{k} \geq \sum_{j=1}^{B} s_j + (s_0 -1)\,\,
       d_{0,\bm{X}_{A}}.
     \end{gather}

   \item We proceed by induction, assuming that
     \eqref{multilinearMixed:eq:ineqsXP} holds for $i+1 \leq A$, we
     prove the property for $i > c$. We consider the cohomology
  $$H^{r_{\bm{X}_{i}}}_{\Pr^{\alpha_{i}}} \Big(
    \sum_{j \in D} \beta_{j} + \sum_{k = 1}^{i-1}
    \alpha_{k} + (1 - s_0) \, d_{0,\bm{X}_{i}} -
    \sum_{j=1}^{B} s_j \Big).$$
  By definition (see \eqref{multilinearMixed:eq:casesForD0inMixedUnmixed}),
  $d_{0,\bm{X}_{i+1}} = d_{0,\bm{X}_{i}}$. So,
  if the previous cohomologies do not vanish, by induction hypothesis,
  $$
  \sum_{j \in D} \beta_{j} + \sum_{k = 1}^{i} \alpha_{k}
  \geq
  (s_0 - 1) \, \, d_{0,\bm{X}_{{i + 1}}} + \sum_{j=1}^{B} s_j
  =
  (s_0 - 1) \, \, d_{0,\bm{X}_{{i}}} + \sum_{j=1}^{B} s_j
  $$
  Equivalently,
  $$
  \sum_{j \in D} \beta_{j} + \sum_{k = 1}^{i - 1} \alpha_{k} +
  (1 - s_0) \, \, d_{0,\bm{X}_{{i}}} - \sum_{j=1}^{B} s_j
  \geq
   - \alpha_{i}.
   $$
   Hence,  by 
   \Cref{intro:thm:possiblesRi},
   \begin{align*}
        \begin{gathered}
       H^{r_{\bm{X}_{i}}}_{\Pr^{\alpha_{i}}}
       \Big( \sum_{j \in D} \beta_{j} +
       \sum_{k = 1}^{i-1} \alpha_{k} + (1- s_0) \,
       d_{0,\bm{X}_{i}}
       - \sum_{j=1}^{B} s_j  \Big) \neq 0 \iff  \\
       r_{\bm{X}_{i}} = 0
       \quad \text{and} \quad
       \sum_{j \in D} \beta_{j} +
       \sum_{k = 1}^{i-1} \alpha_{k} 
       \geq \sum_{j=1}^{B} s_j + (s_0 - 1) \, \,
       d_{0,\bm{X}_{i}}.
     \end{gathered}
   \end{align*}
   \end{itemize}

 \end{proof}

 The previous theorem gives us Sylvester-like determinantal formulas
 in some cases.
 
 \begin{corollary}[Sylvester-type formulas]
   \label{multilinearMixed:thm:sylvType}
   Consider $\bm{d}_0$ corresponding to the Center-Vertex or Edge case.
   Let
   $\sigma : \{1,\dots,A\} \rightarrow \{1,\dots,A\}$ be any
   permutation and consider the determinantal data
   $(\{1,\dots,B\},\emptyset,0)$. Then, by
   \Cref{intro:thm:sylvesterMatrix}, the overdetermined systems from
   \Cref{multilinearMixed:thm:caseUnmixedMixed} have a
   \emph{Sylvester-like formula} coming from the degree vector
   $\bm{m}$ related to the determinantal data
   $(\{1,\dots,B\},\emptyset,0)$ and the permutation $\sigma$.
 \end{corollary}

 For each determinantal formula given by
 \Cref{multilinearMixed:thm:caseUnmixedMixed}, we get another one from
 its dual.
 
  \begin{remark}
    Consider a degree vector $\bm{m}$ related to the determinantal data
    $(P,D,c)$ and the permutation $\sigma$. Then, the triplet $(D,P,A-c)$
    is also determinantal data and the map $i \mapsto (A+1-\sigma(i))$ is a
    permutation of $\{1,\dots,A\}$. Let $\bm{\bar{m}}$ be the degree
    vector associated to $(D,P,A-c)$ and $i \mapsto (A+1-\sigma(i))$,
    then, by \Cref{prop:dualDegreeVectors}, $K_\bullet(\bm{m})$ is
    isomorphic to the dual complex of $K_\bullet(\bm{\bar{m}})$.
  \end{remark}

 \subsection{Size of determinantal formulas}
 \label{multilinearMixed:sec:size-determ-form}

 Following general approaches for resultants as in
 \cite[Ch.~3]{cox2006using} or specific ideas for Koszul-type formulas
 as in \cite{bender_bilinear_2018}, we can use the matrices associated
 to the determinantal formulas from
 \Cref{multilinearMixed:thm:caseUnmixedMixed} to solve the square
 systems $(f_1,\dots,f_N)$. To express the complexity of these
 approaches in terms of the size of the output, that is, the expected
 number of solutions,  we study the size of the
 determinantal formulas of
 \Cref{multilinearMixed:thm:caseUnmixedMixed} and we compare them with
 the number of solutions of the system.
 
The \mhbb (\Cref{intro:thm:bezoutBound}) implies the following lemma.
 \begin{lemma} \label{multilinearMixed:thm:nsolsMixedUnmixed} \!\!\!\! The
   expected number of solutions, \!\!\! $\nsols$\!, \!\!\! of a square star
   multilinear system is
   
     $$\nsols := \frac{(\sum_{i = 1}^A \alpha_i)!}{\prod_{i = 1}^A \alpha_i!} \cdot
   \prod_{j = 1}^B {\eqs_j \choose \beta_j} .$$
 \end{lemma}
 
 \begin{lemma}\label{multilinearMixed:thm:sizeRes}
   The degree of the resultant and the sizes of the matrices
   corresponding to the determinantal formulas of
   \Cref{multilinearMixed:thm:caseUnmixedMixed}, that is, the rank of
   the modules $K_{0}(\bm{m})$ and $K_{1}(\bm{m})$, are (See the
   beginning of \Cref{sec:det-formulas} for the definition of the four
   cases and the notation in the bounds) as follows:
   \begin{itemize}
   \item \textbf{Center-Vertex case}: The rank of the modules is
     $\displaystyle
       \nsols \cdot (1 + \sum\nolimits_{i = 1}^A \alpha_i)
     $.
   \item
     \textbf{Outer-Vertex case}: The rank of the modules is
     $\displaystyle
     \nsols \cdot \frac{ \eqs_j + \beta_j \, (\sum\nolimits_{i = 1}^A \alpha_i) + 1}{\eqs_j - \beta_j + 1}
     $.
   \item \textbf{Edge case}: The rank of the modules is
     $\displaystyle
     \nsols \cdot
     \frac{(1 + \sum\nolimits_{i = 1}^A \alpha_i) (\eqs_j + 1)}{\eqs_j-\beta_j+1} .
     $
   \item \textbf{Triangle case}: The rank is
     $\displaystyle
     \nsols \cdot 
       ( 1 + \sum\nolimits_{i=1}^A \alpha_i ) \,
       (
       1 +
       \frac{\beta_{j_1}}{\eqs_{j_1} - \beta_{j_1} + 1} + 
       \frac{\beta_{j_2}}{\eqs_{j_2}  - \beta_{j_2} + 1}
      ).
     $
   \end{itemize}
 \end{lemma}

 The proof of this lemma can be found in \Cref{multilinearMixed:proof:sizeRes} and follows from a
 direct computation, see \Cref{intro:thm:degRes}.

\subsection{Example}
\label{multilinearMixed:example}
We follow the notation from the beginning of
\Cref{multilinearMixed:sec:mixedUmixed}.  Consider four blocks of variables
$\bm{X_1},\bm{X_2},\bm{Y_1},\bm{Y_2}$ that we partition to two sets:
$\{\bm{X_1},\bm{X_2}\}$, of cardinality $A = 2$, and
$\{\bm{Y_1},\bm{Y_2}\}$, of cardinality $B = 2$.  The number of
variables in the blocks of the first set are $\alpha = (1,1)$ and in
the second $\beta = (1,1)$. That is, we consider
{\small
$$  \bm{X_1} := \{X_{1,0}, X_{1,1} \}, 
  \bm{X_2} := \{X_{2,0}, X_{2,1}\}, \quad 
  \bm{Y_1} := \{Y_{1,0}, Y_{1,1}\},  
  \bm{Y_2} := \{Y_{2,0}, Y_{2,1}\}.
  $$
  }
Let $(f_1,\dots,f_4)$ be a square star multilinear system
corresponding to the following graph,
  \begin{center}
\begin{tikzpicture}[thick,
  fsnode/.style={draw,circle},
  ssnode/.style={draw,circle},
  every fit/.style={ellipse,draw,inner sep=-2pt,text width=2cm},
  -,shorten >= 3pt,shorten <= 3pt
]

\begin{scope}
  \node[fsnode] (x1) [label=left: ${\K[{\bm{X}_{1}},{\bm{X}_{2}}]}$] {};
\end{scope}

\begin{scope}[xshift=6cm,yshift=0.5cm,start chain=going below,node distance=5mm]
\foreach \j in {1,2}
  \node[ssnode,on chain] (y\j) [label=right: ${\K[\bm{Y}_{\j}]}$] {};
\end{scope}


\draw (x1) -- (y1) node [midway, above,sloped] {$\eqs_{1} = 2$};
\draw (x1) -- (y2) node [midway, below,sloped] {$\eqs_{2} = 2$};
\end{tikzpicture}
\end{center}
By \Cref{multilinearMixed:thm:nsolsMixedUnmixed}, the expected number of solutions of the system is $8$.
We introduce a polynomial $f_0$ and we consider the multiprojective
resultant of $\bm{f} := (f_0,f_1,\dots,f_N)$. By \Cref{multilinearMixed:thm:sizeRes}, the degree of the resultant,
depending on the choice of $f_0$, is as follows:
\begin{itemize}
\item In the Center-Vertex case, that is
  {\small $f_0 \in \K[\bm{X}_1,\bm{X}_2]_{\bm{1}}$}, the degree of the
  resultant is $24$.
\item In the Outer-Vertex case, that is
  $f_0 \in \K[\bm{Y}_j]_{\bm{1}}$ for $j \in [2]$,
  its degree is $20$.
\item In the Edge case, that is
  $f_0 \in \K[\bm{X}_1,\bm{X}_2,\bm{Y}_j]_{\bm{1}}$ for $j \in [2]$,
  its degree is $36$.
\item In the Triangle case, that
  is $f_0 \in \K[\bm{X}_1,\bm{X}_2,\bm{Y}_{1},\bm{Y}_{2}]_{\bm{1}}$,
  its degree is $48$.
\end{itemize}
We consider the Edge case, where $f_0 \in \K[\bm{X}_1,\bm{X}_2]_{\bm{1}}$, and the overdetermined system $\bm{f} = (f_0, f_1, f_2, f_3, f_4, f_5)$, where,
{\footnotesize
\begin{align*}
    \begin{array}{l l}
      f_0\!\! := \!\!\!\!\!\!&\left( a_{{1}} \, x_{2;{0}}\!+\!a_{{2}} \, x_{2;{1}} \right) x_{1;{0}}\!+\! \left( a
      _{{3}} \, x_{2;{0}}\!+\!a_{{4}} \, x_{2;{1}} \right) x_{1;{1}}  \\[7pt]
      f_1\!\! := \!\!\!\!\!\!&
 \left(  \left( b_{
{1}} \, y_{1;{0}}\!+\!b_{{2}} \, y_{1;{1}} \right) x_{2;{0}}\!+\! \left( b_{{3}} \, y_{1;{0
}}\!+\!b_{{4}} \, y_{1;{1}} \right) x_{2;{1}} \right) x_{1;{0}}   \!+\! \left(
 \left( b_{{5}} \, y_{1;{0}}\!+\!b_{{6}} \, y_{1;{1}} \right) x_{2;{0}}\!+\! \left( b_
{{7}} \, y_{1;{0}}\!+\!b_{{8}} \, y_{1;{1}} \right) x_{2;{1}} \right) x_{1;{1}}  \\[7pt]
      f_2\!\! := \!\!\!\!\!\!&
 \left(  \left( c_{{1}} \, y_{1;{0}}\!+\!c_{{2}} \, y_{1;{1}} \right) x_{2;{0}}\!+\!
 \left( c_{{3}} \, y_{1;{0}}\!+\!c_{{4}} \, y_{1;{1}} \right) x_{2;{1}} \right) x_
{1;{0}}  \!+\!  \left(  \left( c_{{5}} \, y_{1;{0}}\!+\!c_{{6}} \, y_{1;{1}} \right) x_{2
;{0}}\!+\! \left( c_{{7}} \, y_{1;{0}}\!+\!c_{{8}} \, y_{1;{1}} \right) x_{2;{1}}
 \right) x_{1;{1}} \\[7pt]
      f_3\!\! := \!\!\!\!\!\!&
 \left(  \left( d_{{1}} \, y_{2;{0}}\!+\!d_{{2}} \, y_{2;{1}}
 \right) x_{2;{0}}\!+\! \left( d_{{3}} \, y_{2;{0}}\!+\!d_{{4}} \, y_{2;{1}} \right) x
_{2;{1}} \right) x_{1;{0}}  \!+\! \left(  \left( d_{{5}} \, y_{2;{0}}\!+\!d_{{6}} \, y_{
2;{1}} \right) x_{2;{0}}\!+\! \left( d_{{7}} \, y_{2;{0}}\!+\!d_{{8}} \, y_{2;{1}}
 \right) x_{2;{1}} \right) x_{1;{1}} \\[7pt]
      f_4\!\! := \!\!\!\!\!\!&
      \left(  \left( e_{{1}} \, y_{2;{0}}\!+\!
e_{{2}} \, y_{2;{1}} \right) x_{2;{0}}\!+\! \left( e_{{3}} \, y_{2;{0}}\!+\!e_{{4}} \, y_{
2;{1}} \right) x_{2;{1}} \right) x_{1;{0}}  \!+\! \left(  \left( e_{{5}} \, y_{2
;{0}}\!+\!e_{{6}} \, y_{2;{1}} \right) x_{2;{0}}\!+\! \left( e_{{7}} \, y_{2;{0}}\!+\!e_{{
8}} \, y_{2;{1}} \right) x_{2;{1}} \right) x_{1;{1}} 
    \end{array}
\end{align*}
}

\noindent
We consider the \emph{determinantal data} $(\{1\},\{2\},1)$ and the
identity map $i \mapsto i$. Then, the degree vector of
\Cref{multilinearMixed:thm:caseUnmixedMixed} is
$\bm{m} = (0, 3, 1, -1)$ and the vector spaces of the Weyman complex
$K(\bm{m}, \bm{f})$ become
{\footnotesize
\begin{align*}
  K_1(\bm{m}, \bm{f}) = & \quad\,\,
                          S^*_{\bm{X}_1}(-1) \otimes S_{\bm{X}_2}(0) \otimes S_{\bm{Y}_1}(0) \otimes S^*_{\bm{Y}_2}(0) \otimes
                          \left\{
                          \begin{array}{l}
                          (e_{0} \wedge e_{1} \wedge e_{3})  \oplus 
                          (e_{0} \wedge e_{1} \wedge e_{4}) \; \oplus \\
                          (e_{0} \wedge e_{2} \wedge e_{3})  \oplus 
                          (e_{0} \wedge e_{2} \wedge e_{4})
                          \end{array}
  \right\} \\
                        &
                          \oplus \; S^*_{\bm{X}_1}(-1) \otimes S_{\bm{X}_2}(0) \otimes S_{\bm{Y}_1}(0) \otimes S^*_{\bm{Y}_2}(-1) \otimes
                          \left\{
                          \begin{array}{l}
                            (e_{1} \wedge e_{3} \wedge e_{4})  \oplus 
                            (e_{2} \wedge e_{3} \wedge e_{4})
                          \end{array}
  \right\} \\
                        &
                          \oplus \; S^*_{\bm{X}_1}(-1) \otimes S_{\bm{X}_2}(0) \otimes S_{\bm{Y}_1}(1) \otimes S^*_{\bm{Y}_2}(-1) \otimes
                          \left\{
                          \begin{array}{l}
                          (e_{0} \wedge e_{3} \wedge e_{4}) 
                          \end{array}
                          \right\}  ,    \\
  K_0(\bm{m}, \bm{f}) = & \quad\,\,
                          S^*_{\bm{X}_1}(0) \otimes S_{\bm{X}_2}(1) \otimes S_{\bm{Y}_1}(0) \otimes S^*_{\bm{Y}_2}(0) \otimes \left\{
                          \begin{array}{l}
                          (e_{1} \wedge e_{3}) \oplus (e_{1} \wedge e_{4}) \; \oplus \\ (e_{2} \wedge e_{3}) \oplus (e_{2} \wedge e_{4})
                          \end{array}
                          \right\} \\
                        &
                          \oplus \; S^*_{\bm{X}_1}(0) \otimes S_{\bm{X}_2}(1) \otimes S_{\bm{Y}_1}(1) \otimes S^*_{\bm{Y}_2}(0) \otimes \left\{
                          \begin{array}{l}
                            (e_{0} \wedge e_{3}) \oplus (e_{0} \wedge e_{4})
                          \end{array}
                          \right\} \\
                        &
                          \oplus \; S^*_{\bm{X}_1}(0) \otimes S_{\bm{X}_2}(1) \otimes S_{\bm{Y}_1}(1) \otimes S^*_{\bm{Y}_2}(-1) \otimes \left\{
                          \begin{array}{l}
                            (e_{3} \wedge e_{4})
                          \end{array}
                          \right\} .
\end{align*}
}
The Koszul determinantal matrix representing the map
$\delta_1(\bm{m}, \bm{f})$ between the modules \wrt the monomial basis
is
{\scriptsize
      \begin{align*}
      \arraycolsep=1.5pt
 \left[ \begin {array}{cccccccccccccccccccccccc} a_{{1}}&a_{{2}}&&&
&&&&-b_{{1}}&-b_{{2}}&-b_{{3}}&-b_{{4}}&&&&&&&&&&&& \\
 a_{{3}}&a_{{4}}&&&&&&&-b_{{5}}&-b_{{6}}&-b_ {{7}}&-b_{{8}}&&&&&&&&&&&&\\
 &&a_{{1}} &a_{{2}}&&&&&&&&&-b_{{1}}&-b_{{2}}&-b_{{3}}&-b_{{4}}&&&&& &&&\\
 &&a_{{3}}&a_{{4}}&&&&&&&&&-b_{{5} }&-b_{{6}}&-b_{{7}}&-b_{{8}}&&&&&&&&\\
 && &&a_{{1}}&a_{{2}}&&&-c_{{1}}&-c_{{2}}&-c_{{3}}&-c_{{4}}&&&&&& &&&&&&\\
 &&&&a_{{3}}&a_{{4}}&&&-c_{{5}}& -c_{{6}}&-c_{{7}}&-c_{{8}}&&&&&&&&&&&& \\
 &&&&&&a_{{1}}&a_{{2}}&&&&&-c_{{1}}&-c_{ {2}}&-c_{{3}}&-c_{{4}}&&&&&&&&\\
 &&&&& &a_{{3}}&a_{{4}}&&&&&-c_{{5}}&-c_{{6}}&-c_{{7}}&-c_{{8}}&&&&& &&&\\
 e_{{1}}&e_{{3}}&-d_{{1}}&-d_{{3}}&&&&& &&&&&&&&b_{{1}}&b_{{2}}&b_{{3}}&b_{{4}}&&&& \\
 e _{{5}}&e_{{7}}&-d_{{5}}&-d_{{7}}&&&&&&&&&&&&&b_{{5}}&b_{{6 }}&b_{{7}}&b_{{8}}&&&&\\
 e_{{2}}&e_{{4}}&-d_{{2}}&-d_{{4}}&&&&&&&& &&&&&&&&&b_{{1}}&b_{{2}}&b_{{3}}&b_{{4}}\\
 e_{{6}}&e_{{8}}&-d_{{6} }&-d_{{8}}&&&&&&&&&&&&&&&&&b_{{5}}&b_{{6}}&b_{{7}}&b_{ {8}}\\
 &&&&e_{{1}}&e_{{3}}&-d_{{1}}&-d_{{3}}&& &&&&&&&c_{{1}}&c_{{2}}&c_{{3}}&c_{{4}}&&&& \\
  &&&&e_{{5}}&e_{{7}}&-d_{{5}}&-d_{{7}}&&&&&&&&&c_{{5}}&c_{{6 }}&c_{{7}}&c_{{8}}&&&&\\
 &&&&e_{{2}}&e_{{4}}&-d_{{2}}&-d_{{4}}&&&& &&&&&&&&&c_{{1}}&c_{{2}}&c_{{3}}&c_{{4}}\\
 &&&&e_{{6}}&e_{{8}} &-d_{{6}}&-d_{{8}}&&&&&&&&&&&&&c_{{5}}&c_{{6}}&c_{{7}}&c_{ {8}}\\
 &&&&&&&&e_{{1}}&&e_{{3}}&&-d_{{1}}& &-d_{{3}}&&a_{{1}}&&a_{{2}}&&&&&\\
 &&&&&&&&e_{{5}}&&e_{{7}}&&-d_{{5} }&&-d_{{7}}&&a_{{3}}&&a_{{4}}&&&&&\\
 &&&&& &&&e_{{2}}&&e_{{4}}&&-d_{{2}}&&-d_{{4}}&&&&&&a_{{1}}&&a_{{2 }}&\\
 &&& &&&&&e_{{6}}&&e_{{8}}&&-d_{{6}}&&-d_{{8}}&&&&&&a_{{3}}&&a _{{4}}&\\
 &&&&&&&&&e_{{1}}&&e_{{3}}&&-d_{{1} }&&-d_{{3}}&&a_{{1}}&&a_{{2}}&&&&\\
 &&&&&&&&&e_{{5}}&&e_{{7}}&&-d_ {{5}}&&-d_{{7}}&&a_{{3}}&&a_{{4}}&&&&\\
 &&&& &&&&&e_{{2}}&&e_{{4}}&&-d_{{2}}&&-d_{{4}}&&&&&&a_{{1}}&&a _{{2}}\\
 && &&&&&&&e_{{6}}&&e_{{8}}&&-d_{{6}}&&-d_{{8}}&&&&&&a_{{3}} &&a_{{4}}\end {array} \right] . 
        \end{align*}
}


 \section{Determinantal formulas for bipartite bilinear system}
\label{sec:bipartiteBilinear}

In this section, we define the \emph{bipartite bilinear systems} and
we construct determinantal formulas for two different kinds of
overdetermined multihomogeneous systems related to them. These
formulas are Koszul-type determinantal formulas.
This section follow the same notation as
\Cref{multilinearMixed:sec:mixedUmixed}.

\begin{definition}[Bipartite bilinear system]
  \label{def:bipartiteBilinear}
    A square multihomogeneous system $\bm{f} = (f_1,\dots,f_N)$ in
  $\K[\bm{\bar{X},\bar{Y}}]$ with multidegrees
  ${\bm{d}_1},\dots,{\bm{d}_N} \in \Z^{q}$ is a
  \emph{bipartite bilinear system} if for every $k \in [N]$, there are
  $i_k \in [A]$ and $j_k \in [B]$ such that
  $f_k \in \K[{\bm{X}_{i_k}},{\bm{Y}_{j_k}}]_{\bm{1}}.$
  For each $i \in [A]$ and $j \in [B]$,
  let  $\eqs_{i,j}$ be the number of polynomials of $\bm{f}$ in
  $\K[{\bm{X}_i},{\bm{Y}_{\!j}}]_{\bm{1}}.$
\end{definition}

We use the term \emph{bipartite} because we can represent such systems
using a bipartite graph.  The vertices of the graph are the algebras
${\K[\bm{X}_{1}]},\dots,{\K[\bm{X}_{A}]}$,
${\K[\bm{Y}_{1}]},\dots,{\K[\bm{Y}_{B}]}$.  For each ${\bm{d}_k}$
there is an edge between the vertices ${\K[\bm{X}_{i}]}$ and
${\K[\bm{Y}_{j}]}$ whenever $d_{k,X_i} = d_{k,Y_j} = 1$.
That is, when it holds
$f_k \in \K[\bm{X}_{i},\bm{Y}_{j}]_{\bm{1}}$.
The graph is
bipartite because we can partition the vertices to two sets,
$\{{\K[\bm{X}_{1}]},\dots,{\K[\bm{X}_{A}]}\}$ and
$\{{\K[\bm{Y}_{1}]},\dots,{\K[\bm{Y}_{B}]}\}$ such that there is no
edge between vertices belonging to the same set.

\begin{example}
  \label{ex:bipartiteBilinear}  
  Let ${\bm{X}_1},{\bm{X}_2},{\bm{Y}_1},{\bm{Y}_2},{\bm{Y}_3}$ be five
  blocks of variables. Consider the multihomogeneous system
  $(f_1,f_2,f_3,f_4) \subset \K[\bm{\bar{X},\bar{Y}}]$ with
  multidegrees
    \setlength{\columnseprule}{.5pt}
  \def\columnseprulecolor{\color{black}}
  \begin{multicols}{2}
    {  
    \small
  \begin{align*}
    \begin{array}{r c c | c c c c c}
      (\!\!\!\! & d_{i,X_1}, & d_{i,X_2}, & d_{i,Y_1}, & d_{i,Y_2}, & d_{i,Y_3} & \!\!\!\!)\\ 
    {\bm{d}_1} = (\!\!\!\! & 1, & 0, & 1, & 0, & 0 & \!\!\!\!) \\
    {\bm{d}_2} = (\!\!\!\! & 1, & 0, & 0, & 1, & 0 & \!\!\!\!) \\
    {\bm{d}_3} = (\!\!\!\! & 0, & 1, & 0, & 1, & 0 & \!\!\!\!) \\
    {\bm{d}_4} = (\!\!\!\! & 0, & 1, & 0, & 0, & 1 & \!\!\!\!)
  \end{array}
  \end{align*}
  }
  \columnbreak
    \begin{center}
\begin{tikzpicture}[thick,scale=0.75, every node/.style={scale=0.75},
  fsnode/.style={draw,circle},
  ssnode/.style={draw,circle},
  every fit/.style={ellipse,draw,inner sep=-2pt,text width=2cm},
  -,shorten >= 3pt,shorten <= 3pt
]

\begin{scope}[start chain=going below,node distance=10mm]
\foreach \i in {1,2}
  \node[fsnode,on chain] (x\i) [label=left: ${\K[\bm{X}_{\i}]}$] {};
\end{scope}

\begin{scope}[xshift=6cm,yshift=0.75cm,start chain=going below,node distance=10mm]
\foreach \j in {1,2,3}
  \node[ssnode,on chain] (y\j) [label=right: ${\K[\bm{Y}_{\j}]}$] {};
\end{scope}


\draw (x1) -- (y1) node [midway, above,sloped] {$\eqs_{1,1} = 1$};
\draw (x1) -- (y2) node [pos=0.45, above,sloped] {$\eqs_{1,2} = 1$};
\draw (x2) -- (y2) node [pos=0.55, above,sloped] {$\eqs_{2,2} = 1$};
\draw (x2) -- (y3) node [midway, above,sloped] {$\eqs_{2,3} = 1$};

\draw[dashed] (x1) -- (y3) node [pos=0.8, above,sloped] {$\eqs_{1,3} = 0$};
\draw[dashed] (x2) -- (y1) node [pos=0.8, below,sloped] {$\eqs_{2,1} = 0$};
\end{tikzpicture}
\end{center}
\end{multicols}
\noindent
This system is a bipartite bilinear system where $\eqs_{1,1} = 1$,
  $\eqs_{1,2} = 1$, $\eqs_{1,3} = 0$, $\eqs_{2,1} = 0$,
  $\eqs_{2,2} = 1$ and $\eqs_{2,3} = 1$.
  The corresponding bipartite graph is the one above.

\end{example}

\begin{remark} \label{intro:thm:sumOfVariablesBipartite}
    For each square bipartite bilinear system, it holds
  $N = \sum_{i=1}^A \sum_{j=1}^B \eqs_{i,j}$.
  Moreover, if the system has a nonzero finite number of solutions,
  then for $i \in \{1,\dots,A\}$, it holds
  $\sum_{j = 1}^B \eqs_{i,j} \geq \alpha_i$ and for each
  $j \in \{1,\dots,B\}$ it holds
  $\sum_{i = 1}^A \eqs_{i,j} \geq \beta_j$, see
  \Cref{intro:thm:bezoutBound}.
\end{remark}

As we did in \Cref{multilinearMixed:sec:mixedUmixed},
we study overdetermined polynomial systems $(f_0,f_1,\dots,f_N)$ in
$\K[\bm{\bar{X},\bar{Y}}]$ where $(f_1,\dots,f_N)$ is a square
bipartite bilinear system and $f_0$ is a multilinear polynomial.
We consider different types of polynomials $f_0$. The obvious choice
for $f_0$ is to have the same structure as one of the polynomials
$f_1,\dots,f_N$; still we also choose $f_0$ to have a different
support. This leads to resultants of smaller degrees and so to
matrices of smaller size. The following $f_0$ lead to determinantal
formulas:
 \begin{enumerate}
   \item $f_0 \in \K[\bm{X}_i]_{\bm{1}}$, for any $i \in \{1,\dots,A\}$.
   \item $f_0 \in \K[\bm{Y}_j]_{\bm{1}}$, for any $j \in \{1,\dots,B\}$.
   \item $f_0 \in \K[\bm{X}_i,\bm{Y}_{j}]_{\bm{1}}$, for any
     $i \in \{1,\dots,A\}$ and $j \in \{1,\dots,B\}$.
   \end{enumerate}
 
  \begin{theorem} \label{bipartite:thm:caseBipartite}
    Consider a generic overdetermined system
    $\bm{F} = (F_0,\dots,F_N)$ in 
    $\Z[\bm{u}][\bm{\bar{X},\bar{Y}}]$ of multidegrees
    ${\bm{d}_0},\dots,{\bm{d}_N}$ (\Cref{intro:def:genMultiHom}),
    where $(F_1,\dots,F_n)$ is a square bipartite bilinear
    system. Assume that for each $i \in \{1,\dots,A\}$,
    $\sum_{j = 1}^B \eqs_{i,j} \geq \alpha_i$ and for each
    $j \in \{1,\dots,B\}$, $\sum_{i = 1}^A \eqs_{i,j} \geq \beta_j$
    (see \Cref{intro:thm:sumOfVariablesBipartite}), and $F_0$ is a
    multilinear polynomial as detailed in the paragraph above of
    multidegree $\bm{d}_0$.

    The degree vector
    $\bm{m} =
    (m_{\bm{X}_{1}},\dots,m_{\bm{X}_{A}},m_{\bm{Y}_{1}},\dots,m_{\bm{Y}_{B}})$ define by
    \begin{align*}
      \left\{
      \begin{array}{lcl}
        m_{\bm{X}_{i}} = \sum_{j = 1}^B \eqs_{i,j} - \alpha_i + d_{0,\bm{X}_i} && \text{for } 1 \leq i \leq A \\
        m_{\bm{Y}_{j}} = -1 && \text{for } 1 \leq j \leq B
      \end{array}\right.
    \end{align*}
    corresponds to a Koszul-type determinantal
    formula (\Cref{def:KoszulDetFormula})
   $$K_\bullet(\bm{m}) : 0 \rightarrow
   K_{1,\sum_{j = 1}^B \beta_j+1}(\bm{m}) \xrightarrow{\delta_1(\bm{m})}
    K_{0,\sum_{j = 1}^B \beta_j}(\bm{m})
    \rightarrow 0.$$
 \end{theorem}

 \noindent For the sake of brevity, we present the proof of
 \Cref{bipartite:thm:caseBipartite} in \Cref{proof:bipartiteBilinear}
 as it is similar to the one of
 \Cref{multilinearMixed:thm:caseUnmixedMixed}.
 Additionally, we present an example of the determinantal formulas
 constructed in this section in \Cref{bipartite:example}.

\section{Solving the Multiparameter Eigenvalue Problem}
\label{sec:ex-MEP}

We present an algorithm (and an example) for solving a nonsingular MEP.
The polynomial system associated to MEP, see \eqref{eq:bilinearMEP}, corresponds to a \textit{star multilinear system} (\Cref{def:multilinearMixed}), where
$A = 1$, $B = \alpha$ and
$\eqs_j = \beta_j + 1$, for each $j \in [B]$.
In particular, following (\ref{eq:bilinearMEP}), the system is
$\bm{f}^{\mathrm{MEP}} := (f_{1,0},\dots,f_{1,\beta_1},\dots,f_{\alpha,0},\dots,f_{\alpha,\beta_\alpha})$, where
$f_{i,j} \in \K[\bm{Y}_1,\dots,\bm{Y}_{i-1},\bm{Y}_{i+1},\dots,\bm{Y}_\alpha]_{\bm{1}}$, for $i \in [\alpha]$ and $j+1 \in [\beta_i+1]$.
The expected number of solutions is
$\prod_{j = 1}^\alpha (\beta_j + 1)$,
see~\Cref{multilinearMixed:thm:nsolsMixedUnmixed}.

We introduce a linear form $f_0 \in \K[\bm{X}]_{\bm{1}}$ and consider the
Sylvester-type determinantal formula of \Cref{multilinearMixed:thm:sylvType}.
The map $\delta$ associated to this formula is as follows,
 \begin{align}\label{eq:detMapMEP}
   \begin{array}{r c}
     \delta : 
                K[\bm{Y}_1,\dots,\bm{Y}_\alpha]_{\bm{1}}
                \times \prod_{i=1}^\alpha
                \K[\bm{Y}_1 \dots \bm{Y}_{i-1},\bm{Y}_{i+1} \dots \bm{Y}_\alpha]_{\bm{1}} \rightarrow
     &
       \K[\bm{X},\bm{Y}_1,\dots,\bm{Y}_\alpha]_{\bm{1}}
     \\
               (g_0,g_{1,0},\dots,g_{1,\beta_1},\dots,g_{\alpha,0},\dots,g_{\alpha,\beta_\alpha})  \mapsto &
                                                                                                              g_0 \, f_0 + \sum_{i = 1}^\alpha \sum_{j=0}^{\beta_i} g_{i,j} f_{i,j} .
   \end{array}
   \nonumber
   \\[-15pt]
\end{align}

We fix a monomial basis for the domain and codomain of $\delta$ and we
construct a matrix $C$ associated to it.
We arrange the rows
and columns of $C$ so that we can write  it as 
$\left[ \begin{smallmatrix} C_{1,1} & C_{1,2} \\ C_{2,1} & C_{2,2}
  \end{smallmatrix} \right] $, such that
\begin{itemize}[leftmargin=*]
\item The submatrix $\left[ \begin{smallmatrix} C_{2,1} & C_{2,2}
    \end{smallmatrix} \right] $ corresponds to the rows
  $\bm{Y}^\theta \, f_0$, for 
  $\bm{Y}^\theta \in \K[\bm{Y}_1,\dots,\bm{Y}_\alpha]_{\bm{1}}$.
\item The submatrix $\left[ \begin{smallmatrix}  C_{1,2} \\ C_{2,2}
    \end{smallmatrix} \right] $ corresponds to the column associated
  to the monomial $\bm{Y}^\theta \, X_0$.
\item If the $k$-th row of $C$ corresponds to $\bm{Y}^\theta \, f_0$,
  then the $k$-th column corresponds to the monomial
  $\bm{Y}^\theta \, x_0$.
\end{itemize}
We say that a \textsc{MEP} is \emph{affine} if $\bm{f}^{\mathrm{MEP}}$
is a zero-dimensional system and for every solution the
$x_0$-coordinate is not zero. When $\bm{f}^{\mathrm{MEP}}$ has a
finite number of solutions, we can always assume that it is
\emph{affine} by performing a structured linear change of coordinates.
When the \textsc{MEP} is affine, by
\cite[Prop.~4.5]{bender_bilinear_2018}, the matrix $C_{1,1}$ is
invertible. Moreover, by \cite[Lem.~4.4]{bender_bilinear_2018}, we
have a one to one correspondence between the eigenvalues of the MEP
and the (classical) eigenvalues of the Schur complement of $C_{2,2}$,
$\widetilde{C_{2,2}} := C_{2,2} - C_{2,1} C^{-1}_{1,1} C_{1,2}$ : each
eigenvalue of $\widetilde{C_{2,2}}$ is the evaluation of
$\frac{f_0}{x_0}$ at an eigenvalue of the original MEP.
Also, the right eigenspaces of $\widetilde{C_{2,2}}$ correspond to the
vector of monomials $x_0 \, \bm{Y}^\alpha$, for
$\bm{Y}^\theta  \in \K[\bm{Y}_1,\dots,\bm{Y}_\alpha]_{\bm{1}}$,
evaluated at each solution of
$(f_{1,0},\dots,f_{1,\beta_1},\dots,f_{\alpha,0},\dots,f_{\alpha,\beta_\alpha})$
\cite[Lem.~5.1]{emiris1996complexity}.
Even more, if we multiply these
eigenvectors by
$\bigl[\begin{smallmatrix} M_{1,1}^{-1} \cdot M_{2,1} \\
  I \end{smallmatrix}\bigr]$, then we recover a vector
corresponding to the evaluation of every monomial in
$\K[\bm{X},\bm{Y}_1,\dots,\bm{Y}_\alpha]_{\bm{1}}$ evaluated at the
original solution. This information suffices to recover all the
coordinates of the solution.

\begin{remark}[Multiplication map]
  For an affine \textsc{MEP}, the matrix $\widetilde{C_{2,2}}$
  corresponds to the multiplication map of
   the rational function
  $\frac{f_0}{x_0}$ in the quotient ring
  $\K[\bm{X},\bm{Y}_1,\dots,\bm{Y}_\alpha] / \bm{f}^{\mathrm{MEP}}$ at
  multidegree $(0,1,\dots,1) \in \N^{\alpha+1}$, \wrt the monomial
  basis
  $\{\bm{Y}^\theta\}_{\bm{Y}^\theta \in
    \K[\bm{Y}_1,\dots,\bm{Y}_\alpha]_{\bm{1}}}$.
\end{remark}

{\setlength{\textfloatsep}{5pt}
  \small
    \begin{algorithm}
      \caption{$\texttt{SolveMEP}(\{\{M^{(i,j)}\}_{j \in [\beta_i + 1]}\}_{i \in [\alpha]})$}
      \begin{algorithmic}[1]
        \label{alg:solveMEP}
        \REQUIRE
        {Affine \textsc{MEP} $\{\{M^{(i,j)}\}_{j \in [\beta_i + 1]}\}_{i \in [\alpha]}$}

        \STATE $(f_{1,0},\dots,f_{\alpha,\beta_\alpha}) \leftarrow$
          Multilinear system associated to
          $\{\{M^{(i,j)}\}_{j \in [\beta_i + 1]}\}_{i \in [\alpha]}$
          (Eq.~\ref{eq:bilinearMEP}).

          \STATE $f_0 \leftarrow$ Generic linear polynomial in
            $\K[\bm{X}]_1$.

          \STATE
            $ \bigl[\begin{smallmatrix}C_{1,1} & C_{1,2} \\ C_{2,1} &
              C_{2,2} \end{smallmatrix} \bigr]$ $\leftarrow$ Matrix
            corresponding to $\delta$ (Eq.~\ref{eq:detMapMEP});
            partitioned in four block

      \STATE $\left\{\left(\frac{f_0}{x_0}(p), \bar{v}_p\right)\right\}_{p} \leftarrow$ 
              Set of
              Eigenvalue-Eigenvector of the Schur compl. of $C_{2,2}$.

        \FORALL{$\left(\frac{f_0}{x_0}(p), \bar{v}_p\right) \in \left\{\left(\frac{f_0}{x_0}(p), \bar{v}_p\right)\right\}_{p}$}

        \STATE Extract the coordinates of $p$ from  $\bigl[\begin{smallmatrix} C_{1,1}^{-1} \cdot C_{2,1} \\
            I \end{smallmatrix}\bigr] \cdot \bar{v}_p$. 
        \ENDFOR
      \end{algorithmic}
       \end{algorithm}
    }

    \begin{remark}
      [Atkinson's Delta Method]
      \label{rem:comparison-w-Atkinson}
      By inspecting the eigenvalues and eigenvectors of the Schur
      complement of $C_{2,2}$ for $f_0 = \frac{x_i}{x_0}$, we can
      conclude that $\widetilde{C_{2,2}}$ equals the matrix
      $\Delta^{-1}_0 \, \Delta_i$ from Atkinson's Delta method, see
      \cite[Chp.~6]{atkinson_multiparameter_1972}. It worth to
      point out that our construction of $\widetilde{C_{2,2}}$
      improves Atkinson's construction of $\Delta^{-1}_0 \, \Delta_i$
      as we avoid the symbolic expansion by minors that he considers
      \cite[Eq.~6.4.4]{atkinson_multiparameter_1972}.
      The dimension of the matrix $C_{2,2}$, and so the dimension of
      Atkinson's Delta matrices, is
      $\prod_{i = 1}^\alpha (\beta_i + 1) \times \prod_{i = 1}^\alpha
      (\beta_i + 1)$. These matrices are dense. In contrast, the
      dimension of the matrix $C$ is
      $(\alpha + 1) \, \prod_{i = 1}^\alpha (\beta_i + 1) \times
      (\alpha + 1) \, \prod_{i = 1}^\alpha (\beta_i + 1)$ (the degree
      of the resultant of the system), but the matrix is structured
      (i.e., multi-Hankel matrix) and sparse; it has at most
      $(\sum_{i = 1}^\alpha \beta_i + \alpha + 1) \, (\alpha + 1) \,
      \prod_{i = 1}^\alpha (\beta_i + 1)$ non-zero positions.
    \end{remark}

    In what follows we present an example of our algorithm in the
    two-parameter eigenvalue problem (2EP).
    The interested reader can found information about the applications
    of 2EP in physics in \cite{gheorghiu_spectral_2012}.
    We consider the 2EP given by the matrices
{\small
\begin{align} \label{exMEP:eq:def}
\begin{array}{c c c}
  M^{(1,0)} := \left[  \begin {array}{cc} -7&-3 \\ \noalign{\medskip}-8&-2 \end {array} \right]
  &
  M^{(1,1)} := \left[ \begin {array}{cc} 12&2 \\ \noalign{\medskip}13&1\end {array} \right]
  &
  M^{(1,2)} := \left[ \begin {array}{cc} -7&-1\\ \noalign{\medskip}-7&-1\end {array} \right] 
  \\[20pt]
  M^{(2,0)} := \left[ \begin {array}{cc} -11&-3\\ \noalign{\medskip}4&1\end {array} \right]
  &
  M^{(2,1)} := \left[ \begin {array}{cc} 7&-1\\ \noalign{\medskip}1&2 \end {array} \right]
  &
  M^{(2,1)} := \left[ \begin {array}{cc} -4&0\\ \noalign{\medskip}-1&-1\end {array} \right]
\end{array}.
\end{align}
}
For simplicity, we
will name the three blocks of variables as $\bm{X},\bm{Y},\bm{Z}$,
instead of $\bm{X}_1,\bm{Y}_1,\bm{Y}_2$. Following
\eqref{eq:bilinearMEP}, we write 2EP as the following bilinear system
{\small
$$
\begin{array}{l  c l}
\left[ \begin {array}{c} f_1 \\ \noalign{\medskip} f_2 \end {array} \right]
= &
\left[ \begin {array}{cc}
         -7\,x_{{0}}+12\,x_{{1}}-7\,x_{{2}}&-3\,x_{{0}}+2\,x_{{1}}-x_{{2}}\\
         \noalign{\medskip}-8\,x_{{0}}+13\,x_{{1}}-7\,x_{{2}}&-2\,x_{{0}}+x_{{1}}-x_{{2}}
\end{array} \right]
& \cdot 
\left[ \begin {array}{c} y_{0}\\ \noalign{\medskip}y_{1} \end {array} \right]
\\[20pt]
\left[ \begin {array}{c} f_3 \\ \noalign{\medskip} f_4 \end {array} \right]
= &
\left[ \begin {array}{cc}
         -11\,x_{{0}}+7\,x_{{1}}-4\,x_{{2}}&-3\,x_{{0}}-x_{{1}}\\
         \noalign{\medskip}4\,x_{{0}}+x_{{1}}-x_{{2}}&x_{{0}}+2\,x_{{1}}-x_{{2}}
       \end {array} \right]
& \cdot 
\left[ \begin {array}{c} z_{0}\\ \noalign{\medskip}z_{{1}} \end {array} \right] 
\end{array} .
$$
} According to \Cref{multilinearMixed:thm:nsolsMixedUnmixed}, the 2EP
should have $4$ different solutions. To solve this system, we
introduce a linear polynomial $f_0 \in \K[\bm{X}]$ that separates the
eigenvalues of the 2EP, that is, if $\bm{\lambda}_1$ and
$\bm{\lambda}_2$ are different eigenvalues, then
$\frac{f_0}{x_i}(\bm{\lambda}_1) \neq
\frac{f_0}{x_i}(\bm{\lambda}_2)$, for some ${x_i} \in \bm{X}$. Then,
we consider a Sylvester-like determinantal formula for the resultant
of $(f_0,\dots,f_4)$ (\Cref{multilinearMixed:thm:sylvType}) and we
solve the original system using eigenvalue and eigenvector
computations as in \cite{bender_bilinear_2018}.

If the MEP problem has a finite number of eigenvalues and all
of them are different, then any generic $f_0 \in \K[\bm{X}]$ separates the eigenvalues.
  In our case, we choose $f_0 := -x_0 + 5 \, x_1- 3 \, x_2$.
  Following \Cref{multilinearMixed:thm:sylvType}, there is a
  Sylvester-type formula for the resultant of
  the system 
  $\bm{f} := (f_0,\dots,f_4)$
  using the degree vector $\bm{m} := ( 1, 1, 1 )$.
  The latter is related 
  to the determinantal data $(\{1,2\},\emptyset,0)$.
The  Weyman complex reduces to
{\small
  \begin{align*}
0 \rightarrow
\left(
  \begin{array}{l }
    \phantom{\oplus \,}
  S_{X}(0) \otimes S_{Y}(1) \otimes S_{Z}(1) \otimes \{ e_0 \} \\
   \oplus \,
    S_{X}(0) \otimes S_{Y}(0) \otimes S_{Z}(1) \otimes \{ e_1 \oplus e_2 \} \\
   \oplus \,
    S_{X}(0) \otimes S_{Y}(1) \otimes S_{Z}(0) \otimes \{ e_3 \oplus e_4 \}
\end{array}
\right)
   \xrightarrow{\delta_1(\bm{m}, \bm{f})} 
  \left(S_{X}(1) \otimes S_{Y}(1) \otimes S_{Z}(1) \otimes \K  \right)
  \rightarrow 0 ,
\end{align*}
}

\noindent
where the map $\delta_1(\bm{m}, \bm{f})$ is a Sylvester map
(\Cref{intro:thm:sylvesterMatrix}).
Hence, the resultant of $\bm{f}$ is the determinant of a matrix $C$
representing this map, which has dimensions $12 \times 12$ (Case 1,
\Cref{multilinearMixed:thm:sizeRes}).
     We split $C$ in
     $\left[ \begin{smallmatrix} C_{1,1} & C_{1,2} \\ C_{2,1} & C_{2,2}
       \end{smallmatrix} \right] $ according to
     \cite[Def.~4.1]{bender_bilinear_2018}.
{\footnotesize
  $$
  \left[
\arraycolsep=2pt
 \begin {array}{c | cccccccc | cccc} & x_{{2}}y_{{0}}z_{{0}}&
      x_{{2}}y_{{0}}z_{{1}}& x_{{2}}y_{{1}}z_{{0}}&
      x_{{2}}y_{{1}}z_{{1}}& x_{{1}}y_{{0}}z_{{0}}&
      x_{{1}}y_{{0}}z_{{1}}& x_{{1}}y_{{1}}z_{{0}}&
      x_{{1}}y_{{1}}z_{{1}}& x_{{0}}y_{{0}}z_{{0}}&
      x_{{0}}y_{{0}}z_{{1}}& x_{{0}}y_{{1}}z_{{0}}&
      x_{{0}}y_{{1}}z_{{1}}\\ \hline
           z_{{0}} e_1&-7&&-1&&12&&2&&-7&&-3&\\
           z_{{1}} e_1&&-7&&-1&&12&&2&&-7&&-3\\
           z_{{0}} e_2&-7&&-1&&13&&1&&-8&&-2&\\
           z_{{1}} e_2&&-7&&-1&&13&&1&&-8&&-2\\
           y_{{0}} e_3&-4&0&&&7&-1&&&-11&-3&&\\
           y_{{1}} e_3&&&-4&0&&&7&-1&&&-11&-3\\
           y_{{0}} e_4&-1&-1&&&1&2&&&4&1&&\\
           y_{{1}} e_4&&&-1&-1&&&1&2&&&4&1\\ \hline
           y_{{0}}z_{{0}} e_0&-3&&&&5&&&&-1&&&\\
           y_{{0}}z_{{1}} e_0&&-3&&&&5&&&&-1&&\\
           y_{{1}}z_{{0}} e_0&&&-3&&&&5&&&&-1&\\
           y_{{1}}z_{{1}} e_0&&&&-3&&&&5&&&&-1\end{array} \right] . $$ }
     \begin{remark}
       If the original MEP has a finite number of 
       eigenvalues, after performing a generic linear change of
       coordinates in the variables $\bm{X}$, we can assume that
       there is no solution of $(f_1,\dots,f_n)$ such that $x_0 = 0$.
     \end{remark}

     \begin{wrapfigure}{r}{0.35\textwidth}
       \vspace{-30pt}
       \begin{center}
         {\scriptsize
           $$
           \widetilde{C_{2,2}} =
           \left[ \begin {array}{rrrr}
                    \frac{7}{4}&0&-\frac{1}{4}&-\frac{1}{2}\\
                    \noalign{\medskip}-\frac{3}{4}& \frac{3}{2}&\frac{9}{4}&2\\
                    \noalign{\medskip}-{\frac {21}{4}}&-3&{\frac {27}{4}}&\frac{5}{2} \\
              \noalign{\medskip}{\frac {69}{4}}&\frac{19}{2}&-{\frac {63}{4}}&-6
              \end {array} \right] .
            $$
          }
        \end{center}
       \vspace{-30pt}
      \end{wrapfigure}
      \vspace{-.2\baselineskip}
     By \cite[Prop.~4.5]{bender_bilinear_2018}, as the
     system $(f_1,\dots,f_n)$ has no solutions such that $x_0 = 0$,
     the matrix $C_{1,1}$ is nonsingular.
     Hence, by \cite[Lem.~4.4]{bender_bilinear_2018}, we have a one
     to one correspondence between the eigenvalues of the 2EP and the
     (classical) eigenvalues of the Schur complement of $C_{2,2}$,
     $\widetilde{C_{2,2}} := C_{2,2} - C_{2,1} C^{-1}_{1,1} C_{1,2}$ :
     each eigenvalue of $\widetilde{C_{2,2}}$ is the evaluation

     \noindent of
     $\frac{f_0}{x_0}$ at an eigenvalue of the original 2EP.
     In our case $\widetilde{C_{2,2}}$ is as it appears at the left.

     Let $p_1,\dots,p_4$ be the four solutions of
     $(f_1,\dots,f_4)$. Then, the eigenvalues of $\widetilde{C_{2,2}}$
     are
     $\frac{f_0}{x_0}(p_1) = 1, \frac{f_0}{x_0}(p_2) = 2,
     \frac{f_0}{x_0}(p_3) = 3$ and
     $\frac{f_0}{x_0}(p_4) = -2$. As $\delta_1(\bm{m},\bm{f})$ is
     a Sylvester-like map, the right eigenspaces of
     $\widetilde{C_{2,2}}$ contain the vector of monomials
     $\bm{v} := \left[ \begin{smallmatrix}
         x_{{0}}y_{{0}}z_{{0}} \\ x_{{0}}y_{{0}}z_{{1}} \\
     x_{{0}}y_{{1}}z_{{0}} \\ x_{{0}}y_{{1}}z_{{1}}
   \end{smallmatrix} \right] $ evaluated at each solutions of
 $(f_1,\dots,f_4)$ \cite[Lem.~5.1]{emiris1996complexity};
 this is as it appears in the table below. If each
 eigenspace has dimension one, then we can recover some coordinates of
 the solutions by inverting the monomial map given by $\bm{v}$.
 
 \begin{wrapfigure}{l}{0.35\textwidth}
   \vspace{-17.5pt}
     $$
     \begin {array}{| c|r|r|r|r|}
       \hline
       \bm{v} &p_1&p_2&p_3&p_4\\ \hline
       x_{{0}}y_{{0}}z_{{0}}&1&1&1&1\\
       x_{{0}}y_{{0}}z_{{1}}&-3&-1&-2&-3 \\
       x_{{0}}y_{{1}}z_{{0}}&1&1&-1&-3\\
       x_{{0}}y_{{1}}z_{{1}}&-3&-1&2&9 \\ \hline
       \end {array}
     $$
       \vspace{-20pt}
\end{wrapfigure}
 For
 example, the $z_0$-coordinate of $p_1$ is non-zero as  $(x_0 \, y_0 \, z_0)(p_1) \neq 0$, and so its $z_1$-coordinate equals
 $\frac{(x_0 \, y_0 \, z_1)}{(x_0 \, y_0 \, z_0)}(p_1) = -3$.
      To compute the remaining coordinates, either we substitute the
     computed coordinates of the solutions in the original system and
     we solve a linear system, or we extend each eigenvector
     $\bm{v}(p_i)$ to $\bm{w}(p_i)$, where
     $\bm{w}(p_i)$ is the solution of the following linear system:
     $$
          \left[ \begin{smallmatrix} C_{1,1} & C_{1,2} \\ C_{2,1} & C_{2,2}
  \end{smallmatrix} \right]  \bm{w}(p_i) = \frac{f_0}{x_0}(p_i) \left[ \begin{smallmatrix} \\[1.5pt] 0 \\[3pt] \bm{v}(p_i) \\[1.5pt]
  \end{smallmatrix} \right], \text{ and so }
     \bm{w}(p_i) = \left[ \begin{smallmatrix} - C_{1,1}^{-1} \cdot C_{1,2} \, \bm{v}(p_i) \\ \bm{v}(p_i) 
       \end{smallmatrix} \right].
     $$
     Each coordinate of $\bm{w}(p_i)$ is a monomial in
     $\K[\bm{X}]_1 \otimes \K[\bm{Y}]_1 \otimes \K[\bm{Z}]_1$
     evaluated at $p_i$. Hence, we can recover the coordinates of
     $p_i$ from $\bm{w_i}(p_i)$ by inverting a monomial
     map. In this case, the solutions to $(f_1,\dots,f_4)$, and so
     eigenvalues and eigenvectors of 2EP are
     {\footnotesize
     \begin{align*}
       \begin{array}{r c | c | c l}        
         & \text{Ext. Eigenvalues}  &  \omit\rlap{\text{\, Eigenvectors}} &  & \\
         & x_0,x_1,x_2  &  y_0,y_1  &  z_0,z_1 & \\ \hline
         p_1  = (& 1,-1,-3 &  1,1  &  1, -3 &) \\
         p_2  = (& 1,3,4  &  1,1  &  1, -1 &) \\
         p_3  = (& 1,1,1  &  1,-1  &  1, -2 &) \\
         p_4  = (& 1,1,2  &  1,-3  &  1, -3 &) 
       \end{array}                                 
     \end{align*}
     }

  \section*{Acknowledgments}
  We thank Laurent Bus\'e, Carlos D'Andrea, and Agnes Szanto for
  helpful discussions and references, Jose Israel Rodriguez for
  pointing out the relation between our paper
  \cite{bender_bilinear_2018} and MEP, and the anonymous reviewers for
  their helpful comments.
     The authors are partially
     supported by ANR JCJC GALOP (ANR-17-CE40-0009), the PGMO grant ALMA
     and the PHC GRAPE. M.~R.~Bender is supported by  ERC under
     the European's Horizon 2020
     research and innovation programme
     (grant No 787840).




  \newpage
  \appendix

  \section{Additional proofs and examples}
  \subsection{Proof of Lemma~\ref{multilinearMixed:thm:sizeRes}}
  \label{multilinearMixed:proof:sizeRes}
   \begin{proof}[Proof of Lemma~\ref{multilinearMixed:thm:sizeRes}]
   Following \Cref{intro:thm:sizeVectorSpacesKoszulMap}, the size of
   the Koszul determinantal matrix is the degree of the resultant.
   By \Cref{intro:thm:degRes}, it holds,
   $$\text{degree}(\res) = \sum_{k = 0}^N \MHB({\bm{d}_0},\dots,{\bm{d}_{k-1}},{\bm{d}_{k+1}},\dots,{\bm{d}_N}).$$
   As many multidegrees $\bm{d}_k$ are identical, we couple the
   summands in the previous equation.  For each $j \in \{1,\dots,B\}$,
   let $I_j \in \{1,\dots,N\}$ be the index of a polynomial in
   $\bm{F}$ such that
   $F_j \in \K[\bm{X}_1,\dots,\bm{X}_A,\bm{Y}_{I_j}]$.  Recall that
   $\eqs_j$ is the number of polynomials with multidegree equal to
   $\bm{d_{I_j}}$. Hence, we rewrite the degree of the resultant as
   $$\deg(\Res_\P(\bm{d}_0,\dots,\bm{d}_n)) =
   \MHB(\bm{d}_1,\dots,\bm{d}_n)
   +
   \sum_{j = 1}^B \eqs_j \;
   \MHB(\bm{d}_0,\dots,\bm{d}_{I_{j}-1},\bm{d}_{I_{j}+1},\dots,\bm{d}_n).$$

   From \Cref{multilinearMixed:thm:nsolsMixedUnmixed},
   $ \MHB(\bm{d}_1,\dots,\bm{d}_n) = \frac{(\sum_{i = 1}^A
     \alpha_i)!}{\prod_{i = 1}^A \alpha_i!} \cdot \prod_{j = 1}^B
   {\eqs_j \choose \beta_j} = \nsols$.
   By \Cref{intro:thm:bezoutBound}, for every $1 \leq j \leq B$,
   $\MHB(\bm{d}_0,\dots,\bm{d}_{I_j-1},\bm{d}_{I_j+1},\dots,\bm{d}_n)$
   is the coefficient of
   $(\prod_{i=1}^A Z_{X_i}^{\alpha_i}) (\prod_{t=1}^B
   Z_{Y_t}^{\beta_t})$ in
   $$ \Big(\sum_{i=1}^A d_{0,\bm{X}_i} Z_{X_i} + \sum_{t = 1}^B d_{0,\bm{Y}_t} Z_{X_t} \Big)
    \Big(\sum_{i=1}^A Z_{X_i} + Z_{Y_j} \Big)^{\eqs_j - 1}
    \prod_{k \in \{1,\dots,B\} \setminus \{j\}}  \Big(\sum_{i=1}^A Z_{X_i} + Z_{Y_k} \Big)^{\eqs_k}.$$

    Consider the last two factors of the previous equation, that is
    \begin{align}\label{multilinearMixed:eq:polOfCoeffs}
      \Big(\sum_{i=1}^A Z_{X_i} + Z_{Y_j} \Big)^{\eqs_j - 1}
      \prod_{k \in \{1,\dots,B\} \setminus \{j\}}  \Big(\sum_{i=1}^A Z_{X_i} + Z_{Y_k} \Big)^{\eqs_k}.
    \end{align}
    Then
    $$
    \MHB(\bm{d}_0,\dots,\bm{d}_{I_j-1},\bm{d}_{I_j+1},\dots,\bm{d}_n) = \sum_{s = 1}^{A} d_{0,\bm{X}_s} \theta_{j,s}^X + \sum_{l = 1}^{B}
    d_{0,\bm{Y}_l} \theta_{j,l}^Y,
    $$
    where $\theta_{j,s}^X$ is the coefficient of
    $\frac{(\prod_{i=1}^A Z_{X_i}^{\alpha_i}) (\prod_{t=1}^B
      Z_{Y_t}^{\beta_t})}{Z_{X_s}}$ in
    \eqref{multilinearMixed:eq:polOfCoeffs}, and $\theta_{j,t}^Y$ is
    the coefficient of
    $\frac{(\prod_{i=1}^A Z_{X_i}^{\alpha_i}) (\prod_{l=1}^B
      Z_{Y_l}^{\beta_t})}{Z_{Y_t}}$ in
    \eqref{multilinearMixed:eq:polOfCoeffs}.
    After some computations, we have 
   \begin{align*}
     \theta_{j,s}^X  =
   \frac{((\sum_{i=1}^A \alpha_i) - 1)!}{(\alpha_{s} - 1)! \prod_{i \in \{1,\dots, A\} \setminus \{s\}} \alpha_i!}
     {\eqs_j - 1 \choose \beta_j} \prod_{k \in \{1, \dots, B\} \setminus \{j\} } {\eqs_k \choose \beta_k} 
     =
     \nsols \cdot \frac{\alpha_s}{\sum_{i = 1}^A \alpha_i} \frac{\eqs_j - \beta_j}{\eqs_j},
   \end{align*}
   \begin{align*}
   \theta_{j,t}^Y =
   \left\{
   \begin{array}{ll}
     \frac{(\sum_{i=1}^A \alpha_i)!}{\prod_{i  =1}^A \alpha_i!}
     {\eqs_j - 1 \choose \beta_j - 1} \prod_{k \in \{1,\dots,B\} \setminus \{j\}}
     {\eqs_k \choose \beta_k} = \nsols \cdot \frac{\beta_j}{\eqs_j}
          & \text{ if } t = j, \\
     \frac{(\sum_{i=1}^A \alpha_i)!}{\prod_{i  =1}^A \alpha_i!}
     {\eqs_t \choose \beta_t - 1} {\eqs_j - 1 \choose \beta_j} \prod_{k \in \{1,\dots,B\} \setminus \{t,j\}}
     {\eqs_k \choose \beta_k} =
     \nsols \cdot \frac{\beta_t}{\eqs_t - \beta_t + 1} \frac{\eqs_j - \beta_j}{\eqs_j}
          & \text{ otherwise.} \\
   \end{array}\right. 
   \end{align*}
   Using the formulas for  $\theta_{j,s}^X$  and $\theta_{j,t}^Y$
   we get 
   \begin{align*}
   \deg(\Res_\P(\bm{d}_0,\dots,\bm{d}_n)) =
   \MHB(\bm{d}_1,\dots,\bm{d}_n)
   +
   \sum_{j = 1}^B \eqs_j \;
   \MHB(\bm{d}_0,\dots,\bm{d}_{I_{j}-1},\bm{d}_{I_{j}+1},\dots,\bm{d}_n)
     = \\
     \nsols + \sum_{j = 1}^B \eqs_j \;
   (\sum_{s = 1}^{A} d_{0,\bm{X}_s} \theta_{j,s}^X + \sum_{l = 1}^{B}
   d_{0,\bm{Y}_t} \theta_{j,t}^Y)
     =
     \nsols + \sum_{j = 1}^B \eqs_j \;
     (\sum_{s = 1}^{A} d_{0,\bm{X}_s} \theta_{j,s}^X) +
     \sum_{j = 1}^B \eqs_j \;
     (\sum_{l = 1}^{B}
   d_{0,\bm{Y}_t} \theta_{j,t}^Y) .
   \end{align*}

   Next, we simplify the last two summands of the previous equation.
   For the first one, as
   $\sum_{i=1}^A \alpha_i = \sum_{j = 1}^B (\eqs_j - \beta_j)$ and
   for all $s$ it holds 
   $d_{0,\bm{X}_s} = d_{0,\bm{X}_1}$,
   we obtain 
   $$ \sum_{j = 1}^B \eqs_j \;
    \Big(\sum_{s = 1}^{A} d_{0,\bm{X}_s} \theta_{j,s}^X \Big)
    =
    \sum_{j = 1}^B \eqs_j \;
     \Big(
     \sum_{s = 1}^{A} d_{0,\bm{X}_s} \nsols  \frac{\alpha_s}{\sum_{i = 1}^A \alpha_i} \frac{\eqs_j - \beta_j}{\eqs_j}
     \Big)
     =
     \nsols \, d_{0,\bm{X}_1} \sum_{i=1}^A \alpha_i.$$

     For the second one, we perform the following direct calculations
     \begin{multline*}
     \sum_{j = 1}^B \eqs_j \;
     (\sum_{l = 1}^{B}
     d_{0,\bm{Y}_t} \theta_{j,t}^Y)
     =
     \sum_{j = 1}^B \eqs_j \;
     \Big(
     \sum_{t \in \{1,\dots,B\} \setminus \{j\}}
     d_{0,\bm{Y}_t}
     \nsols 
     \frac{\beta_t}{\eqs_t - \beta_t + 1} \frac{\eqs_j - \beta_j}{\eqs_j}
     +
     d_{0,\bm{Y}_j}
     \nsols   \frac{\beta_j}{\eqs_j}
   \Big)
   = \\
        \nsols \;  \sum_{j = 1}^B  
       \Big(
         \sum_{t = 1}^B
           d_{0,\bm{Y}_t}
           \cdot
            \frac{\beta_t}{\eqs_t - \beta_t + 1}
           (\eqs_j - \beta_j)
        -
           d_{0,\bm{Y}_j}
           \cdot
            \frac{\beta_j (\eqs_j - \beta_j)}{\eqs_j - \beta_j + 1}
     +
           d_{0,\bm{Y}_j} \beta_j
         \Big) = \\
                \nsols \;
       \sum_{j = 1}^B  (\eqs_j - \beta_j)
       \Big(
         \sum_{t = 1}^B
           d_{0,\bm{Y}_t}
           \cdot
            \frac{\beta_t}{\eqs_t - \beta_t + 1}
     \Big)
     +
     \nsols \;
     \sum_{j = 1}^B  
     d_{0,\bm{Y}_j} \cdot \frac{\beta_j}{\eqs_j - \beta_j + 1} = \\
     \nsols \;
            (1 + \sum_{i=1}^A \alpha_i )
       \Big(
         \sum_{t = 1}^B
           d_{0,\bm{Y}_t}
           \cdot
           \frac{\beta_t}{\eqs_t - \beta_t + 1}
           \Big).
         \end{multline*}
   At last we have the formula 
   \begin{align*}
     \deg(\Res_\P(\bm{d}_0,\dots,\bm{d}_n)) =
     \nsols
     \Big(1 + d_{0,\bm{X}_1} \sum_{i=1}^A \alpha_i
           + (1 + \sum_{i=1}^A \alpha_i )
          \Big( \sum_{t = 1}^B d_{0,\bm{Y}_t} \cdot
            \frac{\beta_t}{\eqs_t - \beta_t + 1}    \Big)
     \Big).
   \end{align*}
   The proof follows from instantiating the values of
   $d_{0,\bm{X}_1}, d_{0,\bm{Y}_{1}}, \dots, d_{0,\bm{Y}_{B}}$
   according to the multidegree of $f_0$.
 \end{proof}

\subsection{Proof of Theorem \ref{bipartite:thm:caseBipartite}}
\label{proof:bipartiteBilinear}

\begin{proof}[Proof of Theorem \ref{bipartite:thm:caseBipartite}]
  In this proof, we follow the same strategy as in the proof of
\Cref{multilinearMixed:thm:caseUnmixedMixed}.
As we did in \Cref{sec:det-formulas},
we can interpret the various multidegrees of $f_0$,
${\bm{d}_0} = (d_{0,{X}_1},\dots,d_{0,{X}_A}, d_{0,{Y_1}}, \dots,
d_{0,{Y_B}})$, that we want to prove that lead to determinantal
formulas as solutions of the following system of inequalities:
 \begin{align} \label{bipartite:eq:casesForD0}
   \begin{cases}
     (\forall 1 \leq i \leq A) \, & 0 \leq d_{0,\bm{X}_i} \leq 1 \\
     (\forall 1 \leq j \leq B) \, & 0 \leq d_{0,\bm{Y}_j} \leq 1 \\
     & \sum_{i = 1}^A d_{0,\bm{X}_i} \leq 1 \\
     & \sum_{j = 1}^B d_{0,\bm{Y}_j} \leq 1.
   \end{cases}
 \end{align}

 Consider the set $\{0, \dots, N\}$ that corresponds to generic
 polynomials $\bm{F}=(F_0, \dots, F_N)$
 of multidegrees ${\bm{d}_0},\dots,{\bm{d}_N}$
 (\Cref{intro:def:genMultiHom}), where $(F_1,\dots,F_n)$ is a square
 bipartite bilinear system such that, for each $i \in \{1,\dots,A\}$,
 $\sum_{j = 1}^B \eqs_{i,j} \geq \alpha_i$ and for each
 $j \in \{1,\dots,B\}$, $\sum_{i = 1}^A \eqs_{i,j} \geq \beta_j$, and
 ${\bm{d}_0}$ is the multidegree of $F_0$.
 As many of these polynomials have the
 same support,
 similarly to \eqref{multilinearMixed:eq:definitionSetOfIndices},
 we can gather them to simplify the cohomologies of
 \eqref{eq:Kvp}. For that we introduce the following notation.
 For
 each tuple $s_0,s_{1,1},\dots,s_{A,B} \in \N$, let
 $\mathcal{I}_{s_0,s_{1,1},\dots,s_{A,B}}$ be the set of all the
 subsets of $\{0,\dots,N\}$, such that
 \begin{itemize}
 \item For $1 \leq i \leq A$ and $1 \leq j \leq B$, the index
   $s_{i,j}$ indicates that we consider exactly $s_{i,j}$ polynomials
   from $(F_1, \dots, F_N)$ that belong to
   $\Z[\bm{u}][\bm{X}_i,\bm{Y}_j]_{\bm{1}}$.
   
 \item In addition, if $s_0 = 1$, then 0 belongs to all the sets in
 $\mathcal{I}_{s_0,s_{1,1},\dots,s_{A,B}}$
 \end{itemize}
 That is,
   \begin{align}
     \label{bipartite:eq:definitionSetOfIndices}
     \nonumber
     \mathcal{I}_{s_0,s_{1,1},\dots,s_{A,B}} := \Big\{ I : & I \subset \{0,\dots,n\}, \left( 0 \in I
                                                      \!\Leftrightarrow\! s_0 = 1 \right) \text{ and } 
     \\ &  (\forall 1 \leq i \leq A)(\forall 1 \leq  j \leq B) \; s_{i,j} = \#\{k \in I : f_k \in
          \K[\bm{X}_i,\bm{Y}_j]\} \Big\}.
   \end{align}

   As in \Cref{multilinearMixed:thm:rewriteKvp},
   we exploit the sets $\mathcal{I}_{s_0,s_{1,1},\dots,s_{A,B}}$ to
   rewrite the cohomologies $ K_v(\bm{m}) = \bigoplus_{p=0}^{N+1} K_{v,p} \otimes \Z[\bm{u}] $ of \eqref{eq:Kvp}
   in the following way,
    \begin{align}
      \label{bipartite:eq:eqWeyman}
      \nonumber
      K_{v,p}(\bm{m}) = 
      \hspace{-25px}
      \bigoplus_{\substack{0
      \leq s_0 \leq 1 \\
          (\forall 1 \leq i \leq A)(\forall 1 \leq  j \leq B) \; 0 \leq s_{i,j} \leq \eqs_{i,j}  \\
      s_0 + \sum_{i = 1}^A \sum_{j = 1}^B s_{i,j} = p}}
      \hspace{-25px}
        \Big(
        H^{p-v}_{\P}  \big( \bm{m} - \big( \sum_{j = 1}^B s_{1,j},\dots,\sum_{j = 1}^B s_{A,j},
                     \sum_{i = 1}^A s_{i,1},\dots,\sum_{i = 1}^A s_{i,B} \big)
      - s_0 \, {\bm{d}_0} \big)  
      \\[-20px]
        \otimes 
        \bigoplus_{I \in \mathcal{I}_{s_0,s_{1,1},\dots,s_{A,B}}}  \bigwedge_{k \in I} e_k
        \Big)
        .
      \end{align}

      Hence, using the K\"unneth Formula
      (\Cref{intro:thm:kunnethFormula}) with the degree vector
      $\bm{m}$ defined in \Cref{bipartite:thm:caseBipartite}, we have
      the following isomorphisms of cohomologies,
      \begin{multline}
     \label{bipartite:eq:rewriteHpv} 
     H^{p-v}_{\P} \left( \bm{m} - \left( \sum_{j = 1}^B
       s_{1,j},\dots,\sum_{j = 1}^B s_{A,j}, \sum_{i = 1}^A
       s_{i,1},\dots,\sum_{i = 1}^A s_{i,B} \right) - s_0 \, {\bm{d}_0}
   \right) \cong \\
       \bigoplus_{r_{\bm{X}_1} + \dots + r_{\bm{X}_A}  + r_{\bm{Y}_1} + \dots + r_{\bm{Y}_B} = p-v} \;
       \left(   \displaystyle
     \begin{array}{l c}   \displaystyle
                 \bigotimes_{i = 1}^A
       H^{r_{\bm{X}_i}}_{\Pr^{\alpha_i}}\left(\sum_{j = 1}^B (\eqs_{i,j} - s_{i,j}) - \alpha_i + (1-s_0) \, d_{0,X_i} \right) & \text{[Case X]} \\ \displaystyle
       \otimes
         \bigotimes_{j = 1}^B
                                                                                                                                                            H^{r_{\bm{Y}_j}}_{\Pr^{\beta_j}}\left(- 1 - \sum_{i = 1}^A s_{i,j}  - d_{0,Y_j} s_0  \right) & \text{[Case Y]}
     \end{array}
     \right)
   \end{multline}

   We will study the values for
   $p,v,s_0,s_{1,1},\dots,s_{A,B},r_{{\bm{X}_1}},\dots,r_{{\bm{X}_A}},r_{\bm{Y}_1},\dots,r_{\bm{Y}_B}$
   such that $K_{v,p}(\bm{m})$ does not vanish. Clearly, if
   $0 \leq s_0 \leq 1$ and
   $(\forall i \in \{1,\dots,A\}) \, (\forall j \in \{1,\dots,B\}) \;
   0 \leq s_{i,j} \leq \eqs_{i.j}$, then the module
   $\bigoplus_{I \in \mathcal{I}_{s_0,s_{1,1},\dots,s_{A,B}}}
   \bigwedge_{k \in I} e_k$ is not zero. Hence, assuming
   $0 \leq s_0 \leq 1$ and
   $(\forall i \in \{1,\dots,A\}) \, (\forall j \in \{1,\dots,B\}) \;
   0 \leq s_{i,j} \leq \eqs_{i.j}$
   $(\forall i \in \{1,\dots,B\}) \; 0 \leq s_i \leq \eqs_i$, we study
   the vanishing of the modules in the right-side part of
   \eqref{bipartite:eq:rewriteHpv}.
   We will study the cohomologies independently.
   By \Cref{intro:thm:bottFormula}, the modules in the right hand side
   of \eqref{bipartite:eq:rewriteHpv}
   are not zero only when, for $1 \leq i \leq A$,
   $r_{\bm{X}_i} \in \{0,\alpha_i\}$ and, for $1 \leq j \leq B$,
   $r_{\bm{Y}_j} \in \{0,\beta_j\}$.
   At the end of the proof we show that if
   \eqref{bipartite:eq:rewriteHpv} does not vanish then the following
   conditions hold,
      \begin{align} \label{bipartite:eq:valuesForR}
        \begin{array}{| c | c | c |}
                 \hline
          \text{[Case X]} &
          \text{For } 1 \leq i \leq A  &
              \begin{array}{c}
              r_{\bm{X}_{i}} = 0 \\ \displaystyle
              \sum_{j=1}^B (\eqs_{i,j} - s_{i,j}) - \alpha_i + (1  - s_0) \, d_{0,{\bm{X}_i}} \geq 0
              \end{array}
          \\ \hline
          \text{[Case Y]} &
          \text{For } 1 \leq j \leq B & 
              \begin{array}{c}
              r_{\bm{Y}_j} = \beta_j \\ \displaystyle
                                    \sum_{i=1}^A s_{i,j} + s_0 \, d_{0,{\bm{Y}_j}} - \beta_j \geq 0   
              \end{array} \\ \hline
            \end{array}
   \end{align}

   Using \eqref{bipartite:eq:valuesForR}, we study the possible values
   for $v$ such that $K_{v,p}(\bm{m})$ does not vanish. From
   \eqref{bipartite:eq:eqWeyman}, it holds
   $p = \sum_{i = 1}^A \sum_{j=1}^{B} s_{i,j} + s_0$.
     By \Cref{intro:thm:kunnethFormula},
     $p - v = \sum_{i = 1}^A r_{\bm{X}_i} + \sum_{j = 1}^B
     r_{{\bm{Y}_j}}$. Hence, we deduce that, when $K_{v,p}(\bm{m})$
     does not vanish, it holds,
   $$v =
   \sum_{i = 1}^A \sum_{j=1}^{B} s_{i,j} + s_0 -
    \sum_{j = 1}^B \beta_j
    =
    \sum_{j=1}^{B} \left(\sum_{i = 1}^A s_{i,j} - \beta_j \right) + s_0 
    .$$
    We bound the values for $v$ for which $K_{v,p}(\bm{m})$ does not
    vanish.
     \begin{itemize}
     \item First, we lower-bound $v$. 
       Assume that the cohomologies involving ${\bm{Y}_j}$ are not
       zero. Hence, if we sum over $j \in \{1,\dots,B\}$ the
       inequalities of [Case Y], we conclude that
       $$
       0 \leq \sum_{j=1}^{B} (\sum_{i=1}^A s_{i,j} - \beta_j) + s_0 \, \sum_{j=1}^{B} d_{0,{\bm{Y}_j}}  =
       v + s_0 \, \left(\sum_{j=1}^{B} d_{0,{\bm{Y}_j}} - 1 \right).
       $$
       By definition, \eqref{bipartite:eq:casesForD0},
       $0 \leq \sum_{j=1}^{B} d_{0,{\bm{Y}_j}} \leq 1$, and
       $0 \leq s_0 \leq 1$, hence
       $
       0 \leq v.
       $

   \item Finally, we upper-bound $v$.
     Assume that the cohomologies involving ${\bm{X}_j}$ are not
     zero. Hence, if we sum over $i \in \{1,\dots,A\}$ the
     inequalities of [Case X], we conclude that
     \begin{align*}
     0 \leq &\sum_{i = 1}^A \left( \sum_{j=1}^B (\eqs_{i,j} - s_{i,j}) - \alpha_i + (1  - s_0) \, d_{0,{\bm{X}_i}} \right) \\
     = &
     \sum_{i = 1}^A \sum_{j=1}^B \eqs_{i,j} - \sum_{i = 1}^A \alpha_i - \sum_{i = 1}^A \sum_{j=1}^B s_{i,j} +
     (1 - s_0) \, \left(\sum_{i = 1}^A d_{0,{\bm{X}_i}} \right).
     \end{align*}
     Recall that
     $N = \sum_{i = 1}^A \sum_{j=1}^B \eqs_{i,j} = \sum_{i = 1}^A
     \alpha_i + \sum_{i = j}^B \beta_j$ and
     $v = \sum_{i = 1}^A \sum_{j=1}^B s_{i,j} + s_0 - \sum_{i = j}^B
     \beta_j$. Also, as $\bm{d}_0$ is a solution of
     \eqref{bipartite:eq:casesForD0}, it holds
     $0 \leq \sum_{j=1}^{B} d_{0,{\bm{Y}_j}} \leq 1$, and
     $0 \leq s_0 \leq 1$. Hence
     $$
     v = \sum_{i = 1}^A \sum_{j=1}^B s_{i,j} + s_0 - \sum_{i = j}^B
     \beta_j \leq s_0 + (1 - s_0) \, \left(\sum_{i = 1}^A
       d_{0,{\bm{X}_i}} \right) \leq 1.
     $$
   \end{itemize}
  
   We conclude that the possible values for $v$, $p$,
   $r_{\bm{X}_1},\dots,r_{\bm{X}_A}$,
   $r_{\bm{Y}_1},\dots,r_{\bm{Y}_B}$ such that
   \eqref{bipartite:eq:rewriteHpv} is not zero are $v \in \{0, 1\}$,
   the possible values for $r_{\bm{X}_1},\dots,r_{\bm{X}_A}$,
   $r_{\bm{Y}_1},\dots,r_{\bm{Y}_B}$ are the ones in
   \eqref{bipartite:eq:valuesForR} and
   $p = \sum_{j = 1}^B \beta_j + v$. Hence, our Weyman complex looks
   like \eqref{eq:determinantalMap}, where
   $$\delta_1(\bm{m}) : K_{1,\sum_{j = 1}^B \beta_j+1}(\bm{m})
   \to K_{0,\sum_{j = 1}^B \beta_j}(\bm{m})$$
   is a Koszul-type determinantal formula.

   In what follows we prove each case in \eqref{bipartite:eq:valuesForR}.
   
   \textbf{Case (X)}
   We consider the modules that involve the variables in the
     block ${\bm{X}_i}$, for $1 \leq i \leq A$.
     As $(\forall j) \, s_{i,j} \leq \eqs_{i,j}$, $0 \leq s_0 \leq 1$ and
     $0 \leq d_{0,{\bm{X}_i}} \leq 1$, we have
     $
     \sum_{j=1}^B (\eqs_{i,j} - s_{i,j}) - \alpha_i + (1  - s_0) \, d_{0,{\bm{X}_i}} > -1 - \alpha_i.
     $
     Hence, by \Cref{intro:thm:possiblesRi},
        \begin{align*}
          \begin{gathered}
            H^{r_{\bm{X}_i}}_{\Pr^{\alpha_i}}(\sum_{j=1}^B (\eqs_{i,j} - s_{i,j}) - \alpha_i + (1  - s_0) \, d_{0,{\bm{X}_i}})
            \neq 0 \iff \\
            r_{\bm{X}_i} = 0 \quad \text{ and } \quad
            \sum_{j=1}^B (\eqs_{i,j} - s_{i,j}) - \alpha_i + (1  - s_0) \, d_{0,{\bm{X}_i}} \geq 0.
          \end{gathered}
        \end{align*}

     \textbf{Case (Y)}
      We consider the modules that involve the variables in the
     block ${\bm{Y}_j}$, for $1 \leq j \leq B$.
     As $(\forall j \in \{1,\dots,B\}) \, s_{i,j} \geq 0$ and $s_0, d_{0,{\bm{Y}_j}} \geq 0$, then
     $-1 - \sum_{i=1}^A s_{i,j} - s_0 \, d_{0,{\bm{Y}_j}} < 0$, and so by \Cref{intro:thm:possiblesRi},
          \begin{align*}
            \begin{gathered}
       H^{r_{\bm{Y}_j}}_{\Pr^{\beta_j}}(-1 - \sum_{i=1}^A s_{i,j} - s_0 \, d_{0,{\bm{Y}_j}}) \neq 0 \iff \\ 
       r_{\bm{Y}_j} = \beta_j \quad \text{ and } \quad \sum_{i=1}^A s_{i,j} + s_0 \, d_{0,{\bm{Y}_j}} - \beta_j \geq 0.
     \end{gathered}
     \end{align*}

 \end{proof}

\subsection{Example of determinantal formula for bipartite bilinear system}
\label{bipartite:example}
Consider four blocks of variables such that $A = 2$, $B = 2$,
$\alpha = (1,2)$, $\beta = (1,2)$, and
$$
\left\{
\begin{array}{r l}
  X_1 := & \{X_{1,0}, X_{1,1}\} \\
  X_2 := & \{X_{2,0}, X_{2,1}, X_{2,2}\} \\
  Y_1 := & \{Y_{1,0}, Y_{1,1}\} \\
  Y_2 := & \{Y_{2,0}, Y_{2,1}, Y_{2,2}\}.
\end{array}
\right.
$$
Let $(f_1,\dots,f_6)$ be the square bipartite bilinear system
represented by the following graph:
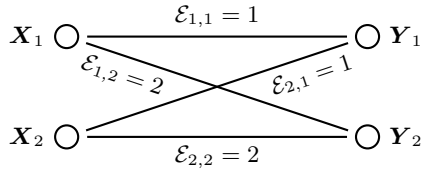
\begin{center}
\begin{tikzpicture}[thick,
  fsnode/.style={draw,circle},
  ssnode/.style={draw,circle},
  every fit/.style={ellipse,draw,inner sep=-2pt,text width=2cm},
  -,shorten >= 3pt,shorten <= 3pt
]

\begin{scope}[start chain=going below,node distance=10mm]
\foreach \i in {1,2}
  \node[fsnode,on chain] (x\i) [label=left: ${\bm{X}_{\i}}$] {};
\end{scope}

\begin{scope}[xshift=4cm,yshift=-0.0cm,start chain=going below,node distance=10mm]
\foreach \j in {1,2}
  \node[ssnode,on chain] (y\j) [label=right: ${\bm{Y}_{\j}}$] {};
\end{scope}


\draw (x1) -- (y1) node [midway, above] {$\eqs_{1,1} = 1$};
\draw (x1) -- (y2) node [pos=0.18, below,sloped] {$\eqs_{1,2} = 2$};
\draw (x2) -- (y1) node [pos=0.82, below,sloped] {$\eqs_{2,1} = 1$};
\draw (x2) -- (y2) node [midway, below] {$\eqs_{2,2} = 2$};
\end{tikzpicture}
\end{center}
We introduce a polynomial $f_0 \in \K[\bm{X}_1,\bm{Y}_1]_{\bm{1}}$ and
consider the following overdetermined system $\bm{f}$ where,
\begin{align}
  \bm{f} := 
    \left\{
    \begin{array}{l r}
      f_0 := &
      \left( a_{{1}}\,y_{1;{0}}+a_{{2}}\,y_{1;{1}} \right) x_{1;{0}}+ \left( a_{{3}}\,y_{1;{0}}+a_{{4}}\,y_{1;{1}} \right) x_{1;{1}}\\
      f_1 := &
            \left( b_{{1}}\,y_{1;{0}}+b_{{2}}\,y_{1;{1}} \right) x_{1;{0}}+ \left( b_{{3}}\,y_{1;{0}}+b_{{4}}\,y_{1;{1}} \right) x_{1;{1}}\\
      f_2 := &
 \left( c_{{1}}\,y_{1;{0}}+c_{{2}}\,y_{1;{1}} \right) x_{2;{0}}+ \left( c_{{5}}\,y_{1;{0}}+c_{{6}}\,y_{1;{1}} \right) x_{2;{1}}+ \left( c_{{3}}\,y_{1;{0}}+c_{{4}}\,y_{1;{1}} \right) x_{2;{2}}\\
      f_3 := &
 \left( d_{{1}}\,y_{2;{0}}+d_{{2}}\,y_{2;{2}}+d_{{3}}\,y_{2;{1}} \right) x_{1;{0}}+ \left( d_{{4}}\,y_{2;{0}}+d_{{5}}\,y_{2;{2}}+d_{{6}}\,y_{2;{1}} \right) x_{1;{1}}\\
      f_4 := &
 \left( e_{{1}}\,y_{2;{0}}+e_{{2}}\,y_{2;{2}}+e_{{3}}\,y_{2;{1}} \right) x_{1;{0}}+ \left( e_{{4}}\,y_{2;{0}}+e_{{5}}\,y_{2;{2}}+e_{{6}}\,y_{2;{1}} \right) x_{1;{1}}\\
      f_5 := &
 \left( g_{{1}}\,y_{2;{0}}+g_{{2}}\,y_{2;{2}}+g_{{3}}\,y_{2;{1}} \right) x_{2;{0}}+ \left( g_{{7}}\,y_{2;{0}}+g_{{8}}\,y_{2;{2}}+g_{{9}}\,y_{2;{1}} \right) x_{2;{1}} \\ & \hfill + \left( g_{{4}}\,y_{2;{0}}+g_{{5}}\,y_{2;{2}}+g_{{6}}\,y_{2;{1}} \right) x_{2;{2}}\\
      f_6 := &
 \left( h_{{1}}\,y_{2;{0}}+h_{{2}}\,y_{2;{2}}+h_{{3}}\,y_{2;{1}} \right) x_{2;{0}}+ \left( h_{{7}}\,y_{2;{0}}+h_{{8}}\,y_{2;{2}}+h_{{9}}\,y_{2;{1}} \right) x_{2;{1}}  \\ & \hfill +  \left( h_{{4}}\,y_{2;{0}}+h_{{5}}\,y_{2;{2}}+h_{{6}}\,y_{2;{1}} \right) x_{2;{2}}
    \end{array}
    \right.
  \end{align}
Following \Cref{bipartite:thm:caseBipartite}, we consider the
degree vector $\bm{m} = (3, 1, -1, -1)$. The vector spaces of the
Weyman complex $K(\bm{m}, \bm{f})$ looks like,
{\small
  \begin{align*}
  K_1(\bm{m}, \bm{f}) =  &  \quad\,\,
                          S_{\bm{X}_1}(0) \otimes S_{\bm{X}_2}(0) \otimes S^*_{\bm{Y}_1}(0) \otimes S^*_{\bm{Y}_2}(-1) \otimes \left\{
                          \begin{array}{l}
                            (e_{0} \wedge e_{3} \wedge e_{4} \wedge e_{5}) \oplus (e_{0} \wedge e_{3} \wedge e_{4} \wedge e_{6}) \; \oplus \\
                            (e_{2} \wedge e_{3} \wedge e_{4} \wedge e_{5}) \oplus (e_{2} \wedge e_{3} \wedge e_{4} \wedge e_{6})
                          \end{array}
  \right\} \\
                        & \oplus S_{\bm{X}_1}(0) \otimes S_{\bm{X}_2}(0) \otimes S^*_{\bm{Y}_1}(-1) \otimes S^*_{\bm{Y}_2}(0) \otimes \left\{
                          \begin{array}{l}
                            (e_{0} \wedge e_{2} \wedge e_{3} \wedge e_{6}) \oplus (e_{0} \wedge e_{2} \wedge e_{4} \wedge e_{5}) \; \oplus \\
                            (e_{0} \wedge e_{2} \wedge e_{4} \wedge e_{6}) \oplus (e_{0} \wedge e_{2} \wedge e_{3} \wedge e_{4}) \; \oplus \\
                            (e_{2} \wedge e_{2} \wedge e_{3} \wedge e_{4}) \oplus (e_{0} \wedge e_{2} \wedge e_{3} \wedge e_{5})
                          \end{array}
  \right\} \\
  K_0(\bm{m}, \bm{f}) = & \quad\,\,
                          S_{\bm{X}_1}(0) \otimes S_{\bm{X}_2}(1) \otimes S^*_{\bm{Y}_1}(0) \otimes S^*_{\bm{Y}_2}(0) \otimes
                          \left\{
                          \begin{array}{l}
                          (e_{0} \wedge e_{3} \wedge e_{4}) \oplus (e_{2} \wedge e_{3} \wedge e_{4})
                          \end{array}
                          \right\}
                                                    \\ &
                                                         \oplus S_{\bm{X}_1}(1) \otimes S_{\bm{X}_2}(0) \otimes S^*_{\bm{Y}_1}(0) \otimes S^*_{\bm{Y}_2}(0) \otimes
                                                         \left\{
                          \begin{array}{l}
                            (e_{0} \wedge e_{3} \wedge e_{5}) \oplus (e_{0} \wedge e_{3} \wedge e_{6}) \; \oplus \\
                            (e_{0} \wedge e_{4} \wedge e_{5}) \oplus (e_{0} \wedge e_{4} \wedge e_{6}) \; \oplus \\
                            (e_{1} \wedge e_{3} \wedge e_{5}) \oplus (e_{1} \wedge e_{3} \wedge e_{6}) \; \oplus \\
                            (e_{1} \wedge e_{4} \wedge e_{5}) \oplus (e_{1} \wedge e_{4} \wedge e_{6}) \; \oplus \\
                            (e_{2} \wedge e_{3} \wedge e_{4})
                          \end{array}
  \right\}
  \end{align*}}

\noindent
The Koszul determinantal matrix representing the map
$\delta_1(\bm{m}, \bm{f})$ between the modules \wrt a monomial basis
is,
    \begin{align*}
      \arraycolsep=2pt
\text{  {\scriptsize $
 \left[ \begin {array}{cccccccccccccccccccccccc} &&&-b_{{1}}&-b_{{3}}&&&&&&&&&&a_{{1}}&a_{{3}}&&&&&&&&\\
 &&&-b_{{2}}&-b_{{4}}&&&&&&&&&&a_{{2}}&a_{{4}}&&&&&&&&\\
                                                 &&&&&-b_{{1}}&-b_{{3}}&&&&&&&&&&a_{{1}}&a_{{3}}&&&&&&\\
 &&&&&-b_{{2}}&-b_{{4}}&&&&&&&&&&a_{{2}}&a_{{4}}&&&&&&\\
 &&&&&&&-b_{{1}}&-b_{{3}}&&&&&&&&&&a_{{1}}&a_{{3}}&&&&\\
 &&&&&&&-b_{{2}}&-b_{{4}}&&&&&&&&&&a_{{2}}&a_{{4}}&&&&\\
 &&&&&&&&&-b_{{1}}&-b_{{3}}&&&&&&&&&&a_{{1}}&a_{{3}}&&\\
 &&&&&&&&&-b_{{2}}&-b_{{4}}&&&&&&&&&&a_{{2}}&a_{{4}}&&\\
 -c_{{1}}&-c_{{5}}&-c_{{3}}&&&&&&&&&&&&&&&&&&&&a_{{1}}&a_{{3}}\\
 -c_{{2}}&-c_{{6}}&-c_{{4}}&&&&&&&&&&&&&&&&&&&&a_{{2}}&a_{{4}}\\
 -g_{{1}}&-g_{{7}}&-g_{{4}}&e_{{1}}&e_{{4}}&&&-d_{{1}}&-d_{{4}}&&&&&&&&&&&&&&&\\
 -g_{{3}}&-g_{{9}}&-g_{{6}}&e_{{3}}&e_{{6}}&&&-d_{{3}}&-d_{{6}}&&&&&&&&&&&&&&&\\
 -g_{{2}}&-g_{{8}}&-g_{{5}}&e_{{2}}&e_{{5}}&&&-d_{{2}}&-d_{{5}}&&&&&&&&&&&&&&&\\
 -h_{{1}}&-h_{{7}}&-h_{{4}}&&&e_{{1}}&e_{{4}}&&&-d_{{1}}&-d_{{4}}&&&&&&&&&&&&&\\
 -h_{{3}}&-h_{{9}}&-h_{{6}}&&&e_{{3}}&e_{{6}}&&&-d_{{3}}&-d_{{6}}&&&&&&&&&&&&&\\
 -h_{{2}}&-h_{{8}}&-h_{{5}}&&&e_{{2}}&e_{{5}}&&&-d_{{2}}&-d_{{5}}&&&&&&&&&&&&&\\
 &&&&&&&&&&&-c_{{1}}&-c_{{5}}&-c_{{3}}&&&&&&&&&b_{{1}}&b_{{3}}\\
 &&&&&&&&&&&-c_{{2}}&-c_{{6}}&-c_{{4}}&&&&&&&&&b_{{2}}&b_{{4}}\\
 &&&&&&&&&&&-g_{{1}}&-g_{{7}}&-g_{{4}}&e_{{1}}&e_{{4}}&&&-d_{{1}}&-d_{{4}}&&&&\\
 &&&&&&&&&&&-g_{{3}}&-g_{{9}}&-g_{{6}}&e_{{3}}&e_{{6}}&&&-d_{{3}}&-d_{{6}}&&&&\\
 &&&&&&&&&&&-g_{{2}}&-g_{{8}}&-g_{{5}}&e_{{2}}&e_{{5}}&&&-d_{{2}}&-d_{{5}}&&&&\\
 &&&&&&&&&&&-h_{{1}}&-h_{{7}}&-h_{{4}}&&&e_{{1}}&e_{{4}}&&&-d_{{1}}&-d_{{4}}&&\\
 &&&&&&&&&&&-h_{{3}}&-h_{{9}}&-h_{{6}}&&&e_{{3}}&e_{{6}}&&&-d_{{3}}&-d_{{6}}&&\\
 &&&&&&&&&&&-h_{{2}}&-h_{{8}}&-h_{{5}}&&&e_{{2}}&e_{{5}}&&&-d_{{2}}&-d_{{5}}&&\end {array} \right] 
      $}}  \end{align*}


\end{document}